\documentclass[a4paper,11pt]{article}

\usepackage[style=authoryear, backend=bibtex]{biblatex}
\usepackage{fullpage}

\usepackage{tikz}
\usetikzlibrary{automata,positioning}
\usetikzlibrary{arrows.meta}
\tikzset{
	invisible/.style={opacity=0},
	visible on/.style={alt={#1{}{invisible}}},
	alt/.code args={<#1>#2#3}{%
		\alt<#1>{\pgfkeysalso{#2}}{\pgfkeysalso{#3}} 
	},
}

\usepackage{subcaption}
\usepackage{comment}

\usepackage{graphicx}
\usepackage{lscape}

\usepackage{enumitem}
\usepackage{xcolor} 
\usepackage{hyperref}

\usepackage{amsthm}
\usepackage{amsmath}
\usepackage{amsfonts}
\usepackage{amssymb}
\usepackage{euscript}

\def\mydefb#1{\expandafter\def\csname b#1\endcsname{\mathbf{#1}}}
\def\mydefallb#1{\ifx#1\mydefallb\else\mydefb#1\expandafter\mydefallb\fi}
\mydefallb ABCDEFGHIJKLMNOPQRSTUVWXYZ\mydefallb

\def\mydefb#1{\expandafter\def\csname bb#1\endcsname{\mathbb{#1}}}
\def\mydefallb#1{\ifx#1\mydefallb\else\mydefb#1\expandafter\mydefallb\fi}
\mydefallb ABCDEFGHIJKLMNOPQRSTUVWXYZ\mydefallb

\def\mydefb#1{\expandafter\def\csname c#1\endcsname{\mathcal{#1}}}
\def\mydefallb#1{\ifx#1\mydefallb\else\mydefb#1\expandafter\mydefallb\fi}
\mydefallb ABCDEFGHIJKLMNOPQRSTUVWXYZ\mydefallb

\def\mydefb#1{\expandafter\def\csname hb#1\endcsname{\hat{\mathbf{#1}}}}
\def\mydefallb#1{\ifx#1\mydefallb\else\mydefb#1\expandafter\mydefallb\fi}
\mydefallb ABCDEFGHIJKLMNOPQRSTUVWXYZ\mydefallb

\newcommand{\sD}{\EuScript{D}}

\newcommand{\br}{\mathbf{r}}
\newcommand{\bzero}{\mathbf{0}}

\DeclareMathOperator*{\argmin}{arg\,min}
\DeclareMathOperator{\tr}{tr}
\DeclareMathOperator{\dg}{dg}
\DeclareMathOperator{\supp}{supp}
\DeclareMathOperator{\Cov}{Cov}
\newcommand{\CI}{\mathrel{\text{\scalebox{1.07}{$\perp\mkern-10mu\perp$}}}}

\theoremstyle{plain}
\newtheorem{definition}{Definition}[section]
\newtheorem{lemma}[definition]{Lemma}
\newtheorem{theorem}[definition]{Theorem}
\newtheorem{corollary}[definition]{Corollary}
\newtheorem{remark}[definition]{Remark}
\newtheorem{example}[definition]{Example}
\newtheorem{assumption}{Assumption}

\usepackage{multirow}
\usepackage{multicol}
\newcommand\Tstrut{\rule{0pt}{2.6ex}}

\usepackage{caption} 
\usepackage[nottoc,notlot,notlof]{tocbibind}

\usepackage{authblk}

\begin{document}
\date{}

\title{\bfseries Continuously Indexed Graphical Models\footnote{Research supported by a Swiss National Science Foundation Research Grant.}}
\author[1]{Kartik G. Waghmare}
\author[2]{Victor M. Panaretos}
\affil[1,2]{Institute of Mathematics, \'{E}cole polytechnique f\'{e}d\'{e}rale de Lausanne, Switzerland}
\affil[1]{\textit{kartik.waghmare@epfl.ch}}
\affil[2]{\textit{victor.panaretos@epfl.ch}}

\maketitle

\begin{abstract}
    Let $X = \{X_{u}\}_{u \in U}$ be a real-valued Gaussian process indexed by a set $U$. It can be thought of as an undirected graphical model with every random variable $X_{u}$ serving as a vertex. We characterize this graph in terms of the covariance of $X$ through its reproducing kernel property. Unlike other characterizations in the literature, our characterization does not restrict the index set $U$ to be finite or countable, and hence can be used to model the intrinsic dependence structure of stochastic processes in continuous time/space. Consequently, this characterization is not in terms of the zero entries of an inverse covariance. This poses novel challenges for the problem of recovery of the dependence structure from a sample of independent realizations of $X$, also known as structure estimation. We propose a methodology that circumvents these issues, by targeting the recovery of the underlying graph up to a finite resolution, which can be arbitrarily fine and is limited only by the available sample size. The recovery is shown to be consistent so long as the graph is sufficiently regular in an appropriate sense. We derive corresponding convergence rates and finite sample guarantees. Our methodology is illustrated by means of a  simulation study and two data analyses.
\end{abstract}

\textbf{AMS Subject Classification:} 62H22, 62R10 (Primary), 62M05, 62M15 (Secondary)

\tableofcontents

\section{Introduction}

We consider the problem of defining undirected graphical models with \emph{uncountable} vertex sets. The purpose of such models is to describe conditional independence relationships inherent in stochastic processes over continuous time/space -- just as ordinary (finite) undirected graphical models do for random vectors in Euclidean spaces. Furthermore, we consider the statistical problem of \emph{recovering} the graph from a finite number of independent realizations of the process, possibly observed discretely and with measurement error,  up to a degree of resolution commensurate with the amount of data available.

Let $X = \{X_{u}\}_{u \in U}$ be a zero-mean Gaussian process on a (possibly uncountably infinite) set $U$. We would like to think of $X$ as a Gaussian graphical model where every random variable $X_{u}$ corresponds to a vertex in a graph $\Omega$ on the index set $U$. The conditional independence structure of $X$ should likewise correspond to the edge structure of $\Omega$, in that, for $u, v \in U$ separated by $W \subset U$ in $\Omega$ we should have
\begin{equation*}
    X_{u} \CI X_{v} ~|~ X_{W},
\end{equation*} 
where $X_{W} = \{X_{w}: w \in W\}$. To this end, we will  characterize the covariance of processes admitting a given graphical structure through the reproducing property of their covariance kernel. And, going in the other direction, we will use this characterization to define the \emph{graph of a process} in terms of its covariance. 

Although the stated characterization is always valid, it is somewhat unwieldy for the purpose of parsing the graph of a given process from its covariance. In the finite vertex set case, we have a particularly handy result when the covariance matrix is invertible, sometimes called the \emph{inverse zero characterization}. This states that the $ij$-th entry of the of the precision (inverse covariance) matrix is zero if and only if there is no edge between the $i$th and $j$th vertices. For uncountably infinite $U$, a direct analogous characterization for covariance kernels is unavailable to us, if indeed it exists at all. In order to derive an analogous result, we develop a notion of \emph{resolution} of a graph. This allows for an alternative inverse zero characterization, yielding a \emph{pixelated} version of the graph of a continuously-indexed process, from the zero entries of a certain \emph{correlation operator matrix} related to its covariance. The choice of resolution can be arbitrary large, and under appropriate conditions yields an exact characterization of the continuum graph in the limit.

This framework subsequently allows us to meaningfully pose the problem of recovering the graph from $n$ independent realizations of the process, with the resolution dictated by the available sample size. Because arbitrarily small changes in the covariance kernel can greatly alter the graph of the associated process, targeting the graph at a sample-dependent finite resolution can also be seen as quantifying how finely the graphical structure can be resolved with a given amount of finite information. In this context, we propose a graph estimator that relies on thresholding (in the operator norm) the entries of the inverse empirical correlation operator matrix. Under standard regularity assumptions on the correlation operator matrix, we show that the underlying graph can be recovered with high probability as the number of samples increases. Also, we give a lower bound on the sample size required to recover the graph at a given familywise error rate.

Although we restrict focus to Gaussian processes, our analysis can be easily extended to sub-Gaussian processes by interpreting the graph in terms of ``conditional uncorrelatedness" instead of conditional independence (\textcite{waghmare2021}). The resulting structure corresponds to a correlational graphoid (\textcite{pearl1985}) and a basic strong separoid (\textcite{dawid2001}), and therefore serves as a reasonable alternative to conditional independence in the absence of Gaussianity. 

The main contributions of this article are the proper definition of the graph of a Gaussian process, its inverse--zero characterization at finite resolution, and the idea of regularization by pixelation. Furthermore, we develop estimation procedures, and associated  performance guarantees, for the problem of graph recovery based on observing sample paths of the process (where observation could be complete, or discrete/noisy, including regular and sparse designs).


\subsection{Background and Related Work}

Undirected graphical models allow us to distinguish direct and indirect associations in data, and thus have a long history in statistics. They have been investigated as models (\textcite{dempster1972}, \textcite{darroch1980}), and as targets of inference (\textcite{lauritzen1996}), with a particular emphasis on high-dimensional settings more recently (\textcite{meinshausen2006}, \textcite{ravikumar2011} and \textcite{rothman2008}). Graphical models with countably infinite index sets have been investigated by \textcite{montague2018} from an axiomatic and probabilistic point of view.

Graphical models with an uncountably infinite number of vertices have not received much attention in the literature, but they are implicit in the study of Markov processes, which can be regarded as infinite graphical models with infinitesimally small graphs. The generalization of the Markov property to Euclidean spaces by \textcite{mckean1963} using the concept of splitting fields, and to locally compact metric spaces by \textcite{rozanov1982}, along with the generalization of the Markov property itself to the quasi-Markov property by \textcite{chay1972}, can be thought of as important steps in this direction.

Graphical models are frequently used to model continuous time (or space) stochastic processes under the label of ``Gaussian Markov random fields (GMRFs)" (\textcite{rue2005}). Although this is usually done for computational benefits, there are important cases in which there is an explicit link between the underlying process and the GMRF used to model it (\textcite{lindgren2011}). Roughly speaking, this amounts to modelling the graphical structure of the process itself.

Elsewhere, directed graphical models have been implicit in auto-regressive, probabilistic and state space approaches for time series analysis, where random variables indexed by discrete time act as vertices (see \textcite{eichler2000}, \textcite{murphy2002}, \textcite{barber2010}). In contrast, we consider graphs of variables indexed by continuous time, which can be viewed as a time series with a continuous time index. 

In the context of functional data, a graphical model can refer to several distinct possibilities. To explain the nuances involved we introduce some notation. Consider an $\bbR^{p}$-valued stochastic process $X$ on an interval $I \subset \bbR$ given by
\begin{equation*}\label{eqn:related}
  t \mapsto  \begin{bmatrix}
        X_{1}(t) \\
        X_{2}(t) \\
        \vdots \\
        X_{p}(t)
 \end{bmatrix}  \in\mathbb{R}^p .
\end{equation*}
Viewing this as a vector-valued function, \textcite{qiao2020} deals with recovering the graphical structure between $\{X_{j}(t): 1 \leq j \leq p\}$ as a function of $t$. This can be thought of as a pointwise finite graphical model: for every $t$, one has a graphical model on $p$ vertices. This perspective is related to \textcite{mogensen2022}, who consider finitely-indexed graphical models on diffusions in $\mathbb{R}^p$. On the other hand, viewing each function globally $t \mapsto X_{j}(t)$ as a random element in a Hilbert space $\mathbb{H}$, one has a single $p$-vector with Hilbertian entries,
 $$\begin{bmatrix}
        X_{1} \\
        X_{2} \\
        \vdots \\
        X_{p}
 \end{bmatrix}  \in\mathbb{H}^p.$$
 In this context, \textcite{qiao2019}, \textcite{lisolea2018} and \textcite{leeli2021} address the problem of recovering the graphical structure between the $p$ vector coordinates $X_{j}$ for $1 \leq j \leq p$. Thus they address the problem of recovering the structure \emph{between} a finite number of related random functions. This can be seen as a global, rather than pointwise approach.
 
In either case, the problem can be seen as recovering the dependence structure \emph{between} a finite collection of $p$ random functions. In contrast, we wish to study the structure \emph{within} a single random function. That is, our interest lies in the graphical structure of the collection $\{X_{j}(t):t \in I\}$ for a given fixed $j$. Thus, we are interested in an \emph{intrinsic} graphical model. Importantly, this means that we are concerned with the problem of recovering the dependence structure between an \emph{uncountably infinite} number of jointly distributed random variables, unlike the above mentioned literature, which deals with a \emph{finite} number of real random variables or Hilbertian random elements. Indeed, we will see that our setting subsumes existing notions of functional graphical models as special cases (Section \ref{subsec:func_graph_models}).

\subsection{Outline of the article}
In Section \ref{sec:bg_notn}, we introduce some notation and review certain basic concepts concerning the theory of graphs,  linear operators, Gaussian processes, and reproducing kernels.  In Section \ref{sec:graph_repn}, we present our characterization of the conditional independence structure of a Gaussian process in terms of its covariance function. Furthermore, we make concrete the notion of \emph{the graph of a process} and derive the graphs of some familiar classes of Gaussian processes explicitly. In Section \ref{sec:partitioning}, we explain in greater detail the concept of resolution. Moreover, we derive an analogue of the finite-dimensional inverse zero characterization (\ref{eqn:inv_zero}), and use it to develop sufficient criteria for the approximate and exact identifiability of the graph of a process. Additionally, we comparatively discuss parallels and differences in our setting/approach and those of relating to multivariate functional graphical models. We describe our graph recovery algorithm in Section \ref{sec:methodology}. In Section \ref{sec:estimation}, we develop asymptotic rates of convergence as well as finite sample bounds. Section \ref{sec:implementation} discusses implementation details, including the choice of tuning parameters. In Section \ref{sec:numerical_simulations}, we present a simulation study covering a variety of covariances at different resolutions and samples sizes. Finally, in Section \ref{sec:data_analysis}, we illustrate our method by applying it to spectroscopy and intraday stock price data.

\section{Preliminaries and Notation}\label{sec:bg_notn}

\subsection{Graphs and Graphical Models}

An \emph{undirected graph} on a set $U$ is defined as a pair $(U, \Omega)$, where $\Omega \subset U \times U$, and such that for any $(u, v) \in U\times U$ we have  $(u, u) \in \Omega$ and $(u, v) \in \Omega \iff (v, u) \in \Omega$. The set $U$ is called the \emph{vertex set} and the set $\Omega$ is called the \emph{edge set}. All graphs in this article are undirected. Since the vertex set will always be fixed, we shall refer to a graph by its edge set $\Omega$.  We shall say $u,v \in U$ are \emph{adjacent} if $(u, v) \in \Omega$, that is, if they have an edge between them. By convention, we assume that every vertex shares an edge with itself. To visualize the graph $\Omega$, we define the \emph{adjacency function} $\mathbf{1}_{\Omega}: U \times U \to \bbR$ as
\begin{equation*}
    \mathbf{1}_{\Omega}(u, v) = \begin{cases}
        1 & (u, v) \in \Omega \\
        0 & (u, v) \notin \Omega
    \end{cases}.
\end{equation*}
This describes the structure of the graph in a way analogous to how the adjacency matrix does so when $U$ is restricted to be finite. A graph is called \emph{complete} if all vertices are adjacent to each other. The unique complete graph on $U$ is given by $U \times U$. For $u,v \in U$, a \emph{path} on $\Omega$ from $u$ to $v$ is a finite sequence $\{w_{k}\}_{k=0}^{n+1}$ of vertices such that $w_{0} = u$, $(w_{k}, w_{k+1})\in \Omega$ (they are adjacent) for $0 \leq k \leq n$, and $w_{n+1} = v$. The vertices $u$ and $v$ are called \emph{connected} if there is a path between them and \emph{disconnected} otherwise. A subset $W$ of $U$ is said to separate $u, v \in U$ if every path between $u$ and $v$ passes through $W$. If $u$ and $v$ are disconnected, then they can be said to be separated by the empty set $\varnothing$.

A graphical model $(X, \Omega)$ consists of a set of random variables $X = \{X_{u}: u \in U\}$ indexed by a set $U$, and a graph $\Omega\subset U\times U$, such that for every $u, v \in U$ separated by $W \subset U$ in $\Omega$, the process $X$ satisfies
\begin{equation}\label{eqn:sepn_condn}
    X_{u} \CI X_{v} ~|~ X_{W}.
\end{equation}
Here, $X_W:=\{X_w: w\in W\}$ represents the restriction of $X$ to $W\subset U$. In other words, $(X, \Omega)$ is a graphical model if $X$ satisfies the \emph{global Markov property} (\ref{eqn:sepn_condn}) with respect to $\Omega$. It is implicit in the definition that if $u$ and $v$ are disconnected, then $X_{u}$ and $X_{v}$ are independent. The global Markov property relates $X$ and $\Omega$ by making the conditional independence structure of $X$ conform with the edge structure of the graph $\Omega$. Note that, for notational convenience, we have defined our graphical models slightly differently than the standard nomenclature: the vertex set of our graph is the domain $U$ instead the set of random variables $\{X_{u} : u \in U\}$. 

\subsection{Linear Operators on Hilbert Spaces}

To develop our theory and methods, we will work with linear operators acting on Hilbert spaces. Recall that a Hilbert space is a complete inner product space. Let $U$ be a compact subset of the Euclidean space $\bbR^{d}$ equipped with a finite Borel measure $\mu$ -- for example, $U = [0, 1] \subset \bbR$  equipped with the Lebesgue measure. In addition to the reproducing kernel Hilbert spaces, which are introduced later, the Hilbert space we are mostly concerned with in this article is the space of square-integrable functions, $L^{2}(U, \mu) = \{f: U \to \bbR \mbox{ such that } \int_{U}|f(u)|^{2} ~d\mu(u) < \infty\}$ equipped with the inner product and norm given by
\begin{equation*}
    \langle f, g \rangle_{L^{2}(U, \mu)} = \int_{U} f(u) g(u) ~d\mu(u) \quad\mbox{ and }\quad \|f\|_{L^{2}(U, \mu)}^{2} =  \int_{U} |f(u)|^{2} ~d\mu(u),
\end{equation*}
respectively.

Linear transformations between Hilbert spaces are referred to as \emph{linear operators}, and their behavior is similar to that of linear transformations between Euclidean spaces, although there are important differences. We shall use boldface capital letters like $\bA$ to denote them. Let $\cH_{1}$ and $\cH_{2}$ be Hilbert spaces and $\bA: \cH_{1} \to \cH_{2}$ be a linear operator. The operator norm $\|\bA\|$ of an operator $\bA$ is then defined as 
\begin{equation*}
    \|\bA\| = \sup \{ \|\bA f\|_{\cH_{2}} : \|f\|_{\cH_{1}} \leq 1\},
\end{equation*}
which is similar to how one defines the spectral norm for matrices. We say that $\bA$ is \emph{bounded} if $\|\bA\| < \infty$. The \emph{adjoint} of a bounded linear operator $\bA$ is the unique linear operator $\bA^{\ast}: \cH_{2} \to \cH_{1}$ satisfying $\langle \bA f, g \rangle_{\cH_{2}} = \langle f, \bA^{\ast} g \rangle_{\cH_{1}}$ for every $f \in \cH_{1}$ and $g \in \cH_{2}$.

If $\cH_{1}$ and $\cH_{2}$ are the same Hilbert space, denoted as $\cH$, then $\bA$ is called \emph{self-adjoint} if $\bA = \bA^{\ast}$. An operator $\bA: \cH \to \cH$ is said to be \emph{positive} if $\langle f, \bA f \rangle_{\cH} \geq 0$ for every $f \in \cH$, and \emph{ strictly positive} if $\langle f, \bA f \rangle_{\cH} > 0$ unless $f = 0$. Furthermore, every bounded positive operator $\bA: \cH \to \cH$ admits a \emph{square root}, that is, a bounded positive operator $\bB: \cH \to \cH$ satisfying $\bB^{2} = \bA$. The square root of $\bA$ is denoted as $\bA^{1/2}$. The \emph{identity} operator $\bI: \cH \to \cH$ is given by $\bI f = f$. The \emph{inverse} $\bA^{-1}: \cH \to \cH$ of $\bA$ is defined as the linear operator which maps every $g \in \cH$ to the unique $f \in \cH$ such that $\bA f = g$. $\bA$ is called \emph{invertible} if its inverse $\bA^{-1}$ exists. Of course, the inverse does not always exist because $\bA$ need not be bijective. Therefore, we define the \emph{pseudoinverse} $\bA^{-}$ of $\bA$ as the linear operator which maps every $g \in \cH$ to the unique $f \in \cH$ such that among all $h \in \cH$ which minimize $\|\bA h - g\|_{\cH}$, $f$ has the minimum norm (see \textcite{beutler1965}). Unsurprisingly, if $\bA$ is bijective, then $\bA^{-} = \bA^{-1}$. An \emph{eigenvalue} of $\bA$ is defined as any $\lambda \in \bbR$ satisfying $\bA f = \lambda f$ for some nonzero $f \in \cH$, which is then called the \emph{eigenvector} (or \emph{eigenfunction}) corresponding to $\lambda$. A related concept is the \emph{spectrum} $\sigma(\bA)$ of $\bA$ which is the set of $\lambda \in \bbR$ for which $\bA - \lambda \bI$ is not invertible. Note that while every eigenvalue of $\bA$ is included in its spectrum $\sigma(\bA)$, not every value in the spectrum qualifies as an eigenvalue.

Let $V$ be another compact subset of a Euclidean space equipped with a finite Borel measure $\nu$. For $A \in L^{\infty}(U \times V, \mu \otimes \nu)$, we define the \emph{integral operator} $\bA: L^{2}(V, \nu) \to L^{2}(U, \mu)$ corresponding to the \emph{integral kernel} $A$ as $\bA f = \int_{U} A(u, v) f(v) ~d\nu(v)$. This is analogous to how every linear transformation between Euclidean spaces can be represented by a matrix. However, not every bounded operator is an integral operator corresponding to some integral kernel. For example, the identity $\bI$ on $L^2(U,\mu)$ is not an integral operator. 

Recall that every matrix can be represented as a linear combination of finite-rank matrices, and the analogous statement applies to linear transformations between Euclidean spaces. In contrast, not every bounded operator can be represented or even approximated (in the operator norm) by linear combinations of finite-rank operators. The operators which do have this property are called \emph{compact} operators. Every compact self-adjoint operator admits an orthonormal basis, which is at most countably infinite, consisting of eigenvectors $\{f_{j}\}_{j=1}^{N}$ with the corresponding eigenvalues $\{\lambda_{j}\}_{j=1}^{N}$ satisfying $\lambda_{j} \to 0$, where $N$ can be infinite.  

\color{purple}

\color{black}

For $1 \leq j \leq p$, let $\cH_{j}$ be Hilbert spaces equipped with the inner products $\langle\cdot, \cdot\rangle_{j}$. The \emph{Cartesian product} of the spaces $\cH_{j}$ is the Hilbert space
$\cH_{1} \times \cdots \times \cH_{p} = \{ (f_{1}, \dots, f_{p}) : f_{j} \in \cH_{j} \}$
equipped with the inner product
\begin{equation*}
    \langle (f_{1}, \dots, f_{p}), (g_{1}, \dots, g_{p})\rangle = \sum_{j=1}^{p} \langle f_{j}, g_{j} \rangle_{\cH_{j}} \mbox{ for } (f_{1}, \dots, f_{p}), (g_{1}, \dots, g_{p}) \in \cH_{1} \times \cdots \times \cH_{p}.
\end{equation*}
An operator from a product Hilbert space $\mathcal{H}_1\times\hdots\times\mathcal{H}_p$ onto itself can be represented as a $p\times p$ matrix whose $ij$th entry is an operator from $\mathcal{H}_i$ to $\mathcal{H}_j$. Such operators will be called \emph{operator matrices}. Note that when $\mathcal{H}_j = \mathbb{R}^{m_j}$ with $\{m_j\}_{i=1}^{p}$ being finite, then an operator matrix is simply a partitioned matrix. Operator matrices being operators themselves, we will use the same boldface notation to denote them -- in his sense, an operator is a $1\times 1$ operator matrix. Notice that when 
For an operator matrix $\bA = [\bA_{ij}]_{i,j=1}^{p}$, we shall use $\dg \bA$ to denote the diagonal part $[\delta_{ij}\bA_{ij}]_{i,j=1}^{p}$, where $\delta_{ij}$ is the Kronecker delta, and $\bA_{0}=\bA - \dg \bA$ to denote the off-diagonal part.

For a more thorough discussion of these concepts, the reader is advised to consult \textcite{hsing2015} or  \textcite{simon2015}.

\subsection{Gaussian Processes and Random Elements}

Recall that two Gaussian random variables, say $Y_{1}$ and $Y_{2}$, are independent if and only if they are uncorrelated, which means $\mathrm{Cov}(Y_{1}, Y_{2}) = 0$. Similarly, for Gaussian random variables $Y_{1}$, $Y_{2}$ and $Y_{3}$, it holds that $Y_{1}$ and $Y_{2}$ are conditionally independent given $Y_{3}$ if and only if the conditional covariance $\mathrm{Cov}(Y_{1}, Y_{2} | Y_{3}) = \bbE[(Y_{1}-\bbE[Y_{1}|Y_{3}])(Y_{2}-\bbE[Y_{2}|Y_{3}])|Y_{3}] = 0$ almost surely. 

The concept of a Gaussian process extends the notion of a Gaussian vector to accommodate a potentially infinite index set $U$. A stochastic process $X = \{X_{u}: u \in U\}$ is said to be Gaussian if for every $n \geq 1$, $\{\alpha_{i}\}_{i=1}^{n} \subset \bbR$ and $\{x_{i}\}_{i=1}^{n} \subset U$, the linear combination $\sum_{i=1}^{n}\alpha_{i}X_{u_{i}}$ is a Gaussian random variable. Note that if $U$ is finite, this is identical to saying that $X$ is a Gaussian random vector. Consider the set $\cL_{0}$ of linear combinations of $\{X_{u}: u\in U\}$. The closure of $\cL_{0}$ under the inner product $\langle Y_{1}, Y_{2} \rangle_{\cL(X)} = \bbE[Y_{1}Y_{2}]$ forms a Hilbert space $\cL(X)$. The probabilistic structure of $X$ is intimately related to the geometric structure of $\cL(X)$. Assume that $\bbE[X_{u}] = 0$ for every $u \in U$, and let $Y \in \cL(X)$ and $L \subset \cL(X)$. Then the conditional expectation $\bbE[Y|L]$ is equal to the projection in $\cL(X)$ of $Y$ to the closed linear subspace spanned by $L$ in $\cL(X)$ (see \textcite{loeve2017}).

The concept of a random variable can be extended to include variables that assume values in spaces other than the real line, giving rise to the idea of random elements. A random element in a Hilbert space thus corresponds to a Borel probability measure in a Hilbert space, analogously to how a random variable corresponds to a Borel probability measure on the real line. Let $\tilde{X}$ be a random element valued in  $L^{2}(U, \mu)$. If $\bbE[\|\tilde{X}\|_{L^{2}(U, \mu)}^{2}] < \infty$, we can define the mean $\tilde{m} = \bbE[\tilde{X}] \in L^{2}(U, \mu)$ and the covariance operator $\tilde{\bK}: L^{2}(U, \mu) \to L^{2}(U, \mu)$ of $\tilde{X}$ as
\begin{equation*}
    \langle g, \tilde{\bK} f \rangle_{L^{2}(U, \mu)} = \bbE[\langle f, \tilde{X} \rangle_{L^{2}(U, \mu)}\langle g, \tilde{X} \rangle_{L^{2}(U, \mu)}] \qquad \mbox{ for } g,f \in L^{2}(U, \mu).
\end{equation*}
A random element $\tilde{X}$ in the Hilbert space $L^{2}(U, \mu)$ is said to be Gaussian if for every $f \in L^{2}(U, \mu)$, the inner product $\langle f, \tilde{X} \rangle_{L^{2}(U, \mu)}$ is a Gaussian random variable. 

In general, for a given stochastic process $X = \{X_{u}: u \in U\}$ with mean $m(u) = \bbE[X_{u}]$ and covariance $K(u, v) = \bbE[(X_{u} - \bbE[X_{u}])(X_{v} - \bbE[X_{v}])]$, there does not necessarily exist a random element $\tilde{X}$ on $L^{2}(U, \mu)$ such that $m = \tilde{m} \in L^{2}(U, \mu)$ and $\tilde{\bK}$ is the integral operator on $L^{2}(U, \mu)$ corresponding to $K$. However, this condition is satisfied if $U$ is a compact metric space, and both $m$ and $K$ are continuous (see \textcite[Theorem 7.4.4]{hsing2015}).

A random variable $Y$ is said to be sub-Gaussian if its sub-Gaussian norm, given by
\begin{equation*}
    \|Y\|_{\psi_{2}} := \inf \{ t > 0: \bbE \exp (X^{2}/ t^{2}) \leq 2 \},
\end{equation*}
is finite. A random element $\tilde{X}$ in $L^{2}(U, \mu)$ is said to be sub-Gaussian if there exists $C > 0$ such that  we have
\begin{equation*}
    \| \langle f, \tilde{X}\rangle_{L^{2}(U, \mu)} \|_{\psi_{2}} \leq C \bbE\Big[|\langle f, \tilde{X}\rangle_{L^{2}(U, \mu)}|^{2}\Big]^{1/2} \qquad \mbox{ for every } f \in L^{2}(U, \mu).
\end{equation*}

An excellent source treating stochastic processes and random elements in a statistical framework is \textcite{hsing2015}. For additional discussion on the properties of sub-Gaussian random variables or elements, the reader is encouraged to consult \textcite{vershynin2018} and \textcite{koltchinskii2017}, respectively.

\subsection{Covariances and Reproducing Kernels}

Let $X = \{X_{u}: u \in U\}$ be a stochastic process on a set $U$ which is \emph{second-order}, meaning $\bbE[|X_{u}|^{2}] < \infty$ for every $u \in U$. We define its covariance kernel $K: U \times U \to \bbR$ as
\begin{equation*}
    K(u, v) = \bbE[(X_{u} - \bbE[X_{u}])(X_{v} - \bbE[X_{v}])] \mbox{ for } u,v \in U.
\end{equation*}
Note that it is \emph{symmetric}, $K(x, y) = K(y, x)$ for every $x, y \in U$, and \emph{positive-definite},
\begin{equation*}
    \textstyle  \sum_{i,j=1}^{n}\alpha_{i}\alpha_{j}K(x_{i}, x_{j}) \geq 0\quad \mbox{ for every } n\geq 1,  \{x_{i}\}_{i=1}^{n} \subset U, \mbox{ and } \{\alpha_{i}\}_{i=1}^{n} \subset \bbR.
\end{equation*}
Symmetric positive-definite kernels are also known as \emph{reproducing kernels}. They possess a very rich analytical structure, which will be utilized extensively in our development. 

We write $K(u, \cdot)$ and $K(\cdot, u)$ to denote the function $f:W \to \bbR$ given by $f(v) = K(u, v)$ for $v \in W$, where $W = U$ unless indicated otherwise. Consider the linear span $\cH_{0}$ of $\{K(u, \cdot): u \in U\}$, which can be equipped with the inner product $\langle K(u, \cdot), K(\cdot, v) \rangle_{\cH(K)} = K(u, v)$ for $u, v \in U$. The closure $\cH(K)$ of $\cH_{0}$ under the inner product $\langle \cdot, \cdot \rangle_{\cH(K)}$ is, in fact, a Hilbert space. Consequently, $\cH(K)$ is referred to as the reproducing kernel Hilbert space (RKHS) associated with $K$. We denote the inner product of $f, g \in \cH(K)$ as $\langle f, g \rangle_{\cH(K)}$. 

If $U$ is finite, say $U = \{1, \dots, p\}$, the space $\cH(K)$ can be thought of as consisting Euclidean vectors which are linear combinations of the columns of the matrix $\mathsf{K} = [K(i, j)]_{i,j=1}^{p}$. Thus $\cH(K)$ is the range of $\mathsf{K}$. For $\mathbf{x}, \mathbf{y} \in \cH(K)$, the RKHS inner product can be evaluated as the Mahalanobis inner product, that is,
\begin{equation*}
    \langle \mathsf{x}, \mathsf{y} \rangle_{\cH(K)} = \mathsf{x}^{\top} \mathsf{K}^{-} \mathsf{y},
\end{equation*}
where $\mathsf{K}^{-}$ is the pseudoinverse of $\mathsf{K}$ (see \textcite[Section 2.3.3]{paulsen2016}). 

More generally, if $U$ is a compact subset of a Euclidean space equipped with a finite Borel measure $\mu$ supported on $U$, the space $\cH(K)$ is the range of the square root $\bK^{1/2}$ of the integral operator $\bK$ corresponding to the integral kernel $K$ and the RKHS inner product can be calculated as
\begin{equation*}
    \langle f,g\rangle_{\mathcal{H}(K)} = \langle \bK^{-1/2}f,\bK^{-1/2}g\rangle_{L^2(U)},
\end{equation*}
where $\bK^{-1/2}$ is the pseudoinverse of $\bK^{1/2}$ (see \textcite[Theorem 11.18]{paulsen2016}). This is a striking result in light of the fact that the product $\langle \bK^{-1/2}f,\bK^{-1/2}g\rangle_{L^2(U)}$ does not depend on the choice of the measure $\mu$, since $\langle f,g\rangle_{\mathcal{H}(K)}$ can be defined without any reference to a measure.

If $X$ is a Gaussian process, then the Hilbert space $\cL(X)$ generated by the process is isometrically isomorphic to $\cH(K)$. The isometry maps $X_{u} \in \cL(X)$ to $K(u, \cdot) \in \cH(K)$ for every $u \in U$. This result is known as \emph{Lo\`{e}ve's theorem} or \emph{Lo\`{e}ve isometry} (see \textcite[Theorem 35]{berlinet2011}) and it permits us to rewrite statements concerning the Gaussian process $X$ in terms of its RKHS.

We denote the restriction of $f \in \cH(K)$ to $W \subset U$ as $f|_{W}$ and the restriction of $K$ to $V \times W$ for or $V, W \subset U$ as $K|_{V \times W}$. For a given $W \subset U$, every restriction $f|_{W}$ of $f \in \cH(K)$ is, in fact, a member of the RKHS $\cH(K_{W})$ of the restriction $K_{W} = K|_{W \times W}$. Let $\Pi_{W}$ denote the projection in $\cH(K)$ to the closed linear subspace spanned by $\{K(u, \cdot): u \in W\}$. Then the RKHS $\cH(K_{W})$ is isomorphic to a subspace of $\cH(K)$ given by $\{\Pi_{W}K(u, \cdot): u \in U\}$ and the isometry maps $K(u, \cdot)|_{W}$ to $\Pi_{W}K(u, \cdot)$ for $u \in W$. This result is called \emph{subspace isometry} (see appendix of \textcite{waghmare2021}). 

An excellent introduction to the theory of reproducing kernels can be found in \cite{berlinet2011}, \textcite{paulsen2016} and \textcite{aronszajn1950}.

\section{Graphical Representation of Gaussian Processes}\label{sec:graph_repn}

We begin by characterizing the relationship between the conditional independence structure of a Gaussian process $X$ and its covariance kernel $K$. We then use this characterization to define the graph of a Gaussian process and to discuss certain conceptual differences with respect to the finite index setting.

\subsection{The Separation Equation}

Let $X = \{X_{u} : u \in U\}$ be a Gaussian process on a set $U$ satisfying the global Markov property (\ref{eqn:sepn_condn}) for some graph $\Omega \subset U \times U$. Because $X$ is Gaussian, this is equivalent to requiring that for every $u, v \in U$ separated by $W \subset U$ (see Figure \ref{fig:omega_and_pi1} (a)), the conditional covariance  
\begin{equation*}
    \Cov(X_{u}, X_{v}|X_{W}) = \bbE[X_{u}X_{v}| X_{W}] - \bbE[X_{u}|X_{W}]\cdot\bbE[X_{v}|X_{W}] 
\end{equation*}
must vanish almost surely. Taking the expectation and using the law of iterated expectation, this implies that
\begin{equation}\label{eqn:sepn_expctn}
    \bbE[X_{u}X_{v}] = \bbE\Big[\bbE[X_{u}|X_{W}]\cdot\bbE[X_{v}|X_{W}]\Big] 
\end{equation}
almost surely. We shall now express this statement in terms of the kernel $K$.

\begin{figure}
    \begin{subfigure}{0.49\textwidth}
    \centering
    \begin{tikzpicture}
        \draw (0, 0) -- (5, 0) -- (5, 5) -- (0, 5) -- cycle;
        \draw [fill=red!15] 
        (0,0) -- 
        (2,0) .. controls (2.8,2.6) and (4,0.6) .. (5,3) --
        (5,5) -- 
        (3,5) .. controls (0.6,4) and (2.6,2.8) .. (0,2) -- cycle;
        
        \def \x {1.2}
        \def \h {2.4}
        \def \b {0.8}
        \draw [dashed] (\x, \x) -- (\x + \h, \x) -- (\x + \h, \x + \h) -- (\x, \x + \h) -- cycle;
        \draw (\x + \h/2, \x + \h/2) node[fill=red!15] {$K_{W}$};

        \draw[|-|] (\x + \h + \b, \x) -- (\x + \h + \b, \x + \h) node [midway, fill = white, sloped, above, text opacity=1, opacity = 0, near end] {$K(u, \cdot)$};
        \draw[|-|] (\x, \x - 0.9*\b) -- (\x + \h, \x - 0.9*\b) node [midway,fill = white, above, text opacity=1, opacity = 0, near end] {$K(\cdot, v)$};

        \draw [fill] (\x + \h + \b, \x - 0.9*\b) circle [radius=0.05];
        \draw (\x + \h + \b, \x - 0.9*\b) node[below] {\scriptsize $(u, v)$};
        
        \draw (0.6, 0.6) node [fill=red!15] {$\Omega$};
    \end{tikzpicture}
    \caption{}
    \end{subfigure}
    \begin{subfigure}{0.49\textwidth}
    \centering
    \begin{tikzpicture}
        \draw (0, 0) -- (5, 0) -- (5, 5) -- (0, 5) -- cycle;
        \def \h {1}
        
        \draw [dashed, fill=green!15] (0,0) -- (\h, 0) -- (5, 5 - \h) node [above, sloped, near end] {$\Omega + \bbB_{\epsilon}$} -- (5,5) -- (5 - \h, 5) -- (0, \h) -- cycle;

        \draw (0, 0) -- (5, 5) node [fill=green!15, draw, sloped, near start] {$\Omega_{X}$};
    \end{tikzpicture}
    \caption{}
    \end{subfigure}\\[0.3cm]
    \caption{(a) An example of $u, v \in U$ separated by $W \subset U$ indicated by $W \times W$ (dashed square) along with the restrictions $K(u, \cdot)|_{W}$ and $K(\cdot, v)|_{W}$, and (b) The graph $\Omega_{X}$ of Brownian motion (the diagonal) and the $\epsilon$-envelope $\Omega_{X} + \bbB_{\epsilon}$ (in green).}
    \label{fig:omega_and_pi1}
\end{figure}
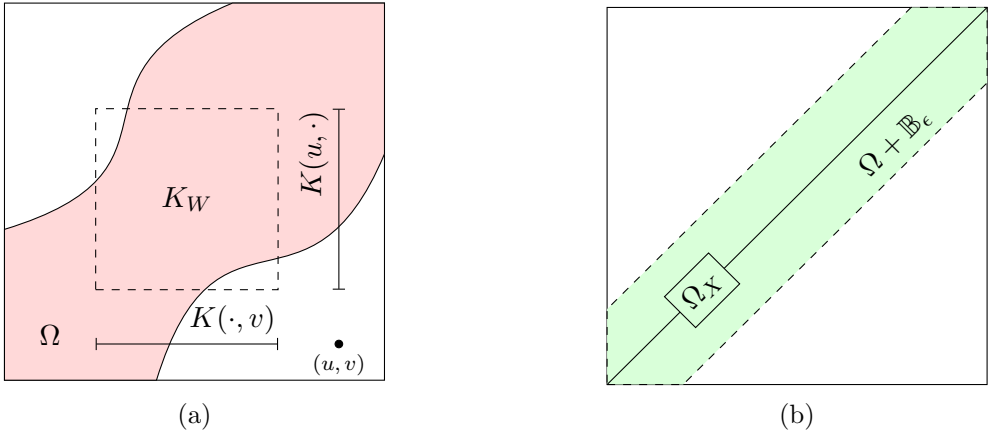

Recall that the closed linear span $\cL(X)$ of $X = \{X_{u}: u\in U\}$ under the norm $\|Y\|_{\cL(X)} = \bbE[Y^{2}]$ forms a Hilbert space under the inner product $\langle Y_{1}, Y_{2} \rangle_{\cL(X)} = \bbE[Y_{1}Y_{2}]$ induced by the norm. Because $X$ is Gaussian, the conditional expectation $\bbE[X_{u}|X_{W}]$ is equal to the projection of $X_{u}$ to the closed linear span of $X_{W}$ in $\cL(X)$ (see \textcite{loeve2017}). Furthermore, according to a result known as  Lo\`{e}ve isometry, the space $\cL(X)$ is isometrically isomorphic to the reproducing kernel Hilbert space $\cH(K)$ with the isometry given by $X_{u} \mapsto K(u, \cdot)$ (see \textcite[Theorem 35]{berlinet2011}). Together, these results allow us to rewrite Equation (\ref{eqn:sepn_expctn}) as 
\begin{equation}\label{eqn:sepn_rkhs}
    \langle K(u, \cdot), K(\cdot, v) \rangle_{\cH(K)} = \langle \Pi_{W} K(u, \cdot), \Pi_{W} K(\cdot, v) \rangle_{\cH(K)}
\end{equation}
where $\Pi_{W}$ denotes the projection in $\cH(K)$ to the closed linear subspace generated by $\{K(w, \cdot) : w \in W\}$. As before, we shall consider it implicit that if $u$ and $v$ are disconnected then they are separated by $W = \varnothing$ and $K(u, v) = 0$. 

By the reproducing property, we can write the left hand side of (\ref{eqn:sepn_rkhs}) as $\langle K(u, \cdot), K(\cdot, v) \rangle_{\cH(K)} = K(u, v)$. The inner product of the projections on the right hand side, $\langle \Pi_{W} K(u, \cdot), \Pi_{W} K(\cdot, v) \rangle_{\cH(K)}$ can be evaluated by taking the inner product of the restrictions $K(u, \cdot)|_{W}$ and $K(\cdot, v)|_{W}$ in the reproducing kernel Hilbert space of the restriction $K_{W} = K|_{W \times W}$ of the kernel $K$ by the subspace isometry (see \textcite{paulsen2016}). Thus, $K(u, v) = \langle K(u, \cdot)|_{W}, K(\cdot, v)|_{W} \rangle_{\cH(K_{W})}$. For the sake of brevity, we write this as
\begin{equation}\label{eqn:sepn_eqn}
    K(u, v) = \langle K(u, \cdot), K(\cdot, v) \rangle_{\cH(K_{W})},
\end{equation}
since the domain of $K(u, \cdot)$ and $K(\cdot, v)$ is understood to be $W$. We shall refer to (\ref{eqn:sepn_eqn}) as the \emph{separation equation}. Going in the opposite direction, notice that (\ref{eqn:sepn_rkhs}) implies
\begin{equation*}
    \langle K(u, \cdot) - \Pi_{W}K(u, \cdot), K(\cdot, v) - \Pi_{W}K(\cdot, v) \rangle_{\cH(K)} = 0.
\end{equation*}
By of Gaussianity and Lo\`{e}ve isometry, this means that $X_{u} - \bbE[X_{u} | X_{W}]$ and $X_{v} - \bbE[X_{v} | X_{W}]$ are independent. Additionally, they are both independent of $X_{W}$. It follows that 
\begin{equation*}
    \Cov(X_{u}, X_{v}|X_{W}) = \bbE\Big[(X_{u} - \bbE[X_{u}|X_{W}])(X_{v} - \bbE[X_{v}|X_{W}])| X_{W} \Big] = 0.
\end{equation*}
To summarize, we have established the following theorem.
\begin{theorem}\label{thm:condnl_ind}
    Given a Gaussian process $X = \{X_{u} : u \in U\}$ and a graph $\Omega \subset U\times U$, the following two statements are equivalent: 
    \begin{itemize}
\item[(A)]   For every $u, v \in U$ separated by $W \subset U$ in $\Omega$  
$$
 X_{u} \CI X_{v} ~|~ X_{W}.
 $$
\item[(B)] For every $u, v \in U$ separated by $W \subset U$ in $\Omega$
    \begin{equation*}
        K(u, v) = \langle K(u, \cdot), K(\cdot, v) \rangle_{\cH(K_{W})}.
    \end{equation*}
\end{itemize}
\end{theorem}

The chief virtue of Theorem \ref{thm:condnl_ind} is that it expresses the conditional independence structure of a Gaussian process in terms of a relatively simple notion, namely inner products in a Hilbert space. Inner products are almost as easy to work with in infinite dimensions as they are in finite dimensions. This is unlike densities, which are often used to deal with finitely-indexed graphical models but do not generalize for want of a suitable analogue of the Lebesgue measure in infinite dimensions. For the same reason, it is unclear how the celebrated Hammersley-Clifford theorem may be formulated for graphical models with infinite index sets.

It is important to note that, in general, the global Markov property (\ref{eqn:sepn_condn}) used to derive (\ref{eqn:sepn_eqn}) is \emph{not} equivalent to the local or pairwise Markov properties when the index set is uncountable. Indeed, for continuous covariances, the counterpart of (\ref{eqn:sepn_eqn}) corresponding to the local or pairwise Markov property is vacuously true for every graph. This trivializes the whole notion. It turns out that the equivalence does hold for \emph{countably} infinite index sets under additional conditions (see \textcite{montague2018}).\\

One of the properties which force a Gaussian process to obey the separation equation with respect to a graph is the analyticity of the covariance kernel, as illustrated by the following example. 
\begin{example}\label{ex:analytic}
    Let $X = \{X_{t}\}_{t \in I}$ be a Gaussian process on the unit interval $I$ with an analytic covariance $K$. Then $K$ satisfies the separation equation for every $\Omega$ which contains the strip $\{(u, v): |u - v| \leq w\}$ for some $w > 0$. Indeed, for any two points $u, v \in I$ separated by $W \subset I$, $W$ must contain an interval of finite length. This implies that the function $f(\cdot) = K(u, \cdot) - \Pi_{W}K(u, \cdot)$ is zero on $W$ because $f(w) = \langle K(u, \cdot) - \Pi_{W}K(u, \cdot), K(\cdot, w) \rangle  = 0$ for $w \in W$ by the projection theorem. Because $W$ contains an interval of finite length and $f$ is analytic (since $f \in \cH(K)$ where $K$ is analytic, see \textcite{saitoh2016}), we have $$f(\cdot) = K(u, \cdot) - \Pi_{W}K(u, \cdot) \equiv 0.$$ Similarly, we can show that $K(\cdot, v) - \Pi_{W}K( \cdot, v) \equiv 0$ and as a result, $$ K(u, v) = \langle K(u, \cdot), K(\cdot, v) \rangle = \langle \Pi_{W}K(u, \cdot), \Pi_{W}K(\cdot, v) \rangle = \langle K(u, \cdot), K(\cdot, v) \rangle_{\cH(K_{W})}.$$ The conclusion follows. This argument can be easily extended to any Gaussian process on a connected domain in a Euclidean space with an analytic covariances.
\end{example}

It is natural to ask why the relationship between the conditional independence structure of $X$ and its covariance $K$ has to be expressed by such tortuous means. After all, if $\bX = \{X_{j}\}_{j = 1}^{p}$ is a Gaussian random vector with a non-singular covariance matrix $\bC$, satisfying the global Markov property (\ref{eqn:sepn_condn}) for some graph $\Omega \subset \{1, \dots, p\}^{2}$, then the relation between $\Omega$ and $\bC$ is described very elegantly by the following well-known result:
\begin{equation}\label{eqn:inv_zero}
    \bP_{ij} = 0 \mbox{ if and only if } i \mbox{ and } j \mbox{ are not adjacent in } \Omega,
\end{equation}
where $\bP=\bC^{-1}$ is the precision matrix. In other words, the zero entries of the matrix $\bP$ correspond precisely to missing edges of the graph $\Omega$.

Having an elegant \emph{inverse zero characterization} similar to (\ref{eqn:inv_zero}) for kernels is impeded by technical difficulties, however. Namely, the ``inverse'' of a kernel on an uncountable domain $U\times U$ is not a well-defined notion in general. If we attempt to make the space of kernels into a ring by defining the product of two kernels $K_{1}$ and $K_{2}$ in a natural way by
\begin{equation*}
    K_{1} \odot K_{2}(u, v) = \int_{U} K_{1}(u, w) K_{2}(w, v) ~d\mu(u) 
\end{equation*}
where $\mu$ is a finite Borel measure on $U$, then the resulting space ends up being a non-unital ring. This because no kernel can serve as a multiplicative identity the way the identity matrix does for matrices. Even if we admit the Dirac delta $\delta(u - v)$ as the identity, no kernel would admit an inverse. On the other hand, we can directly consider the inverse of the integral operator $\bK$ induced by $K$ as
\begin{equation*}
    \bK f(u) = \int_{U} K(u, v) f(v) ~d\mu(v)
\end{equation*}
and define its support indirectly as follows: $U_{1} \times U_{2} \subset \supp(\bK^{-1})^{c}$ if for every pair $f, g$ in the range of $\bK$ such that $\supp f = U_{1}$ and $\supp g = U_{2}$, we have  $\langle f, \bK^{-1} g \rangle_{L^{2}(\mu)} = 0$. This parallels the matrix case, which can also be interpreted via quadratic forms $\mathbf{x}^{\top} \bP \mathbf{y}$ involving sparse vectors $\mathbf{x}, \mathbf{y}$. But this too is inconvenient given that $\bK^{-1}$ is unbounded in general, leading to delicate conditions on suitable test functions $f,g$ -- this is particularly awkward in a statistical context, where $\bK$ is to be estimated from finitely many observations, and hence the true RKHS is not identifiable. 

Unlike the inverse zero characterization (\ref{eqn:inv_zero}), the separation equation (\ref{eqn:sepn_eqn}) has the virtue of holding true \emph{regardless} of whether $U$ is finite or whether the covariance is boundedly invertible. {Furthermore, its defining inner product involves only a pair of specific functions specified by the covariance itself, and that are bona fide assured to be elements of the requisite RKHS.} But this comes at the expense of the condition being tedious to verify since one needs to exhaust all admissible combinations of $u$, $v$ and $W$. 

In Section \ref{sec:partitioning}, however, we will show that this shortcoming can be circumvented, by appealing to the notion of \emph{resolution}. Namely, we will show that an analogue of the inverse zero characterization (\ref{eqn:inv_zero}) holds even for infinite domains $U$, as long as we are willing to specify the graph $\Omega$ up to some finite resolution. We will furthermore show that the characterization behaves coherently under refinement of the resolution.

\subsection{The Graph of a Stochastic Process}

Theorem \ref{thm:condnl_ind} allows us to verify whether the conditional independence structure of a Gaussian process is compatible with a given graph, in the sense of the global Markov property (\ref{eqn:sepn_condn}). But it does not specify the graph, nor does it inform on the uniqueness of a graph compatible with a Gaussian process $X$. In the finite-dimensional setting, these questions are answered unequivocally: the zero pattern of the inverse covariance (\ref{eqn:inv_zero}) defines an adjacency matrix, so the question boils down to the invertibility of the covariance.

To address this question, we note that satisfaction of the separation equation is \emph{heritable} with respect to inclusion, that is, if $K$ satisfies the separation equation (\ref{eqn:sepn_eqn}) for a graph $\Omega$ then it also does so for every graph $\Omega'$ which contains $\Omega$ (see \textcite{waghmare2021}). Assume that the index set $U$ of $X$ is a compact subset of $\bbR^{n}$ with the natural topology. The previous observation suggests intersecting all compatible graphs to define \emph{the} graph of a process.

\begin{definition}\label{def:graph_X}
     We define the graph of $X$, denoted by $\Omega_{X}$,
    to be  the intersection of all closed graphs $\Omega$ for which the separation equation (\ref{eqn:sepn_eqn}) is satisfied by the covariance $K$ of $X$.
\end{definition}

Unlike the finite-dimensional setting, there is no guarantee that $X$ will satisfy the global Markov property (\ref{eqn:sepn_condn}) for $\Omega = \Omega_{X}$. This may seem dissatisfying given that we would have hoped $\Omega_{X}$ to be interpretable as the ``minimal'' graph satisfying the separation equation. But it does point to an interesting aspect special to the uncountably infinite index case, namely that satisfaction of the separation equation is not closed under infinite intersections. This means that for certain processes there is simply no ``minimal" graph for which the process satisfies the global Markov property. The following example illustrates this peculiar feature of uncountably infinite index sets.

\begin{example}\label{ex:brownian}
    Let $W = \{W_{t}\}_{t \in I}$ be the Brownian motion process on the unit interval $I$. Its covariance $K(u, v) = u \wedge v$ satisfies the separation equation for the strip $\Omega_{w} = \{(u,v): |u-v| \leq w\}$ for every $w > 0$. Indeed, if $u, v \in I$ are separated by some subinterval $J \subset I$, then we can assume without loss of generality that $u > v$ and by the Markov property $W_{u} = \bbE[W_{u}| W_{J}]$. By Lo\`{e}ve isometry, $K(u, \cdot) = \Pi_{J}K(u, \cdot)$ and by taking the inner product with $K(\cdot, v)$, we get $K(u, v) = \langle K(u, \cdot), K(\cdot, v) \rangle_{\cH(K_{J})}$.

    By definition, the graph of $W$ is given by $$\Omega_{W} = \cap_{w > 0}~ \Omega_{w} = \{(u, v) : u = v\}$$ which is the empty graph on $I$ in which no two vertices are adjacent. Note that $K$ does not satisfy the separation equation for $\Omega_{W}$ because that would mean that $K(u, v) = 0$, which is false. The same argument can be made for Gaussian processes which are Markov, multiple Markov (\textcite{hida1993}) or possess analytic covariances, as covered in Example \ref{ex:analytic}.
\end{example}

Determining the conditions under which $X$ satisfies the separation equation for $\Omega = \Omega_{X}$ seems to be a challenging technical problem interfacing the theory of infinite graphs and the analytical properties of covariances, and is beyond the scope of this article. However, if $U \times U$ is equipped with a metric, then one can make up for the gap in intuition resulting from this anomaly by thinking of the conditional independence structure of a process $X$ as being represented by $\Omega_{X} + \bbB_{\epsilon}$ instead of $\Omega_{X}$, where $\Omega_{X} + \bbB_{\epsilon}$ is the $\epsilon$\emph{-envelope} of $\Omega_{X}$, that is, the set of points within $\epsilon$ distance from $\Omega_{X}$ where $\epsilon$ can be taken to be arbitrarily small (see Figure \ref{fig:omega_and_pi1} (b)). For Gaussian processes on the unit interval which are Markov, multiple Markov or have analytic covariances, the conditional independence structure is then given by an  $\epsilon$-strip centered along the diagonal. This ``open" formulation rescues the sought intuition in situations like Example \ref{ex:brownian}.

The graph $\Omega_{X}$ (or its $\epsilon$-envelope) presents an interesting target for estimation given $n$ independent and identically distributed realizations of $X$. In Section \ref{sec:partitioning}, we shall present an analogue of the inverse zero characterization (\ref{eqn:inv_zero}) for kernels up to a finite resolution, and we shall present sufficient conditions on $K$ for $\Omega_{X}$ to be identifiable exactly or up to such a finite resolution.

\section{Resolving Uncountably Infinite Graphs}\label{sec:partitioning}

In this section, we shall recover an analogue of the inverse zero characterization (\ref{eqn:inv_zero}) for kernels. This will enable us to verify the global Markov property in a practically feasible manner, and will also allow us to deploy the well-established thresholding approach of graph recovery  (developed in Section \ref{sec:methodology}).

As previously argued, an \emph{exact} inverse zero characterization is unavailable and likely infeasible for our setting, in light of the notably dissimilar algebraic properties of kernels as compared to matrices. Our approach will thus consist of introducing an appropriate notion of \emph{resolution}, and contenting ourselves with a characterisation valid for any given finite resolution. That being said, we will also require that our characterisation be compatible across refinements of the resolution, and that it identify the true graph in the limit as resolution diverges.

From a mathematical point of view, resolving a graph consists of specifying a sequence of constructible approximations thereof.
From a statistical point of view, focussing on a finite resolution is arguably natural, or even necessary, since the number of potential graphs is uncountably infinite, and we need to infer the graph from finitely many realizations. 
Our estimation theory will reflect how the resolution can increase as a function of sample size, thus informing  how finely we can hope to discern the conditional independence structure of the process from a finite amount of data.

Our results thus far apply to any covariance kernel $K$ on any set $U$. Hereonwards, we shall additionally assume $K$ to be continuous and $U$ to be a compact subset of a Euclidean space equipped with a finite Borel measure $\mu$ supported on $U$. The results can be extended without much difficulty to more general sets with topological structure enabling a generalization of Mercer's theorem, however we shall stick to the compact Euclidean setting for simplicity.

\subsection{Resolution}\label{subsec:resolution}

Let $U$ be a compact subset of $\bbR^{d}$ equipped with a finite Borel measure $\mu$ supported on $U$. Let $K: U \times U \to \bbR$ be a continuous covariance kernel. We shall now make precise what we mean  by \emph{resolution} in this context.

\begin{figure}
    \begin{subfigure}{0.49\textwidth}
    \centering
    \begin{tikzpicture}
        \draw (0, 0) -- (5, 0) -- (5, 5) -- (0, 5) -- cycle;
        \draw [fill=green!15] 
        (0,0) -- 
        (2,0) .. controls (2.8,2.6) and (4,0.6) .. (5,3) --
        (5,5) -- 
        (3,5) .. controls (0.6,4) and (2.6,2.8) .. (0,2) -- cycle;
        \foreach \i in {1,...,19} {
            \draw[densely dotted, line width = 0.1 mm] (\i/4, 0) -- (\i/4, 5);
            \draw[densely dotted, line width = 0.1 mm] (0, \i/4) -- (5, \i/4);
        }
        \draw (1,1) node [fill=green!15] {$\Omega$};
    \end{tikzpicture}
   \caption{}
    \end{subfigure}
    \begin{subfigure}{0.49\textwidth}
    \centering
    \begin{tikzpicture}
        \draw (0, 0) -- (5, 0) -- (5, 5) -- (0, 5) -- cycle;
        \draw [fill=red!15] 
        (0/4,0/4) -- (9/4,0/4) -- (9/4, 2/4) -- 
        (10/4, 2/4) -- (10/4, 4/4) -- (11/4, 4/4) -- (11/4, 5/4) -- 
        (13/4, 5/4) -- (13/4, 6/4) -- (16/4, 6/4) -- (16/4, 7/4) -- 
        (18/4, 7/4) -- (18/4, 8/4) -- (19/4, 8/4) -- (19/4, 10/4) -- 
        (20/4, 10/4) -- (20/4, 20/4);
        
        \draw [rotate = -90, xscale = -1, fill=red!15] 
        (0/4,0/4) -- (9/4,0/4) -- (9/4, 2/4) -- 
        (10/4, 2/4) -- (10/4, 4/4) -- (11/4, 4/4) -- (11/4, 5/4) -- 
        (13/4, 5/4) -- (13/4, 6/4) -- (16/4, 6/4) -- (16/4, 7/4) -- 
        (18/4, 7/4) -- (18/4, 8/4) -- (19/4, 8/4) -- (19/4, 10/4) -- 
        (20/4, 10/4) -- (20/4, 20/4);
        \draw [dashed] 
        (0,0) -- 
        (2,0) .. controls (2.8,2.6) and (4,0.6) .. (5,3) --
        (5,5) -- 
        (3,5) .. controls (0.6,4) and (2.6,2.8) .. (0,2) -- cycle;
        \foreach \i in {1,...,19} {
            \draw[densely dotted, line width = 0.1 mm] (\i/4, 0) -- (\i/4, 5);
            \draw[densely dotted, line width = 0.1 mm] (0, \i/4) -- (5, \i/4);
        }
        \draw (1,1) node [fill=red!15] {$\Omega^{\pi}$};
    \end{tikzpicture}
    \caption{}
    \end{subfigure}
    \bigskip\medskip \caption{(a) A graph $\Omega$ on an interval and (b) its $\pi$-resolution approximation $\Omega^{\pi}$. Each cell of the grid represents a pixel $U_{i} \times U_{j}$ where $U_{i}, U_{j} \in \pi$.}
    \label{fig:omega_and_pi}
\end{figure}
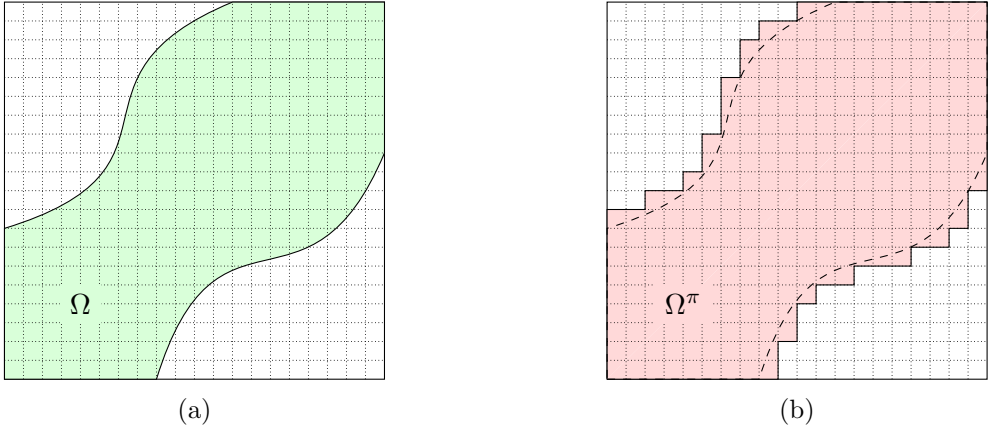

A \emph{partition} $\pi$ of $U$ is a finite collection $\{U_{j}\}_{j=1}^{p}$ such that (a) every $U_{j}$ is a Borel subset of $U$ such that $\mu(\tilde{U}) > 0$ for every nonempty subset $\tilde{U} \subset U_{j}$ which is relatively open in $U_{j}$, (b) $\{U_{j}\}_{j=1}^{p}$ is \emph{exhaustive} in that $\cup_{j=1}^{p} U_{j} = U$, and (c) \emph{disjoint} in that $U_{i} \cap U_{j} = 0$ for $i \neq j$. The additional technical conditions in (a) simply ensure that Mercer's theorem applies to $U_{j}$ individually as it does to $U$ as a whole. In common mathematical parlance, a partition need not be finite nor contain only Borel sets, but using the above definition lends brevity to our presentation.

We shall refer to sets of the form $U_{i} \times U_{j}$ for $1 \leq i,j \leq p$ as \emph{pixels}. A \emph{$\pi$-resolution graph} $\Omega \subset U \times U$ is a union of pixels which includes the pixels on the diagonal, that is, $\cup_{j=1}^{p} U_{j} \times U_{j} \subset \Omega$. 
Every graph $\Omega$ on $U$ admits a \emph{best $\pi$-resolution approximation} $\Omega^{\pi}$ defined as the intersection of all $\pi$-resolution graphs on $U$ which contain $\Omega$. 
Thus, $\Omega^{\pi}$ is the smallest $\pi$-resolution graph which contains $\Omega$. 
Alternatively, we can express $\Omega^{\pi}$ as the union of all $U_{i} \times U_{j}$ which intersect with $\Omega$. As above, we shall denote the best $\pi$-resolution approximation of a graph $\Omega$ on $U$ by $\Omega^{\pi}$. Figure \ref{fig:omega_and_pi} illustrates the difference between $\Omega$ and $\Omega^{\pi}$.

We shall denote by $\tilde{\Omega}^{\pi}$ the intersection $\cap_{\epsilon > 0}~(\Omega + \bbB_{\epsilon})^{\pi}$ where $\bbB_{\epsilon}$ denotes the intersection of the Euclidean ball of radius $\epsilon$ in $\bbR^{2d}$ with $U \times U$ and the sum $A + B$ denotes the set $\{a + b: a \in A \mbox{ and } b \in B\} \cap U \times U$. Because the sets ``decrease" as $\epsilon \to 0$ in that $\Omega_{X} + \bbB_{\epsilon_{1}} \subset \Omega_{X} + \bbB_{\epsilon_{2}}$ for $\epsilon_{1} < \epsilon_{2}$, we can also write $\tilde{\Omega}^{\pi}$ as $\lim_{\epsilon \to 0}~(\Omega_{X} + \bbB_{\epsilon})^{\pi}$. The distinction between $\Omega^{\pi}$ and $\tilde{\Omega}^{\pi}$ is mainly technical and is a consequence of the fact observed in Example \ref{ex:brownian} that for certain processes there is no minimal graph $\Omega$ for which the covariance satisfies the separation equation. For this reason and for lack of a better alternative, we shall refer to both $\Omega^{\pi}$ and $\tilde{\Omega}^{\pi}$ as the best $\pi$-resolution approximation of $\Omega$ while indicating which of the two we mean by their respective symbols.

\begin{example}\label{ex:graphs}
    A simple instance of how $\tilde{\Omega}_{X}$ can differ from $\Omega_{X}^{\pi}$ is given by the processes considered in Examples \ref{ex:analytic} and \ref{ex:brownian}, where $\Omega_{X} = \{(u, v): u = v\}$. Thus, $\Omega_{X}^{\pi} = \cup \{U_{i} \times U_{j} : | i - j | = 0 \} $ but $\tilde{\Omega}_{X}^{\pi} = \cup \{U_{i} \times U_{j} : | i - j | \leq 1 \}$ since the strip $\{(u, v): |u - v| < \epsilon\}$ always intersects the pixels $U_{i} \times U_{j}$ for which $|i - j| = 1$.  
\end{example}

\subsection{Approximate Inverse Zero Characterization}\label{subsec:prec_op_matrix}

We shall now show how one can recover the best $\pi$-resolution approximation of $\Omega$ from the covariance kernel $K(s,t) = \bbE[X_{s}X_{t}]$ of $X$. Let $K_{ij} = K|_{U_{i} \times U_{j}}$. For $1 \leq i,j \leq p$, let $\bK_{ij}: L^{2}(U_{j}, \mu) \to L^{2}(U_{i}, \mu)$ be the integral operator induced by the integral kernel $K_{ij}$ given by
\begin{equation*}
    \bK_{ij}f(u) = \int_{U_{j}} K_{ij}(u, v) f(v) ~d\mu(v)
\end{equation*}
Define the \emph{covariance operator matrix} $\bK_{\pi}$ induced by the partition $\pi$ as $\bK_{\pi} = [\bK_{ij}]_{i,j = 1}^{p}$. 
According to a well-known result of \cite{baker1973}, for $i \neq j$ there exists a unique bounded linear operator $\bR_{ij}: L^{2}(U_{j}, \mu) \to L^{2}(U_{i}, \mu)$ such that $\bC_{ij} = \smash{\bC_{ii}^{1/2}\bR_{ij}\bC_{jj}^{1/2}}$ and $\|\bR_{ij}\| \leq 1$, which can be calculated as $$\bR_{ij} = \smash{\bK_{ii}^{-1/2}\bK_{ij}^{\phantom{/}}\bK_{jj}^{-1/2}},$$ where $\smash{\bC_{ii}^{-1/2}}$ and $\smash{\bC_{jj}^{-1/2}}$ are the operator pseudoinverses of $\bC_{ii}^{1/2}$ and $\bC_{jj}^{1/2}$ respectively. We define the \emph{correlation operator matrix} $\bR_{\pi}$ induced by the partition $\pi$ as $\bR_{\pi} = [\bR_{ij}]_{i,j=1}^{p}$ specified entrywise by $\bR_{ii} = \bI$ and $\bR_{ij} = \smash{\bC_{ii}^{-1/2}\bC_{ij}\bC_{jj}^{-1/2}}$ for $i \neq j$. Alternatively, we can write $\bR_{\pi}$ as $$\bR_{\pi} = \bI + \smash{[\dg \bK_{\pi}]^{-1/2}(\bK_{\pi} - \dg \bK_{\pi})[\dg \bK_{\pi}]^{-1/2}}.$$ If $\bR_{\pi}$ is invertible and then we can define the \emph{precision operator matrix} $\bP_{\pi} = [\bP_{ij}]_{i,j=1}^{p}$ as the inverse of $\bR_{\pi}$, that is $\bP_{\pi}^{} = \bR_{\pi}^{-1}$. 

\smallskip
\noindent The key result can now be stated as follows:

\begin{theorem}\label{thm:best_piapprx}
    If $\bR_{\pi}$ is invertible, then the the graph $\Omega_{X}$ and the precision operator matrix $\bP_{\pi}$ induced by the partition $\pi$ are related as follows:
    \begin{equation}\label{eqn:omega_pi}
        \tilde{\Omega}_{X}^{\pi} \equiv \lim_{\epsilon \to 0}~(\Omega_{X} + \bbB_{\epsilon})^{\pi} \subset \cup~ \{U_{i} \times U_{j}: \|\bP_{ij}\| \neq 0\}.
    \end{equation}
    If, in addition, for every $\epsilon > 0$ there exists a partition $\pi_{\epsilon}$ of $U$ such that every pixel is contained within a ball of radius $\epsilon$ and $\bR_{\pi_{\epsilon}}$ is invertible, then the above relation is an equality. In other words, $\tilde{\Omega}_{X}^{\pi}$ is same as the union of $U_{i} \times U_{j}$ for $(i, j)$ such that $\bP_{ij} \neq \bzero$.
\end{theorem}

Thus by discerning which entries of the partition-induced correlation operator matrix are zero, one can work out a finite resolution approximation $\tilde{\Omega}_{X}^{\pi}$ of $\Omega_{X}^{}$. It follows immediately that $\tilde{\Omega}_{X}^{\pi}$ is identifiable if $\bP_{\pi}$ is invertible. We expect that the technical condition for equality is an artifact of our proof technique, and not an essential feature of the problem.

\subsection{Refinement and Identifiability}

If we know $\Omega^{\pi_{1}}$ and $\Omega^{\pi_{2}}$ then we can get a finer approximation of $\Omega$ by simply taking their intersection. The resulting graph $\Omega^{\pi_{1} \wedge \pi_{2}} = \Omega^{\pi_{1}} \cap \Omega^{\pi_{2}}$ is the best $(\pi_{1} \wedge \pi_{2})$-approximation where the partition $\pi_{1} \wedge \pi_{2}$ is the \emph{refinement} of the partitions $\pi_{1}$ and $\pi_{2}$ given by $\{U_{1} \cap U_{2} : U_{1} \in \pi_{1} \mbox{ and } U_{2} \in \pi_{2}\}$ which is in other words composed of the intersections of the sets in the original partitions. We shall say that $\pi_{2}$ is \emph{finer} than $\pi_{1}$ if $\pi_{2} = \pi_{1} \wedge \pi_{2}$. We can define the refinement of a countable number of partitions $\{ \pi_{j} \}_{j=1}^{\infty}$ as
\begin{equation*}
    \wedge_{j=1}^{\infty} \pi_{j} = \{ \cap_{j = 1}^{\infty} U_{j} : U_{j} \in \pi_{j} \mbox{ for } j \geq 1 \}
\end{equation*}
and thus if we know $\Omega^{\pi_{j}}$ for $j \geq 1$ then the best $\pi$-resolution approximation for $\pi = \wedge_{j=1}^{\infty} \pi_{j}$ is given by $\Omega^{\pi} = \cap_{j=1}^{\infty} \Omega^{\pi_{j}}$. Additionally, we shall say that the partitions $\{ \pi_{j}\}_{j=1}^{\infty}$ \emph{separate points on} $U$ if $\wedge_{j = 1}^{\infty} \pi_{j} = \{ \{u\} : u \in U \}$.

We shall say that $\Omega_{X}$ is \emph{identifiable up to $\pi$-resolution} if its best $\pi$-resolution approximation $\tilde{\Omega}_{X}^{\pi}$ is identifiable. Moreover, we shall say that $\Omega_{X}$ is identifiable \emph{exactly} if its closure in $U$ is identifiable. In essence, the distinction between $\Omega_{X}$ and its closure does not concern us here, nor is it amenable to our method. The following corollary is now almost immediate from Theorem \ref{thm:best_piapprx} and gives sufficient conditions for identifiability of $\Omega_{X}$.
\begin{corollary}\label{cor:identifiability}
    Let $X$ be a Gaussian process on $U$ with a continuous covariance. If $\pi$ is a partition of $U$ such that the correlation operator $\bR_{\pi}$ is invertible, then $\Omega_{X}$ is identifiable up to $\pi$-resolution. 
    
    Furthermore, if there exists a sequence $\{\pi_{j}\}_{j = 1}^{\infty}$ of partitions on $U$ such that (a) the correlation operators $\bR_{\pi_{j}}$ are invertible and (b) the partitions separate points on $U$, then $\Omega_{X}$ is identifiable exactly.
\end{corollary}

The criteria for exact identifiability may appear to be too demanding but they are required only for an \emph{infinite} resolution or \emph{exact} identifiability of $\Omega$. For applications, we can always content ourselves with identifiability up to $\pi$-resolution for a reasonably fine partition $\pi$ which would only require that the correlation operator $\bR_{\pi}$ induced by $\pi$ be invertible. 

\subsection{Relation to Functional Graphical Models}\label{subsec:func_graph_models}

Consider the \emph{functional graphical model} introduced in \textcite{qiao2019} in which the set of vertices consists of $\bX = (X_{1}, \dots, X_{p})$ where every $X_{k}$ is a random real-valued function on an interval $I_{k}$ and there is an edge between $X_{i}$ and $X_{j}$ unless
\begin{equation*}
    \Cov[X_{i}(u), X_{j}(v)| X_{k}(w) \mbox{ for } k \neq i,j \mbox{ and } w \in I_{k}] = 0 \mbox{ for } u \in I_{i} \mbox{ and } v \in I_{j}.
\end{equation*}
If we define 
$$U=\bigsqcup_{j=1}^p I_j= \cup_{j=1}^{p} \{j\} \times I_{j}$$ to be the \emph{disjoint union} of $\{I_1,\ldots,I_p\}$, the vector-valued function $\bX = (X_{1}, \dots, X_{p})$ can be thought of as a single real-valued stochastic process $X = \{X_u: u\in U\}= \{X_{j}(t): 1 \leq j \leq p\,,\, t \in I_j\}$ indexed by both $j$ and $t$. This can be visualized by serially concatenating successive vector components (see Figure \ref{concatenation}) and the set $U$ can thus be thought of as a compact subset of $\bbR$. Recovering the graph of $\bX$ in the functional sense reduces to recovering the graph of $\Omega_{X}$ in the uncountably indexed sense, but only up to a specific $\pi$-resolution, namely where the partition $\pi$ consists of the sets $\{(j,I_j)\}_{j=1}^{p}$. Thus, $\Omega_{\bX} \equiv \tilde{\Omega}_{X}^{\pi}$. This restriction highlights the fact  that functional graphical models concern interactions \emph{between} the random functions $\{X_{j}\}_{j = 1}^{p}$ and not with interactions \emph{within} a random function $X_{j}$ -- the latter requires the notion of coherently resolving an uncountable graph. Furthermore, in the same vein, it shows that functional graphical models can be thought of as special cases of our more general uncountably indexed graphical models.

\begin{figure}
\begin{center}
\includegraphics[scale=0.5, trim={0 5cm 0 1cm}, clip]{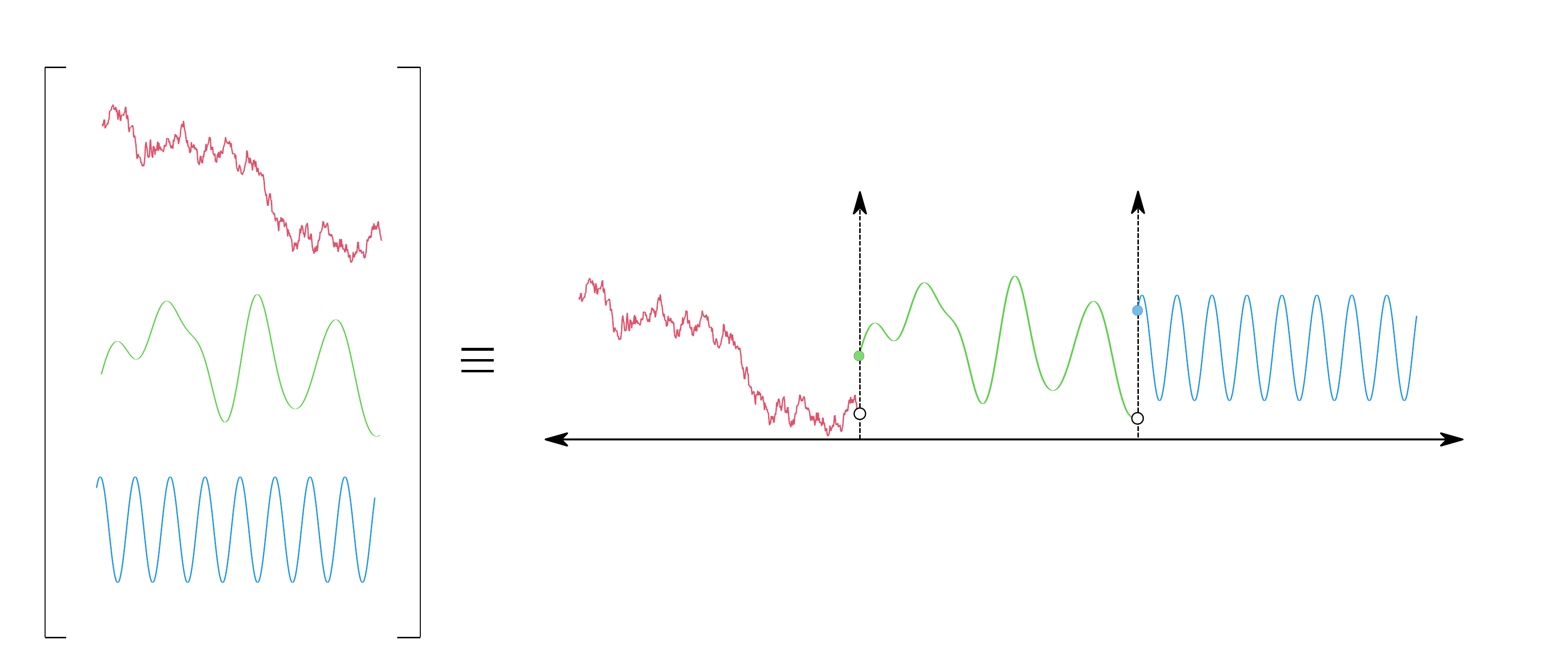}
\bigskip\medskip\caption{A functional graphical model can be seen as a single stochastic process by concatenating successive vector components.}
\label{concatenation}
\end{center}
\end{figure}

\section{Graph Recovery}\label{sec:methodology}

Given a partition $\pi$ of the index set $U \subset \bbR^{d}$, we now present our approach to the problem of recovering the graph $\Omega_{X}$ of a process $X$ given an estimate $\hbK$ of the covariance operator $\bK$.  
This amounts to determining the non-zero entries of the $\pi$-induced precision operator matrix $\bP_{\pi}=\bR_{\pi}^{-1}$.
Evidently, for the last statement to make sense at all, we must assume that $\bR_\pi$ is indeed invertible. Consequently, any consistent estimator of $\bR_\pi$ based on a sample of size $n$ will be eventually invertible as $n$ increases. Whenever the inverse of such an estimator appears, it is implicit that $n$ is sufficiently large.

Since the partition $\pi$ that induces the operators $\bK_{\pi}$, $\bR_{\pi}=\bP_{\pi}^{-1}$ is the same, we shall denote these operators simply as $\bK$, $\bR$, and $\bP=\bR^{-1}$ whenever there is no danger of confusion.  By writing $\bK = \dg \bK + \bK_{0}$, the correlation operator matrix can be expressed as 
$$\bR = \bI + [\dg \bK]^{-1/2}\bK_{0}[\dg \bK]^{-1/2}.$$
Thus the diagonal entries $\bR_{ii}$ of $\bR$ are all equal to identity and we need not burden ourselves with their estimation. Furthermore, since we are effectively trying to invert the compact operator $\dg \bK$, regularization is necessary, and is implemented by adding a ridge of size $\kappa$. Once an estimator of the precision operator matrix is formed, we threshold it entrywise in operator norm to estimate $\Omega_{X}$.

\medskip
\noindent In summary, the plug-in estimation procedure consists of the following two steps:
\begin{enumerate}[label = {\bfseries Step \arabic*.}, leftmargin = 2cm] 
    \item \textbf{Estimation.}    
        Given an estimate $\hbK$ of the covariance operator matrix $\bK$, we estimate the correlation operator matrix as follows:
        \begin{eqnarray*}
              \hbR &:=&  \bI + [\kappa\bI + \dg \hbK]^{-1/2}\hbK_{0}[\kappa\bI + \dg \hbK]^{-1/2}, 
        \end{eqnarray*}
      where $\kappa>0$ is a ridge parameter. If $\dg \hbK$ is not positive, we choose $\kappa$ such that $\kappa + \min_{j} \lambda_{j}(\dg \hbK) \geq 0$, so as to ensure $[\kappa\bI + \dg \hbK]^{-1/2}$ is well defined. 

    \item \textbf{Thresholding.} The estimate $\hat{\Omega}^{\pi}$ of the best $\pi$-resolution approximation $\tilde{\Omega}_{X}^{\pi}$ is calculated as
    \begin{equation*}\label{eqn:hat_omega_pi}
        \hat{\Omega}^{\pi}(\rho) = \cup~ \{U_{i} \times U_{j}: \|\hbP_{ij}\| > \rho\}, 
    \end{equation*}
    where $\hbP = \hbR^{-1}$ and $\rho>0$ is a thresholding parameter. 
\end{enumerate}

There are two tuning parameters involved in the procedure: the ridge $\kappa$, and the threshold $\rho$. Their choice is guided via our asymptotic theory (see Section \ref{sec:estimation}), in relation to the sample size $n$ and the partition size $p$ (the partition $\pi$ will typically be a regular partition into $p$ intervals of equal length). Practical rules for their choice are discussed in Section \ref{sec:implementation}.

Notice that our method of graph recovery relies only upon having a consistent estimator of the covariance. It is therefore very flexible, as it can be applied to any setting where such an estimator is available, including  \emph{serially dependent} samples (such as functional time series). For the same reason, the method is indifferent to the observation regime at hand (complete, dense, sparse) or the presence of measurement error, since all these settings are known to admit consistent estimators of the covariance. It is in this sense that our methodology is \emph{plug-in}.

We also remark that the ridge estimator of the correlation operator matrix in Step 1 is essentially the same as the estimator introduced by \textcite{lisolea2018} in the context of graphical models for random vectors with Hilbertian entries, adapted to our setting. Though the context is somewhat different, there are direct parallels to be drawn, and hence, we occasionally compare to their asymptotic analysis in the next section.

\section{Theoretical Guarantees}

We now turn to establishing performance guarantees for our estimators at a given resolution, in both an asymptotic and a non-asymptotic setting:

\begin{itemize}
\item[] \emph{Asymptotic Guarantees}
 Our general methodology, as described in the last section, is of \emph{plug-in} type: it takes any consistent covariance estimator as input, and outputs the corresponding estimators for the correlation, precision, and graph. This allows the user to employ the covariance estimator most suitable for the observation regime they are working with. In this vein, Section \ref{sec:estimation} correspondingly develops plug-in rates of convergence, which take as input the rate of convergence of the chosen covariance estimator at the given regime, and yield the rate of convergence of the other estimands.

\item[] \emph{Non-Asymptotic Guarantees}. 
Beyond rates of convergence, which are asymptotic in nature, we also consider \emph{finite-sample guarantees} for the various possible observation regimes, in Section \ref{sec:finite_sample}. Finite sample guarantees are by nature specific to the estimator used, which in turn needs to be tailored to the corresponding sampling regime. Hence, we develop new covariance estimators and associated finite sample bounds in each regime, as a first step, and obtain corresponding bounds on for the correlation, precision, and graph.   
\end{itemize}

Finally, in Section \ref{sec:consistency} we make use of our non-asymptotic theory to address the question of recovering the continuum version of the graph $\Omega_X$, i.e. how to successively refine the partition $\pi$ as sample size increases, in order to construct a consistent estimator at infinite resolution.

\subsection{Plug-in Rates of Convergence}\label{sec:estimation}

In this section, we develop asymptotic guarantees for our procedure by deriving rates of convergence in operator norm for the estimators $\hbR$ and $\hbP$ in terms of that of the given covariance estimator $\hbK$, and establishing model selection consistency for a given resolution. As remarked in the previous section, the ridge estimator is of the same form as in \textcite{lisolea2018}, and thus we opt to work with the same regularity conditions. We improve upon their results by proving better and simpler rates of convergence for the estimation of the correlation operator.

Recall that we defined our estimator of the correlation operator matrix as
\begin{equation}\label{eqn:estmr_corr}
    \hbR = \bI + [\kappa_{n}\bI + \dg \hbK]^{-1/2}\hbK_{0}[\kappa_{n}\bI + \dg \hbK]^{-1/2},
\end{equation}
for $\hbK$ our estimator of the covariance operator matrix, and $\kappa_{n}$ the regularization parameter. The error $\hbR - \bR$ of estimating $\bR$ using $\hbR$ can be split into estimation error $\cE = \hbR - \bR_{e}$ (related to variance) and approximation error $\cA = \bR_{e} - \bR$ (related to bias). To control the approximation error, we will require the following regularity condition on $\bR$:
\begin{assumption}\label{asm:regularity}
    For some bounded operator matrix $\Phi_{0}$ with all the diagonal entries zero and $\beta > 0$, we have
    \begin{equation}\label{eqn:regul_condn}
        \bR_{0} = [\dg \bK]^{\beta}\Phi_{0}[\dg \bK]^{\beta}.
    \end{equation}
\end{assumption}
Note that this implies that $\bR_{0}$ is compact. 
The assumption simply ensures that $\bK_{0} = [\dg \bK]^{1/2+\beta}\Phi_{0}[\dg \bK]^{1/2+\beta}$ is linearly well-conditioned for inversion by $[\dg \bK]^{1/2}$. 

\smallskip
\noindent Our first result now relates $\|\hbR-\bR\|$ to $\|\hbK-\bK\|$, $\bK$ and $\|\bR\|$:
\begin{theorem}[Bounding $ \|\hbR-\bR\| $ ]\label{thm:error_corr}
    Under Assumption \ref{asm:regularity}, given any sequences $\kappa_n>0$ and $\delta_n\geq \|\hbK-\bK\| $, we have
    \begin{equation*}
        \|\hbR-\bR\| 
        \leq \|\mathcal{A}\| + \|\mathcal{E}\|\leq 5 \cdot \| \bR \| 
        \cdot \left[(\delta_{n}/\kappa_{n})^{2} + (\delta_{n}/\kappa_{n}) \right]
        + 2 \cdot \kappa_{n}^{\beta \wedge 1} 
        \cdot \|\Phi_{0}\| 
        \cdot \|\bK\|^{2\beta - \beta \wedge 1}.
    \end{equation*}
    The estimator $\hbR$ is consistent so long as the regularization parameter $\kappa_{n}$ is chosen such that $\kappa_{n} \to 0$ and $\delta_{n}/\kappa_{n} \to 0$ as $n \to \infty$. The optimal rate is given by
    \begin{equation*}
        10 \cdot (\|\bR\| \vee \|\Phi_{0}\|\|\bK\|^{2\beta - \beta \wedge 1}) \cdot \delta_{n}^{\frac{\beta \wedge 1}{1 + \beta \wedge 1}}
    \end{equation*}
    and it is achieved for the choice $\kappa_{n} = \delta_{n}^{\frac{1}{1 + \beta \wedge 1}}$.
\end{theorem}



\begin{remark}
    If $\hbK$ is the empirical covariance, $\delta_n$ is $\cO_{\bbP}(n^{-1/2})$, the optimal choice of the regularization parameter is given by $\kappa_{n} \asymp n^{-1/2(\beta \wedge 1 + 1)}$ and we obtain the rate of convergence $\|\hbR-\bR\| = \cO_{\bbP}\big(n^{-\beta \wedge 1/2(\beta \wedge 1 + 1)}\big)$. Note that  when $\beta$ ranges in $(0, 1/2]$, the above rate is strictly better than the rate $\smash{n^{-2\beta/2\beta+5}}$ derived in \textcite{lisolea2018}, and the two rates coincide when $\beta>1/2$.  In addition to slightly improving the rate of convergence for poorly conditioned $\bR_{0}$ corresponding to $\beta < 1/2$, this implies that the apparent phase transition at $\beta = 1/2$ observed in the rates of \textcite{lisolea2018} is an artefact of their analysis. The only transition we observe in the convergence is at $\beta = 1$ as for $\beta > 1$, the rate is same as that for $\beta = 1$ which is $n^{-1/4}$. However, the dependence on $\|\bK\|$ does change, as observed in Theorem \ref{thm:error_corr}.
\end{remark}

\medskip
\noindent
Turning to the precision operator matrix, recall that for $\bP := \bR^{-1}$ to be well defined at all, we need $\bR$ to be strictly positive. The following assumption is only slightly stronger, and is the non-compact counterpart of the familiar assumption that eigenvalues are separated from zero:
\begin{assumption}\label{asm:eigenvalues}
    The spectrum of $\bR_{0}$ satisfies $r = 1 + \inf \sigma(\bR_{0}) > 0$.
\end{assumption}
Under Assumption \ref{asm:eigenvalues}, $\bR$ is strictly positive. The operator $\hbR$ is also strictly positive for all sufficiently large $n$,  by virtue of being consistent. Hence, for all sufficiently large $n$, we may write
\begin{equation}\label{eqn:prec_error}
    \hbP - \bP 
     = \hbR^{-1}\bR\bR^{-1} - \hbR^{-1}\hbR\bR^{-1} 
    = \hbR^{-1}\left[\bR - \hbR\right]\bR^{-1}=\hbP\left[\bR - \hbR\right]\bP.
\end{equation}
Since $\hbP$ is a random quantity, bounding $\|\hbP-\bP\|$ using (\ref{eqn:prec_error}) requires us to find a bound for $\|\hbR-\bR\|$, as well as $\|\hbP\|$. It was shown in \textcite{lisolea2018}, that $\|\hbP\|$ is bounded in probability under Assumption \ref{asm:eigenvalues}. As a result, the convergence rates for $\|\hbR-\bR\|$ also apply to $\|\hbP-\bP\|$. 

\begin{corollary}[Rate of Convergence for $\hbP$ and Consistency]\label{thm:estmn_prec}
    Under the Assumptions \ref{asm:regularity}, \ref{asm:eigenvalues}, and the optimal choice of the regularization parameter $\kappa_{n}$, we have
    \begin{equation*}
        \|\hbP-\bP\| = \|\bP\| (\|\bR\| \vee \|\Phi_{0}\|\|\bK\|^{2\beta - \beta \wedge 1}) \cdot \cO_{\bbP}(\delta_{n}^{\frac{\beta \wedge 1}{1 + \beta \wedge 1}}).
    \end{equation*}
    If we choose $\rho_{n}$ such that $\rho_{n}/\delta_{n}^{\beta \wedge 1/(1 + \beta \wedge 1)} \to 0$, then $\bbP[\hat{\Omega}^{\pi}(\rho_{n}) \neq \hat{\Omega}_{X}^{\pi}] \to 0$ as $n \to \infty$.
\end{corollary}
\noindent In order to derive concentration bounds for the error $\|\hbP-\bP\|$ or quantitative bounds on the familywise error rate $\bbP[\hat{\Omega}^{\pi}(\rho) \neq \hat{\Omega}_{X}^{\pi}]$, we need to be more specific about the observation regime and the choice of covariance estimator $\hbK$. This is done in the next section. \\

\begin{remark}
It is worth mentioning that our assumptions are rather minimal. It is well known in the inverse problem literature that the rate of convergence of the solution of a linear inverse problem can be arbitrarily slow in the absence of regularity such as that offered by Assumption \ref{asm:regularity}. On the other hand, Assumption \ref{asm:eigenvalues} is necessary if we are to connect the empirical covariance with the graph of the process via Theorem \ref{thm:best_piapprx}.  Though it has occasionally been claimed in the literature that $\bR$ always admits a eigenvalue gap (i.e. that $\bR \geq c\bI$ for some $c > 0$), this is not true as the following simple counterexample illustrates: take $\bK = [\bK_{ij}]_{i,j=1}^{2}$ to be given by $\bK_{11} = \bK_{22} = \sum_{j} \lambda_{j} e_{j} \otimes e_{j}$ and $\bK_{12} = \bK_{21} = -\lambda_{1} e_{1} \otimes e_{1}$. Then $\bR = [\bR_{ij}]_{i,j=1}^{2}$ given by $\bR_{11} = \bR_{22} = \bI$ and $\bR_{12} = \bR_{21} = -e_{1} \otimes e_{1}$ is not invertible since $\bR[e_{1} ~e_{1}] = \bzero$. The same counterexample shows that invertibility of $\bR$ itself  cannot be secured by requiring $\mathrm{ Ker }~ \bK_{jj} = \{\bzero \}$.
\end{remark}


\subsection{Finite Sample Guarantees}\label{sec:finite_sample}

In this section, we derive finite sample guarantees for our procedure. We consider independent samples under  complete, regular and sparse observation regimes and derive concentration bounds for appropriate choices of covariance estimators and the corresponding estimators of the correlation and precision operators along with bounds on the familywise error rate.

The complete observations regime is an ideal setting convenient for proving theoretical results, where every sample path is observed noiselessly over the entire domain. For this regime, we consider the empirical covariance estimator. In practice, functional data are observed discretely, usually over dense regular grids but sometimes in the form of sparse measurements between irregularly-sized intervals. In either case, the measurements may be corrupted by noise. These two situations correspond to the regular and sparse observation regimes, respectively. 

To make things concrete, consider $n$ independent samples $\{X_{j}\}_{j=1}^{n}$ of $X$. In the complete observation regime, we are given the observations $\{X_{k}(T): T \in I\}$ for every sample $X_{k}$ on the interval $I=[0,1]$. In contrast, the standard setting in functional data analysis is the discrete observation regime, where we are given the observations $\{(Y_{ki}, T_{ki}): 1 \leq i \leq i_{k}\}$ for every sample $X_{k}$, which satisfy 
\begin{equation*}
    Y_{ki} = X_{k}(T_{ki}) + \xi_{ki}.
\end{equation*}
We assume, in addition, that the noise is independent and sub-Gaussian, that is, $\xi_{ki}$ are independent random variables satisfying $\bbE[\xi_{ki}] = 0$ and $\|\xi_{ki}\|_{\psi_{2}} < \infty$. In the regular observation regime, $i_{k} = M+1$ for some $M \geq 1$ and $T_{ki} = T_{i}$ for $1 \leq k \leq n$, where $\cG = \{T_{i}\}_{i=1}^{M+1}$ forms a regular grid on $I$ with $T_{1} = 0$ and $T_{M+1} = 1$ being the endpoints. In a sparse observation regime, $\{T_{ki}: 1 \leq i \leq i_{k}\}_{k=1}^{n}$ are independent random variables distributed on $I$. 

Covariance estimation for these discrete observation settings is a well-studied problem in functional data analysis and many different methods for estimation have been devised in the literature (e.g. \textcite{fdapace}). 
Although one can readily establish asymptotic rates of convergence for our plug-in procedure when used in conjunction with these estimation methods, as demonstrated in Section \ref{sec:estimation}, deriving finite sample bounds for the same procedure is complicated by the fact that these methods rely on smoothing. 
To circumvent this problem, we propose two simple estimators based on local averaging---one for the regular observation setting and another for a slightly simplified version of the sparse irregular observation setting, where $\{T_{ki}\}$ are uniformly distributed on $I$---and derive finite sample guarantees for these estimators.


\subsubsection{Complete Observations}
\label{sec:complete_obs}

Consider the ideal setting of complete independent observations where every $X_{j}$ is observed over the entire domain $I$. Under certain regularity conditions, the stochastic process $X$ can be thought of as a Gaussian random element in $L^{2}(U, \mu)$. For example, if $X$ is Gaussian with a continuous covariance, it can be thought of as a Gaussian random element in $L^{2}(U, \mu)$. One can then derive the following concentration bound for the empirical covariance operator in the operator norm (see \textcite[Theorem 9]{koltchinskii2017}).
\begin{theorem}\label{thm:large_dev}
    Let $X$ be a sub-Gaussian random element in a Hilbert space,  
    with mean zero and covariance operator $\bK$. Let $X_{1}, \dots, X_{n}$ be i.i.d. replications of $X$. Define the empirical covariance operator $\hbK = \frac{1}{n}\sum_{j=1}^{n} X_{j} \otimes X_{j}$. For every $0 < t \leq \|\bK\|$, 
    \begin{equation*}
        \bbP \{ \|\hbK - \bK\| \geq t\} \leq e^{-cnt^{2}/\|\bK\|^{2}}
    \end{equation*}
    for $n \geq (1 \vee \br(\bK))\|\bK\|^{2}/t^{2}$ where $\br(\bK) = (\bbE \| X \|)^{2} / \|\bK\|$ and $c$ is a universal constant. 
\end{theorem}
\smallskip
Using our earlier results, we can now derive concentration bounds for $\|\hbR-\bR\|$ and$\|\hbP-\bP\|$ and a tail bound for $\hbP$, which will eventually enable us to prove the consistency of our graph recovery procedure:
\begin{theorem}[Concentration and Tail Bounds]\label{thm:conc_ineq}
    Let $X$ be a stochastic process on the set $U$ corresponding to a sub-Gaussian random element in the Hilbert space $L^{2}(U, \mu)$ with the covariance operator $\bK$. Let $c_{K}$ be the universal constant $c$ appearing in Theorem \ref{thm:large_dev}, $\rho_{K} = \|\bK\|$, $n_{K} = [1 \vee \br(\bK)] \|\bK\|^{2}$,
    \begin{equation*}
        M_{R} = 10 \cdot \left[\|\bR\| \vee \|\Phi_{0}\|\|\bK\|^{2\beta - \beta \wedge 1}\right] \mbox{ and } r = \inf_{j} \left[1 + \lambda_{j}(\bR_{0})\right] = \|\bP\|^{-1}.
    \end{equation*}  
    Define $c_{R}^{} = c_{K}^{}M_{R}^{2 + 2/\beta \wedge 1}$, $\rho_{R}^{} = M_{R}^{}\rho_{K}^{\beta \wedge 1/(\beta \wedge 1 + 1)}$, $n_{R}^{} = n_{K}^{} M_{R}^{2 + 2/\beta \wedge 1}$ and $c_{P} = c_{R}(r^{2}/2)^{2 + 2/\beta \wedge 1}$.
    \begin{enumerate}
        \item Under Assumption \ref{asm:regularity}, we have 
        \begin{equation}\label{eqn:conc_ineq_corr}
            \bbP[\|\hbR-\bR\| > \rho] \leq \exp \left\lbrace-c_{R}n\rho^{2+2/\beta \wedge 1}\right\rbrace
        \end{equation}
        for $0 < \rho < \rho_{R}$ and $n > n_{R}/\rho^{2+2/\beta \wedge 1}$.
        \item Under Assumptions \ref{asm:regularity} and \ref{asm:eigenvalues}, we have 
        \begin{equation}\label{eqn:tail_bound}
            \bbP[\|\hbP\| > (r - \rho)^{-1}] \leq \exp \left\lbrace-c_{R}n\rho^{2+2/\beta \wedge 1}\right\rbrace
        \end{equation}
        for $0 < \rho < r\wedge\rho_{R}$ and $n > n_{R}/\rho^{2+2/\beta \wedge 1}$.        

        \item Under Assumptions \ref{asm:regularity} and \ref{asm:eigenvalues}, we have
        \begin{equation}\label{eqn:conc_ineq_prec}
            \bbP[\|\hbP-\bP\| > \rho] \leq 2\cdot \exp \left\lbrace-c_{P}n\rho^{2+2/\beta \wedge 1}\right\rbrace
        \end{equation}
        for $0 < \rho < (r/2)\wedge\rho_{R}$ and $n > n_{R}/\rho^{2+2/\beta \wedge 1}$.
       \end{enumerate}
\end{theorem}
Note that the parameters $\rho_{K}$ and $n_{K}$ depend only on the covariance kernel $K$, whereas the parameters $c_{R}$, $c_{P}$, $\rho_{R}$, $M_{R}$, $r$ and $n_{R}$ depend only on $K$ and $\pi$.
\smallskip
We can now have the tools to establish sufficient conditions for the estimator $\smash{\hat{\Omega}^{\pi}}$ of $\smash{\tilde{\Omega}_{X}^{\pi}}$ to be consistent.
\begin{theorem}[Consistency at Given Resolution]\label{thm:consistency}
    Let $X$ be a Gaussian process on $U$ with continuous covariance kernel $K$, corresponding to a (Gaussian) random element in the Hilbert space $L^{2}(U, \mu)$. Let $\{X_{k}\}_{k = 1}^{n}$ be $n$ independent copies of $X$ and $\pi$ be a partition on $U$. Under Assmptions \ref{asm:regularity} and \ref{asm:eigenvalues}, we have for $0 < \rho < \frac{1}{2}r \wedge \rho_{R} \wedge \rho_{P}$ and $n > n_{R}/\rho^{2 + 2/\beta \wedge 1}$, 
    \begin{equation*}
        \bbP[\hat{\Omega}^{\pi} \neq \tilde{\Omega}_{X}^{\pi}] \leq 2p^{2} \cdot \exp \left[ -c_{P} n \rho^{2+2/(\beta\wedge 1)} \right] \to 0 \mbox{ as } n \to \infty
    \end{equation*}    
    where $p$ is the cardinality of $\pi$, $\rho_{P} = \frac{1}{2}\min \{\|\bP_{ij}\|: \bP_{ij} \neq \bzero \}$, and the parameters $\rho_{R}, n_{R}$ and $c_{P}$ are as in Theorem \ref{thm:conc_ineq} and depend only on $K$ and $\pi$. 
\end{theorem}
Alternatively, for the probability $\bbP[\hat{\Omega}^{\pi} \neq \tilde{\Omega}_{X}^{\pi}]$ to be less than some $\alpha \in (0, 1)$, we need the sample size $n$ to satisfy
\begin{equation}
    \label{eqn:n-bound}
    n > \frac{1}{c_{P}}[\tfrac{1}{2}r \wedge \rho_{R} \wedge \rho_{P}]^{-2-2/\beta \wedge 1} \log \left[\frac{2p^{2}}{\alpha}\right].
\end{equation}
Notice that even if the thresholding parameter is chosen as a function of the sample size, as in $\rho \equiv \rho(n)$, then the estimator is consistent so long as $n \rho_{n}^{2+2/\beta\wedge 1} \to \infty$ as $n \to \infty$. Regardless, Theorem \ref{thm:consistency} guarantees exact recovery of $\tilde{\Omega}_{X}^{\pi}$ with high probability, so long as the thresholding parameter $\rho$ is fixed to be small enough and the sample size $n$ is large enough. It is in contrast to the asymptotic results of \textcite{lisolea2018} in which the thresholding parameter needs to decrease as the sample size increases for consistent recovery of the graph and we do not know how quickly $\bbP[\hat{\Omega}^{\pi} \neq \tilde{\Omega}_{X}^{\pi}]$ converges to $0$ in terms of the sample size. 

A natural question now is: at how fine a resolution $p$ can we estimate the graph $\Omega_{X}$ reliably from a given sample of size $n$? An upper bound for $p$ for given $n$ is implicit in the inequality (\ref{eqn:n-bound}), considering the fact that $c_{P}$, $r$, $\rho_{R}$ and $\rho_{P}$ are all functions of $p$, or rather, the partition $\pi$, along with the covariance $K$. Due to the complicated nature of this dependence, we are unable to derive a closed form expression for an upper bound on $p$. 

\subsubsection{Regular Observations}
\label{sec:regular_obs}

Divide $I$ into the intervals $I_{l} = [T_{l}, T_{l+1})$ for $1 \leq l < M$ and $I_{M} = [T_{M}, T_{M+1}]$. Consider the matrix $\hat{F} = [\hat{F}_{uv}]_{u,v=1}^{M+1}$, given by 
$$\hat{F}_{uv} = \frac{1}{n}\sum_{k=1}^{n} Y_{ku}Y_{kv} \qquad \qquad \mbox{ for } 1 \leq u,v \leq M+1,$$
which can be regarded as a naive estimator of the covariance $K$ at the grid points $\cG \times \cG$. Define $\hat{K}_{\mathrm{regular}}: I \times I \to \bbR$ as 
\begin{equation*}
    \hat{K}_{\mathrm{regular}}(s, t) = \frac{\sum_{i,j=0,1} (1-\delta_{u+i,v+j})\hat{F}_{u+i,v+j}}{\sum_{i,j=0,1} (1 - \delta_{u+i,v+j})} \qquad \mbox{ for } (s, t) \in I_{u} \times I_{v} \mbox{ and } 1 \leq u,v \leq M.
\end{equation*}
In other words, $\hat{K}_{\mathrm{regular}}(s, t)$ is the average of those entries among $\hat{F}_{uv}$, $\hat{F}_{u+1,v}$, $\hat{F}_{u,v+1}$ and $\hat{F}_{u+1,v+1}$, which are off-diagonal in $\hat{F}$.

\begin{theorem}\label{thm:reg-obs-estmtr}
    Suppose $X = \{X_{u}\}_{u \in I}$ is a zero-mean second-order stochastic process with a continuous covariance $K: I \times I \to \bbR$, differentiable outside the diagonal, which satisfies $\sD = \sup_{u \neq v} \left|\frac{\partial}{\partial u}K(u, v)\right| < \infty$. Assume furthermore that $\kappa_{Y} = \sup_{u} \|X(u)\|_{\psi_{2}} + \max_{k,i}\|\xi_{ki}\|_{\psi_{2}} < \infty$.
    
    The estimator $\hat{\bK}_{\mathrm{regular}}$ defined as the integral operator with the kernel $\hat{K}_{\mathrm{regular}}$ satisfies for every $0 \leq t \leq \kappa_{Y}^{2}$,
    \begin{equation*}
        \bbP\{ \|\hat{\bK}_{\mathrm{regular}} - \bK\| \geq t + \tfrac{2}{M}\sD\} \leq 4M^{2} \exp\left[-\frac{cnt^{2}}{\kappa_{Y}^{4}}\right],
    \end{equation*}
    where $c$ is an absolute constant.
\end{theorem}
Note that the fixed resolution $M$ of the grid imposes a strict limit of $2\sD/M$ on how well $\hat{\bK}_{\mathrm{regular}}$ can estimate $\bK$ regardless of the number of samples $n$. For large enough $M$, this can still be adequate for recovering the graph $\tilde{\Omega}_{X}^{\pi}$ for a particular partition $\pi$.

\begin{corollary}\label{corollary-reg-1}
If $\hbP$ and $\hat{\Omega}^{\pi}$ are calculated for $\hbK = \hbK_{\mathrm{regular}}$, then we have the bounds
\begin{alignat*}{3}
    \bbP\{\|\hbP - \bP\| \geq \rho\} &\leq
    g(\rho/M_{R}) \quad &&\mbox{for } \rho &&\in (0, \rho_{0}), \mbox{ and }\\ 
    \bbP[\hat{\Omega}^{\pi} \neq \tilde{\Omega}_{X}^{\pi}] &\leq p^{2}g(\rho/M_{R}) \quad &&\mbox{for } \rho &&\in (0, \rho_{0} \wedge \rho_{P}),
\end{alignat*}
where $\rho_{0} = \tfrac{r}{2} \wedge M_{R}\left(\tfrac{2\sD}{M} + \kappa_{Y}^{2}\right)^{\beta \wedge 1/(1 + \beta \wedge 1)}$, $\rho_{P} = \tfrac{1}{2}\min \{\|\bP_{ij}\|: \bP_{ij} \neq \bzero \}$ and $g: (0, \infty) \to \bbR$ is given by
\begin{equation*}
    g(\rho) = 8M^{2} \exp \left\lbrace - \frac{cn}{\kappa_{Y}^{4}}\left[\rho^{1+1/\beta \wedge 1} - \frac{2\sD}{M}\right]_{+}^{2}\right\rbrace.
\end{equation*}
where $[y]_{+}$ denotes the positive part of $y$.
\end{corollary}
Interestingly, the above bounds offer no control on large deviations of $\hbP$ or $\hat{\Omega}^{\pi}$ unless $M$ is large enough. The following corollary provides a lower bound for $M$ which is sufficient for the expected rate of convergence of $\hbP$ and the finite sample bound on familywise error rate to hold.
\begin{corollary}\label{corollary-reg-2}
For a fixed partition $\pi$, if 
\begin{equation*}
    M > 2\sD \left[\frac{M_{R}}{\tfrac{r}{2} \wedge \rho_{P}}\right]^{1+1/\beta \wedge 1}
\end{equation*}
then $\|\hbP - \bP\| = \cO_{\bbP}(n^{-\beta \wedge 1/2(1+\beta \wedge 1)})$. Additionally, if for some $\alpha \in (0, 1)$, we have
\begin{equation*}
    \rho = M_{R}\left[\frac{2\sD}{M} + \sqrt{\frac{\kappa_{Y}^{4}}{cn} \log \left(\frac{8M^{2}p^{2}}{\alpha}\right)} \right]^{\beta \wedge 1/(1+\beta \wedge 1)} \in (0, \rho_{0} \wedge \rho_{P})
\end{equation*}
then $\bbP[\hat{\Omega}^{\pi}(\rho) \neq \tilde{\Omega}_{X}^{\pi}] \leq \alpha$.
\end{corollary}

\subsubsection{Sparse Observations}
\label{sec:sparse_obs}
Assume that $\{T_{ki}: 1 \leq i \leq i_{k}\}_{k=1}^{n}$ be independent random variables which are uniformly distributed on $I$. As before, we divide $I$ into $M$ contiguous intervals $\{I_{p}\}_{p=1}^{M}$ where $M$ is regarded as a tuning parameter. Consider the estimator $\hat{K}(s, t) = M^{2}\hat{K}_{pq}$ for $(s, t) \in I_{p} \times I_{q}$, where 
\begin{equation*}
    \hat{K}_{pq} = \frac{1}{n}\sum_{k=1}^{n}\frac{1}{i_{k}(i_{k}-1)}\sum_{i, j=1}^{i_{k}} (1-\delta_{ij}) Y_{ki}Y_{kj} \mathbf{1}_{\{T_{ki} \in I_{p}\}}\mathbf{1}_{\{T_{kj} \in I_{q}\}}.
\end{equation*}

\begin{theorem}\label{thm:sparse-obs-estmtr}
    Suppose $X = \{X_{u}\}_{u \in I}$ is a zero-mean second-order stochastic process with a continuous covariance $K: I \times I \to \bbR$, differentiable outside the diagonal, which satisfies $\sD = \sup_{u \neq v} \left|\frac{\partial}{\partial u}K(u, v)\right| < \infty$. Assume furthermore that $\kappa_{Y} = \sup_{u} \|X(u)\|_{\psi_{2}} + \max_{k,i}\|\xi_{ki}\|_{\psi_{2}} < \infty$.
    
    The estimator $\hat{\bK}_{\mathrm{sparse}}$ defined as the integral operator with the kernel $\hat{K}_{\mathrm{sparse}}$ satisfies for every $0 \leq t \leq \kappa_{Y}^{2}$,
    \begin{equation*}
        \bbP\{ \|\hat{\bK}_{\mathrm{sparse}} - \bK\| \geq t + \tfrac{1}{M}\sD\} \leq 2M^{2}\exp \left[- \frac{cnt^{2}}{M^{4}\kappa_{Y}^{4}} \right],
    \end{equation*}
    where $c$ is an absolute constant and $M$ is the number of partitions. For $M \in (2\sD/t, 4\sD/t)$, we have for $0 \leq t \leq 2\kappa_{Y}^{2}$,
    \begin{equation*}
        \bbP\{ \|\hat{\bK}_{\mathrm{sparse}} - \bK\| \geq t\} \leq \frac{32\sD^{2}}{t^{2}} \exp \left[ - \frac{cnt^{6}}{256\sD^{4}\kappa_{Y}^{4}} \right].
    \end{equation*}
\end{theorem}

\begin{remark}
Translating the finite sample bound into a rate, we see that the rate under sparsity is considerably slower than under complete observation: $\|\hat{\bK}_{\mathrm{sparse}} - \bK\| = \cO_{\bbP}(n^{-1/6+\epsilon})$ for every $\epsilon > 0$ as opposed to $\|\hat{\bK} - \bK\| = \cO_{\bbP}(n^{-1/2})$. If instead of using $\hbK_{\mathrm{sparse}}$ we assume additional smoothness and use the PACE covariance estimator $\hbK_{\mathrm{pace}}$, we can attain asymptotic rates between $O(n^{-1/6})$ and $O(n^{-1/4})$, but we know of no finite sample bounds for $\hbK_{\mathrm{pace}}$. Such bounds are instrumental for establishing model selection consistency in the infinite resolution limit, see Section \ref{sec:consistency}. Also see Section \ref{sec:sparse_simulations} for a numerical comparison of $\hbK_{\mathrm{pace}}$ and $\hbK_{\mathrm{sparse}}$.
\end{remark}

\begin{corollary}\label{corollary-sparse-1}
If $\hbP$ and $\hat{\Omega}^{\pi}$ are calculated for $\hbK = \hbK_{\mathrm{sparse}}$, then we have the bounds
\begin{alignat*}{3}
    \bbP\{\|\hbP - \bP\| \geq \rho\} 
    &\leq h(\rho/M_{R}) \quad &&\mbox{ for } \rho &&\in (0, \rho_{0}), \mbox{ and } \\
    \bbP[\hat{\Omega}^{\pi} \neq \tilde{\Omega}_{X}^{\pi}] &\leq p^{2}h(\rho/M_{R}) \quad &&\mbox{ for } \rho &&\in (0, \rho_{0} \wedge \rho_{P}),
\end{alignat*}
where $\rho_{0} = \tfrac{r}{2} \wedge 2M_{R}\left(\tfrac{2\sD}{M} + \kappa_{Y}^{2}\right)^{\beta \wedge 1/(1+\beta \wedge 1)}$, $\rho_{P} = \frac{1}{2}\min \{\|\bP_{ij}\|: \bP_{ij} \neq \bzero \}$ and $h: (0, \infty) \to \bbR$ is given by
\begin{equation*}
    h(\rho) = 64\sD^{2} \exp\left\lbrace 
    -\frac{cn}{256\sD^{4}\kappa_{Y}^{4}} \rho^{6+6/\beta \wedge 1} -2 \left(1+\frac{1}{\beta \wedge 1}\right) \log \rho
    \right\rbrace.
\end{equation*}  
Consequently, $\|\hbP - \bP\| = \cO_{\bbP}(n^{\epsilon-(\beta \wedge 1)/6(1+\beta \wedge 1)})$ for every $\epsilon > 0$.
\end{corollary}
Although, we are unable to derive a closed form expression for the threshold $\rho$ required for maintaining a given familywise error rate $\alpha$, it is clear from the bound on $\|\hbP - \bP\|$ that $\rho$ can not decrease faster than $n^{-(\beta \wedge 1)/6(1+\beta \wedge 1)}$ as a function of the sample size $n$.


\subsection{Model Selection Consistency at Infinite Resolution}\label{sec:consistency}

So far we have considered the question of whether we can consistently recover the graph $\Omega_{X}$ at a given resolution $\pi$ in the form of $\tilde{\Omega}^{\pi}_{X}$. A related question is whether we can recover the graph $\Omega_{X}$ in the limit as $n \to \infty$. Put differently, how should we refine our partition $\pi$ as the sample size $n$ increases in order to construct a consistent estimator for the graph $\Omega_{X}$ itself. We can give a concrete answer to this question using the finite sample guarantees we have derived in this section. For simplicity, let's consider the complete observation regime from Section \ref{sec:complete_obs}. Similar results can be derived for other regimes using the appropriate finite sample guarantees.

Let $\{\pi_{j}\}_{j=1}^{\infty}$ be partitions on $U$ which separate points and $\{\alpha_{j}\}_{j=1}^{\infty} \subset \bbR$ be such that $\sum_{j=1}^{\infty} \alpha_{j} < \infty$. For every $j \geq 1$, let $\hat{\Omega}_{j}$ denote the estimator $\hat{\Omega}^{\pi_{j}}$ constructed only using the sample $\{ X_{k} \}_{k = 1}^{n_{j}}$ with an admissible value of the threshold $\rho_{j}$ according to Theorem \ref{thm:consistency} where the parameter $n_{j}$ has been chosen to be the smallest $n$ such that
\begin{equation}\label{eqn:stpd_condn}
    n > \frac{1}{c_{P_{j}}}[\tfrac{1}{2}r_{j} \wedge \rho_{R_{j}} \wedge \rho_{P_{j}}]^{-2-2/\beta_{j} \wedge 1} \log \left[\frac{2p_{j}^{2}}{\alpha_{j}}\right].
\end{equation}
Here, $p_{j}$ is the cardinality of $\pi_{j}$ while $\beta_{j}$, $r_{j}$, $\rho_{R_{j}}$, $\rho_{P_{j}}$ and $c_{P_{j}}$ are the parameters $\beta$, $r$, $\rho_{R}$, $\rho_{P}$ and $c_{P}$ corresponding to the correlation operator $\bR = \bR_{\pi_{j}}$. Essentially, we are saying that for larger sample sizes $n > n_{j}$ we can recover $\Omega_{X}$ to higher resolution $p_{j}$ with an eventually decreasing probability of failure $\alpha_{j}$ since $\alpha_{j} \to 0$ as $j \to \infty$. We now have the following result.
\begin{theorem}[Consistency under Resolution Refinement]\label{thm:consistency2}
Let $X$ be a Gaussian process on a compact set $U \subset \bbR$ with the continuous covariance $K$ corresponding to a (Gaussian) random element in the Hilbert space $L^{2}(U, \mu)$. Let $\{X_{k}\}_{k = 1}^{n}$ be independent copies of $X$ and $\{\pi_{j}\}_{j=1}^{\infty}$ be partitions on $U$ which separate points such that: (a) $\pi_{j+1}$ is finer than $\pi_{j}$ for every $j \geq 1$ and (b) the associated correlation operators $\bR_{\pi_{j}}$ satisfy Assumptions \ref{asm:regularity} and \ref{asm:eigenvalues}. Then for $\hat{\Omega}_{j}$ as defined before, 
\begin{equation*}
    \lim_{n \to \infty} ~\hat{\Omega}_{\max \{j: n_{j} < n \}} = \Omega_{X} \mbox{ almost surely.}
\end{equation*}
In other words, $\Omega_{\max \{j: n_{j} < n \}}$ is a consistent estimator of $\Omega_{X}$.
\end{theorem}

\section{Implementation}\label{sec:implementation}
To implement the procedure in practice, one needs to specify the partition $\pi$, the ridge $\kappa$, and the threshold $\rho$. We now discuss this specification in a finite-sample context. 
 
\begin{itemize}
 \item \textbf{Partition}. The choice of partition $\pi$ is primarily depends on the graph resolution desired and on which regions of the domain hold more or less interest. It is also affected by considerations of sample size and grid resolution for discrete observations. One should require the partition to be finer over regions of greater interest so as to recover the graph at a higher resolution there. On the other hand, if all regions of the domain are of equal interest, one should work with a regular partition consisting of contiguous subsets of roughly equal volume (eg. contiguous subintervals of equal length for an interval). 
 One can 
 adopt a scale-space approach and consider multiple values of $p$, searching for persistent zero patterns in the associated precision operator matrices.

 
 \item \textbf{Ridge}. We choose $\kappa$ using $k$-fold cross-validation. We construct a partition $S$ of the samples $\{X_{k}\}_{k=1}^{n}$ into $|S|$ subsets of roughly equal size. We calculate the estimators $\hbK_{s}$ and $\hbK_{-s}$ of $\bK$ using samples in and not in $s \in S$, respectively, and choose $\kappa$ according to
\begin{equation}\label{eqn:ridge_choice}
    \kappa = \argmin_{\lambda \in \Lambda} \left[ \frac{1}{|S|}\sum_{s \in S}\| \dg\hbK_{s}  - (\dg\hbK_{s})(\lambda\bI + \dg\hbK_{-s})^{-1}(\dg\hbK_{s}) \| \right]
\end{equation}
where $\Lambda$ is a suitably chosen range of ridge values. Alternatively, one can employ a generalized cross-validation approach as in \textcite{lisolea2018}.

 


\item \textbf{Threshold}. We plot the density function of the values $\{\log_{10} \|\hbP_{ij}\|: 1 \leq i,j \leq p\}$ and treat the local minima and ``elbows" of the curve as corresponding to candidates for the threshold as illustrated in Figure \ref{fig:separate_comps}. Intuitively speaking, the threshold separates the values into two components corresponding to zero and nonzero $\hbP_{ij}$ and represents the decision boundary for the purpose of classifying $\bP_{ij}$ into one of these components. Alternatively, one can use the stability selection approach of \textcite{meinshausen2010} which is often used for model selection in LASSO and graphical LASSO.

\begin{figure}[h]
    \centering
    \begin{subfigure}{.48\textwidth}
        \centering
        \includegraphics[width=\linewidth, trim={0 0 2.5cm 0}, clip]{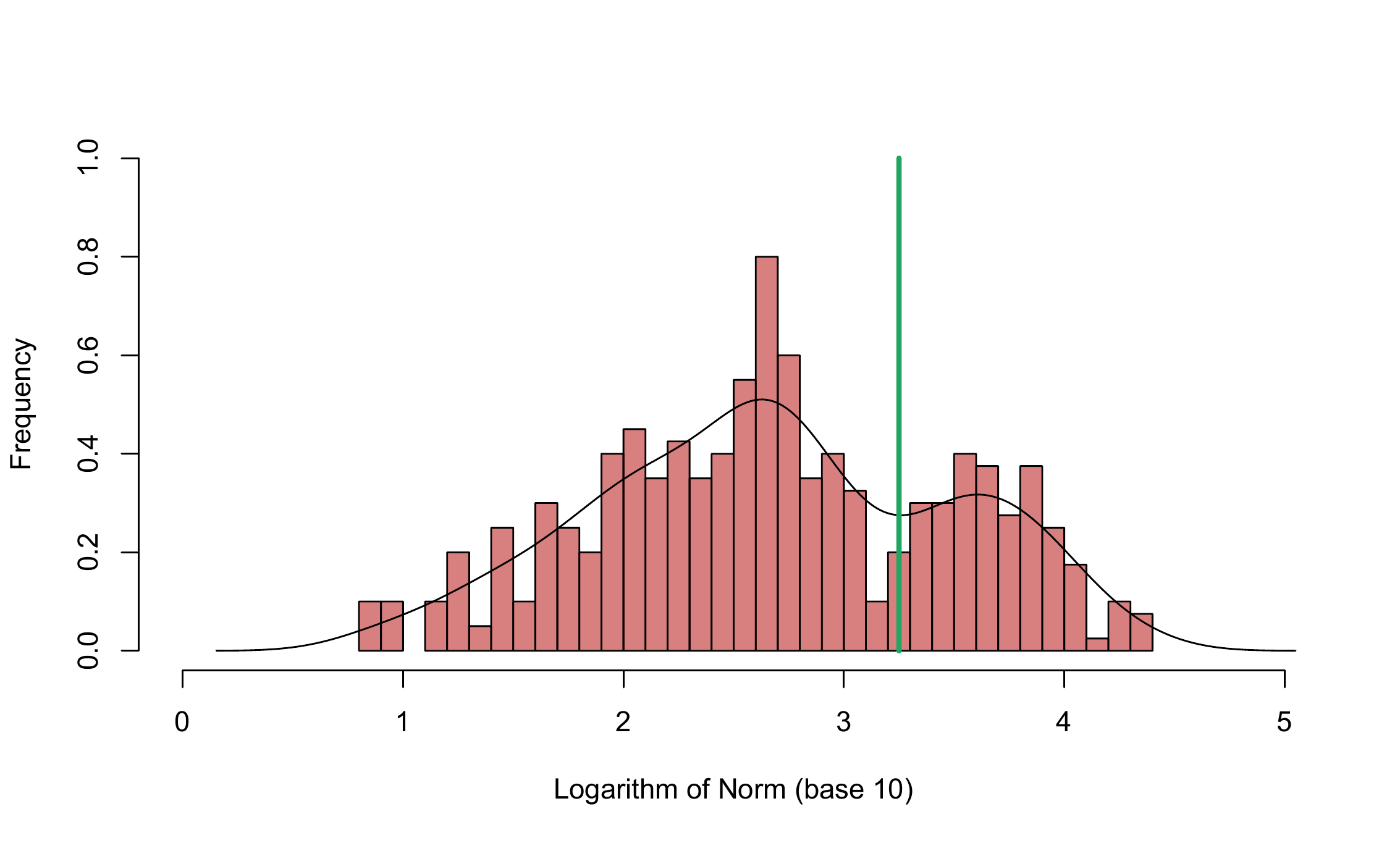}
        \caption{}
        \label{fig:separate_comp}
    \end{subfigure}
    \begin{subfigure}{.48\textwidth}
        \centering
        \includegraphics[width=\linewidth, trim={0 0 2.5cm 0}, clip]{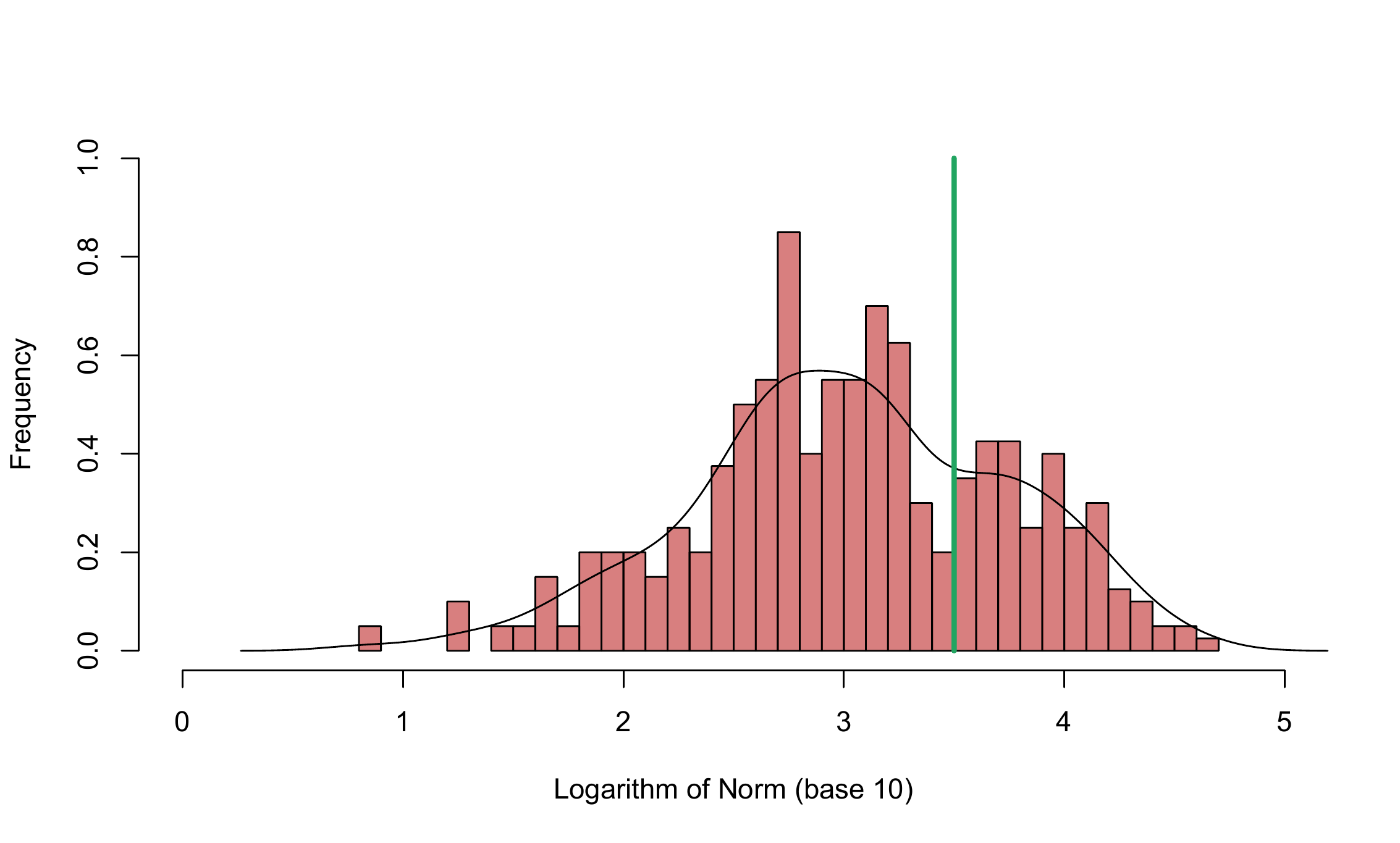}
        \caption{}
        \label{fig:separate_comp2}
    \end{subfigure}\\[0.5cm]
    \caption{The local minima (a) and elbows (b) of the kernel density estimator of $\{\log_{10} \|\hbP_{ij}\|: 1 \leq i,j \leq p\}$ serve as good candidates for the threshold $\rho$.}
    \label{fig:separate_comps}
\end{figure}




According to our theoretical results, $\rho$ need not decrease with $n$, but rather any sufficiently small value will suffice. 
Naturally, as $n \to \infty$, the $ij$-entries of $\hbP_{ij}$ corresponding to $\bP_{ij} = \bzero$ converge to zero while those for which $\bP_{ij} \neq \bzero$ converge to $\bP_{ij}$.

\end{itemize}
  

To implement the operations involving operator matrices, we discretize the covariance operator matrix $\bK = \bK_{\pi}$ on a uniform grid $\{u_{i}\}_{i=1}^{R} \subset U$ as the matrix $\mathsf{K} = [K(u_{i}, u_{j})]_{i,j=1}^{R}$ and its diagonal counterpart $\dg \bK$ as $\mathsf{D} = [\mathsf{D}_{ij}]_{i,j=1}^{R}$ where $\mathsf{D}_{ij} = K(u_{i}, u_{j})$ if $u_{i}, u_{j} \in U_{k}$ for some $1 \leq k \leq p$ and $0$ otherwise. The correlation operator $\bR$ is then computed as $\mathsf{R} = \mathsf{I} + \mathsf{R}_{0}$ where
\begin{equation*}
    \mathsf{R}_{0} = [\kappa_{n} \mathsf{I} + \tfrac{1}{R}\mathsf{D}]^{-1/2}[\mathsf{K} - \mathsf{D}][\kappa_{n} \mathsf{I} + \tfrac{1}{R}\mathsf{D}]^{-1/2} \qquad \mbox{ and } \qquad \mathsf{I} = [\delta_{ij}]_{i,j=1}^{R}.
\end{equation*} The precision operator $\bP$ is computed as $\mathsf{P} = \mathsf{I} - [\mathsf{I} + \tfrac{1}{R}\mathsf{R}_{0}]^{-1}[\tfrac{1}{R}\mathsf{R}_{0}]$, which is mathematically identical to $\mathsf{P} = \mathsf{R}^{-1}$ but numerically more stable. 

\section{Numerical Simulations} \label{sec:numerical_simulations}

In this section, we study the numerical performance of our estimation procedures at different resolutions, sample sizes, and levels of noise, for each of three possible observation regime (completely, densely, or sparsely observed curves). We consider the resolutions $p = 20, 30$ and $40$ and choose the corresponding partitions $\pi$ to be uniform over the unit interval, i.e. given by the collection of subintervals $U_{j} = [j/p, (j+1)/p)$ for $0 \leq j \leq p-1$ and $U_{p} = [(p-1)/p, 1]$.

We consider five covariances over the unit interval $U = [0, 1]$: (a) Gaussian kernel $K_{1}(u, v) = e^{-(u-v)^{2}}$, (b) Brownian motion covariance $K_{2}(u, v) = u \wedge v$, (c) integrated Brownian motion covariance $K_{3}(u, v) = \bbE[X_{u}X_{v}]$ where $X_{t} = \int_{0 \vee (t - 0.5)}^{t \wedge 1} W_{s} ~ds$, (d) P\'{o}lya covariance $K_{4}(u, v) = 0.8 \Delta_{0.7}(u-v) + 0.2\Delta_{0.8}(u-v)$ where $\Delta_{w}(t) = (1 - |t/w|)\mathbf{1}_{\{1 - |t/w| \geq 0\}}$ and (e) interpolated random vector covariance $$K_{5}(u, v) = (1 - u^{\prime})(1 - v^{\prime})\alpha^{|i - j|} + (1 - u^{\prime})v^{\prime}\alpha^{|i - j - 1|} + u^{\prime}(1 - v^{\prime})\alpha^{|i + 1 - j|} + u^{\prime}v^{\prime}\alpha^{|i - j|}$$ for $i = 1 + \lfloor u q \rfloor$, $j = 1 + \lfloor v q \rfloor$, $u^{\prime} = u - i/q$ and $v^{\prime} = v - j/q$. 

Of these, $K_{1}$ is analytic and $K_{2}$ is the covariance of a Markov process, which implies that $\Omega_{X} = \{(u, v): u = v\}$ for both these cases, with their graphs being infinitesimal strips along the diagonal. The graphs of $K_{3}$ and $K_{4}$ could not be ascertained theoretically and were computed numerically to be given by $\Omega_{X} \approx \{(u, v): |u - v| = 0 \mbox{ or } 0.5\}$ and $\Omega_{X} \approx \{(u, v): |u - v| = 0 \mbox{ or } 0.8 \} \cup \{0, 0.2, 0.8, 1\}^{2}$, respectively, by applying our method to the population versions of the covariances. The kernel $K_{3}$ serves to illustrate the effect of applying a linear filter on a Markov process on its graph, while $K_{4}$ represents an important family of covariances constructed from P\'{o}lya-type positive-definite functions, which in this case is $\Delta_{w}$. The kernel $K_{5}$ corresponds to the graph $\Omega_{X} = \{(u, v): |\lfloor qu \rfloor - \lfloor qv \rfloor| \leq 1\}$ and is the covariance of the process $X_{t}$ constructed by linearly interpolating a Gaussian random vector $\bX = (X_{1}, \dots, X_{q+1}) \in \bbR^{q+1}$ with mean zero and the covariance given by the Kac-Murdock-Szeg\"{o} matrix $\bC = [\alpha^{|i-j|}]_{i, j = 1}^{q+1}$ with the parameters $\alpha = 0.3$ and $q = 10$. It helps us verify that our method for graph recovery in continuous time conforms to our intuition for graph recovery in finite dimensions. 
Figure \ref{fig:plots_of_covariances} displays the level plots of the covariances along with the level plots of the matrix $P = [\bP_{ij}]_{i,j = 1}^{p}$ (which contains the norms of the entries of the precision matrix $\bP$) and the graphs $\tilde{\Omega}_{X}^{\pi}$.

\begin{figure}
    \centering
    \begin{subfigure}[b]{0.82\textwidth}
       \includegraphics[width=1\linewidth, trim=1.5cm 1cm 0 0, clip]{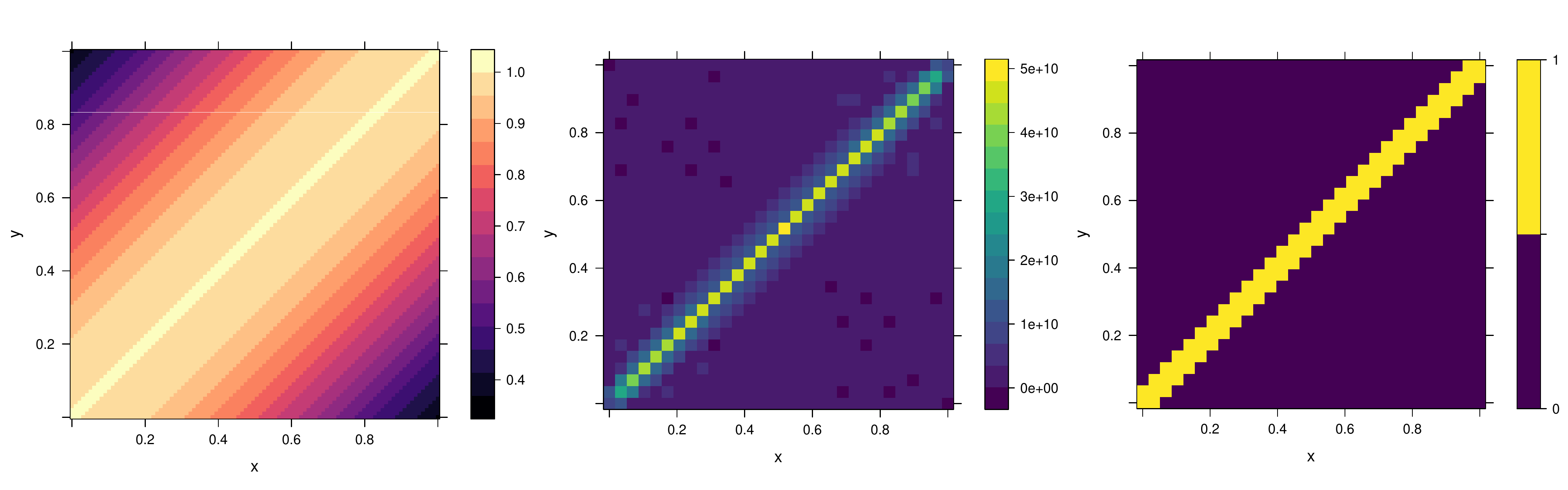}
       \caption{$K_{1}$ (Gaussian kernel)}
       \label{fig:g1} 
    \end{subfigure}    
    \begin{subfigure}[b]{0.82\textwidth}
       \includegraphics[width=1\linewidth, trim=1.5cm 1.2cm 0 0, clip]{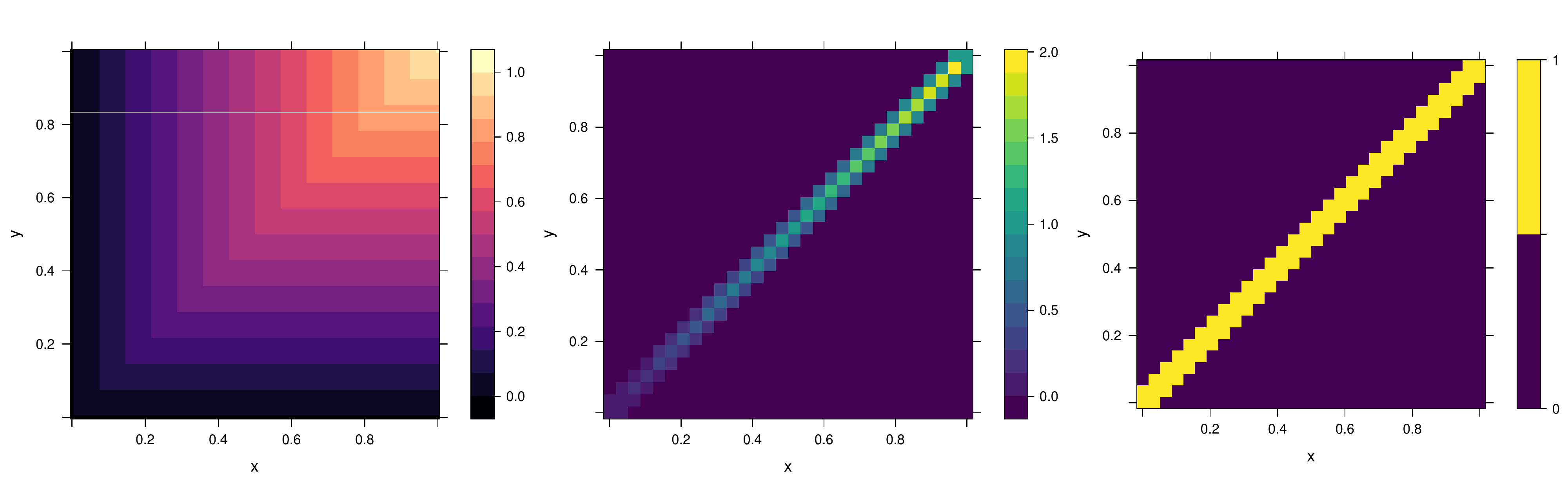}
       \caption{$K_{2}$ (Brownian motion)}
       \label{fig:g2}
    \end{subfigure}
    \begin{subfigure}[b]{0.82\textwidth}
        \includegraphics[width=1\linewidth, trim=1.5cm 1.2cm 0 0, clip]{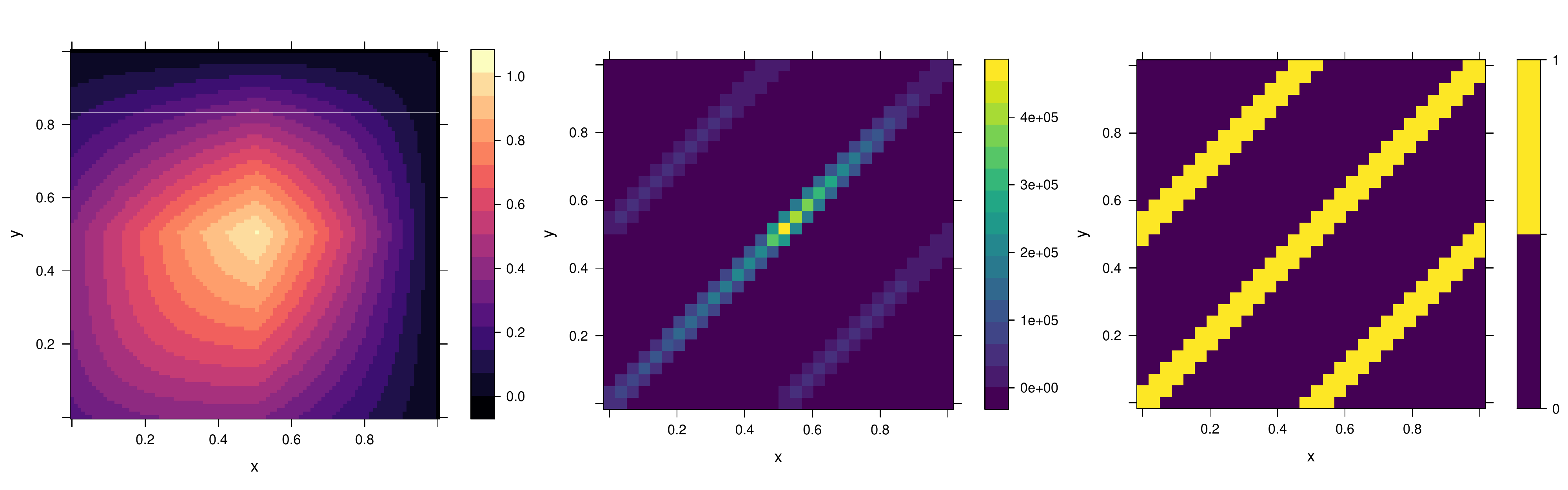}
        \caption{$K_{3}$ (Integrated Brownian motion)}
        \label{fig:g3}
     \end{subfigure}
     \begin{subfigure}[b]{0.82\textwidth}
        \includegraphics[width=1\linewidth, trim=1.5cm 1.2cm 0 0, clip]{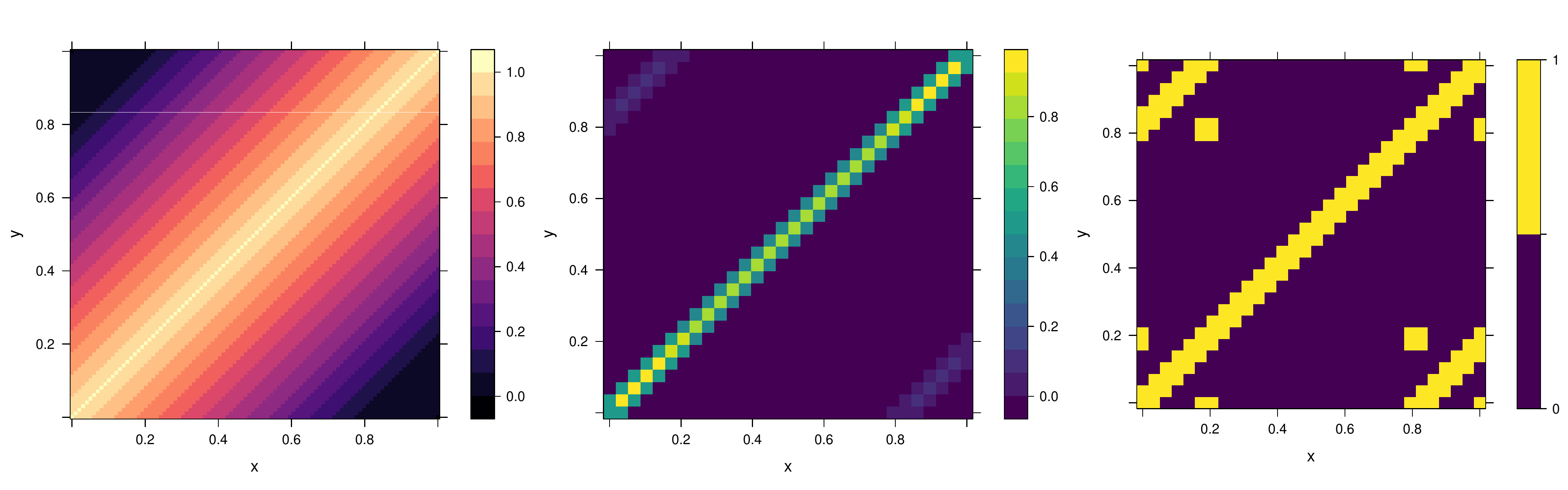}
        \caption{$K_{4}$ (P\'{o}lya covariance)}
        \label{fig:g4}
     \end{subfigure}
     \begin{subfigure}[b]{0.82\textwidth}
        \includegraphics[width=1\linewidth, trim=1.5cm 1.2cm 0 0, clip]{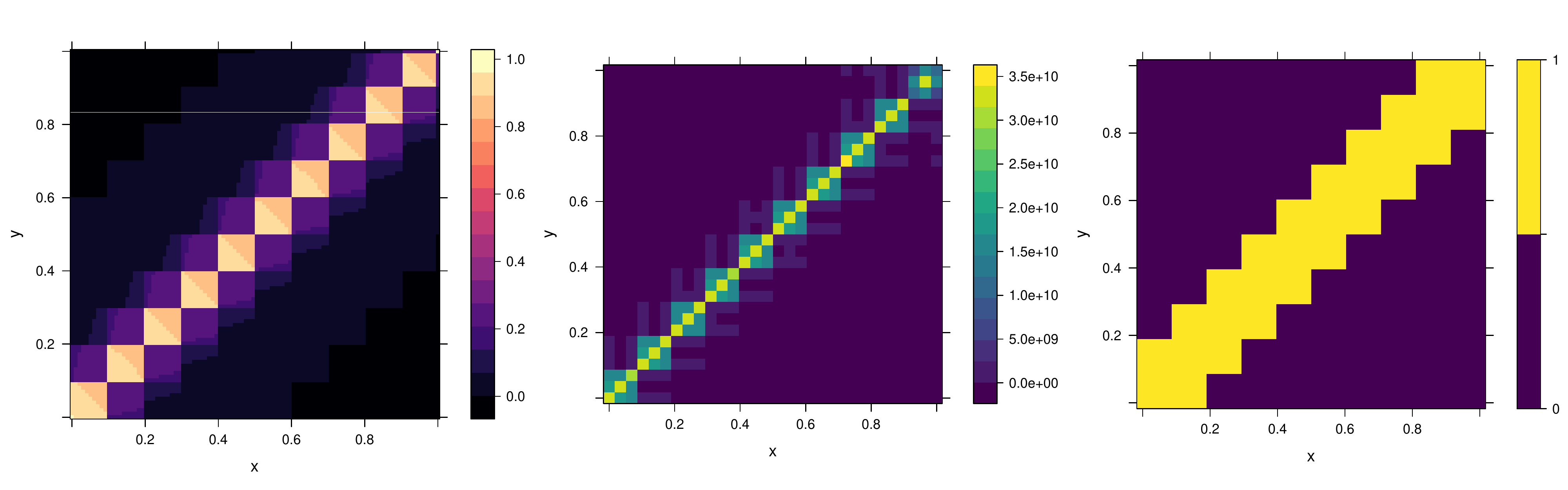}
        \caption{$K_{5}$ (Linear interpolation of random vector)}
        \label{fig:g5}
     \end{subfigure}\\[1cm]
    \caption[Two numerical solutions]{Plots of the covariance $K$ (left), the matrix of norms $P = [\| \bP_{ij}\|]_{i,j=1}^{p}$ (center) and $\Omega_{X}^{\pi}$ (right) for $K = $ (a) $K_{1}$, (b) $K_{2}$, (c) $K_{3}$, (d) $K_{4}$ and (e) $K_{5}$.}
    \label{fig:plots_of_covariances}
\end{figure}

For the covariances $K = K_{j}$, we proceed by generating independent samples from the corresponding multivariate Gaussian distribution in the manner dictated by the corresponding observation regime, and estimate the covariance $\bK$ accordingly. We calculate $\hat{\Omega}^{\pi}(\rho)$ for various values of $\rho$ using the method described in Section \ref{sec:methodology}. We compare $\hat{\Omega}^{\pi}(\rho)$ with the true $\tilde{\Omega}_{X}^{\pi}$ and calculate the true positive rate (TPR) and the false positive rate (FPR) of classifying the pixels $U_{i} \times U_{j}$ for every $\rho$ as 
\begin{equation*}
    \mathrm{ TPR}(\rho) = \frac{\#\{(i, j): U_{i} \times U_{j} \in \hat{\Omega}^{\pi}(\rho) \cap \tilde{\Omega}_{X}^{\pi} \}}{\#\{(i, j): U_{i} \times U_{j} \in \tilde{\Omega}_{X}^{\pi}\}} \mbox{ and } 
    \mathrm{ FPR}(\rho) = \frac{\#\{(i, j): U_{i} \times U_{j} \in \hat{\Omega}^{\pi}(\rho) \setminus  \tilde{\Omega}_{X}^{\pi} \}}{\#\{(i, j): U_{i} \times U_{j} \notin \tilde{\Omega}_{X}^{\pi}\}},
\end{equation*}
where $\#S$ denotes the cardinality of the set $S$.
Then we plot a receiver operating characteristic (ROC) curve as in Figure \ref{fig:roc_curves}. We calculate the area under the curve (AUC) of the ROC curve. We do this $100$ times for every combination of $K$, $p$ and $n$, and report the median and mean absolute deviation of the AUC rounded to two decimal places.

\begin{figure}
    \centering
    \includegraphics[width=0.6\textwidth]{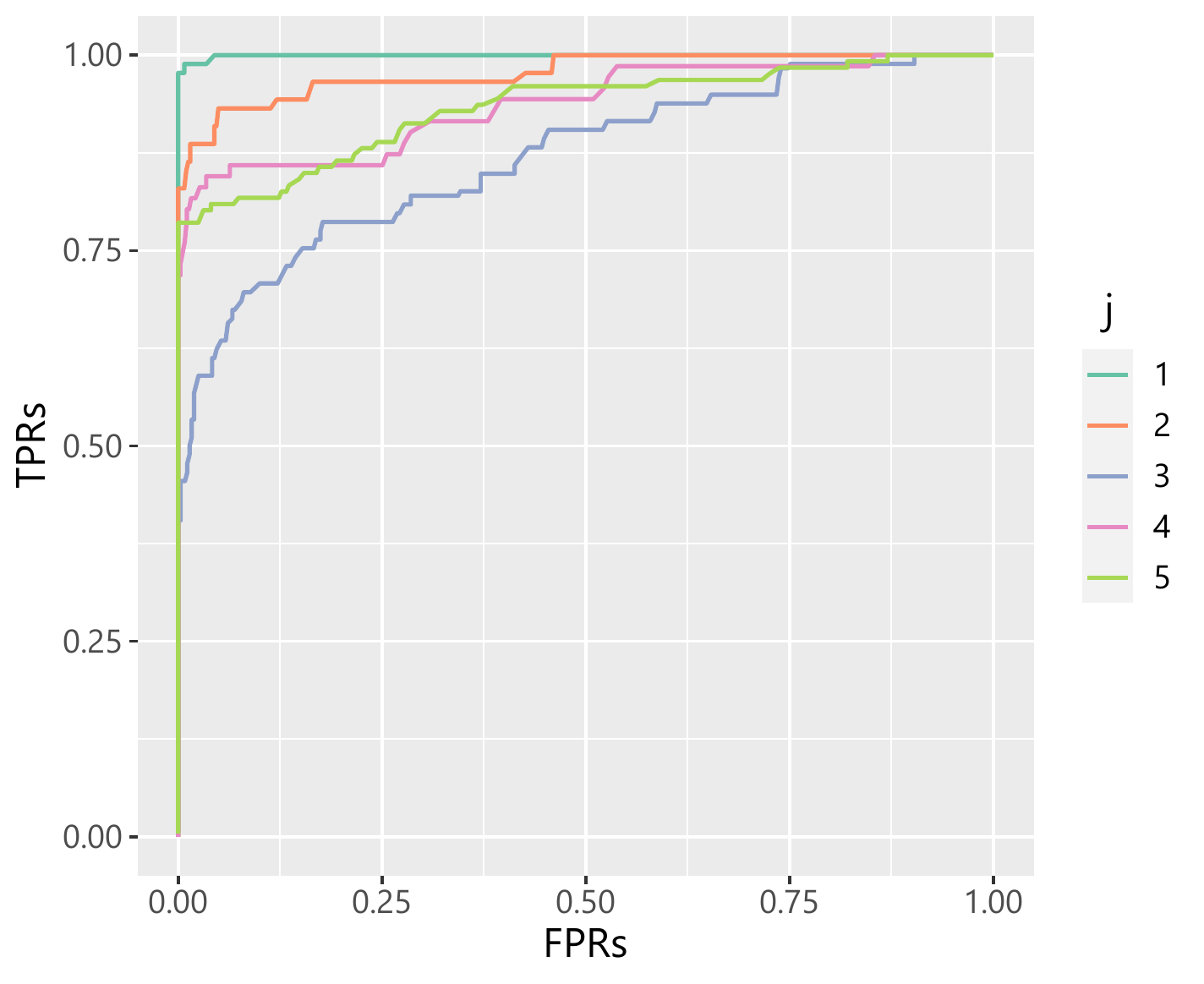}\\[0.3cm]
    \caption{Sample ROC curves for simulated instances of the covariances $K_{j}$ for $1 \leq j \leq 5$.}
    \label{fig:roc_curves}
\end{figure}

\subsection{Complete Observations}

We generate $n = 50, 100, 200$ independent samples from the multivariate Gaussian distribution with mean zero and covariance corresponding to $K = K_{j}$ on a regular grid on $U$ of length $R = 600$. We estimate $\bK$ using the empirical covariance estimator $\hbK$ and proceed as described above for partitions corresponding to the resolutions $p = 20, 30, 40$. The grid $\Lambda$ for implementing (\ref{eqn:ridge_choice}) was chosen to be $\{10^{-j}\|\dg \bK\|: 0 \leq j \leq 14\}$. The results are displayed in Table \ref{tab:simulations-compl}. 

The median AUC generally increases with sample size, as expected. However, no clear pattern is evident between the median AUC and $p$, which is understandable considering the complex nature of this relationship, as described in Section \ref{sec:complete_obs}.

\begin{table}[h]
\centering
\begin{tabular}{|| c | c || c | c | c ||} 
\hline
\multicolumn{2}{||c||}{Regime} 
&\multicolumn{3}{|c||}{Complete Observations} \Tstrut \\[2px]
\hline
\multicolumn{2}{||c||}{Parameters} 
&\multicolumn{3}{|c||}{$n$} \Tstrut \\[2px]
\hline
$K$ & $p$ &50 &100 &200 \Tstrut\\ [1ex] 
\hline\hline 
\multirow{3}{*}{$K_{1}$} 
&20 &1.00$\pm$0.00 &1.00$\pm$0.00 &1.00$\pm$0.00\Tstrut \\
&30 &1.00$\pm$0.00 &1.00$\pm$0.00 &1.00$\pm$0.00\Tstrut \\
&40 &0.99$\pm$0.00 &0.99$\pm$0.00 &0.99$\pm$0.00\Tstrut \\
\hline
\multirow{3}{*}{$K_{2}$} 
&20 &0.96$\pm$0.01 &0.97$\pm$0.01 &0.98$\pm$0.01\Tstrut \\
&30 &0.95$\pm$0.01 &0.97$\pm$0.01 &0.97$\pm$0.00\Tstrut \\
&40 &0.95$\pm$0.01 &0.96$\pm$0.01 &0.97$\pm$0.00\Tstrut \\
\hline
\multirow{3}{*}{$K_{3}$} 
&20 &0.84$\pm$0.03 &0.88$\pm$0.02 &0.88$\pm$0.02\Tstrut \\
&30 &0.86$\pm$0.02 &0.87$\pm$0.01 &0.88$\pm$0.01\Tstrut \\
&40 &0.86$\pm$0.02 &0.88$\pm$0.01 &0.88$\pm$0.01\Tstrut \\
\hline
\multirow{3}{*}{$K_{4}$} 
&20 &0.82$\pm$0.03 &0.84$\pm$0.03 &0.85$\pm$0.03\Tstrut \\
&30 &0.84$\pm$0.02 &0.87$\pm$0.02 &0.89$\pm$0.03\Tstrut \\
&40 &0.86$\pm$0.02 &0.89$\pm$0.02 &0.90$\pm$0.02\Tstrut \\
\hline
\multirow{3}{*}{$K_{5}$} 
&20 &1.00$\pm$0.00 &1.00$\pm$0.00 &1.00$\pm$0.00\Tstrut \\
&30 &1.00$\pm$0.00 &1.00$\pm$0.00 &1.00$\pm$0.00\Tstrut \\
&40 &1.00$\pm$0.00 &1.00$\pm$0.00 &1.00$\pm$0.00\Tstrut \\
\hline
\end{tabular}
\caption{\label{tab:simulations-compl} Medians $\pm$ mean absolute deviations (MAD) of area under the curve (AUC) for complete observations.}
\end{table}

\subsection{Regular Observations}

We generate $n = 50, 100, 200$ independent samples for every covariance $K = K_{j}$ but on a regular grid on $U$ of length $R = 601$ and additively perturb every entry with independent Gaussian noise of mean zero and variance $\eta \tr \bK$ for the noise levels $\eta =0$,  $0.01$ and $0.1$. The marginally larger grid size simply ensures that the estimate of the covariance can be represented as a $600 \times 600$ matrix as in the complete observations case. We estimate $\bK$ using the estimator $\hbK_{\mathrm{regular}}$ proposed in Section \ref{sec:regular_obs} and proceed as before. The grid $\Lambda$ for implementing (\ref{eqn:ridge_choice}) was chosen to be $\{10^{-\alpha}\|\dg \bK\|: \alpha = 2(\tfrac{j}{14})-1(\tfrac{14-j}{14}) \mbox{ where } 0 \leq j \leq 14\}$. The results are displayed in Table \ref{tab:simulations-reg}.  

Unsurprisingly, the median AUC tends to increase and the mean absolute deviation of AUC tends to increase as the sample size $n$ increases and decrease as the the noise level $\eta$ increases. The AUC for noiseless regular observation case ($\eta = 0\%$) is very similar to that for complete observations for most covariances with the exception of $K = K_{3}$.

\subsection{Sparse Observations}\label{sec:sparse_simulations}

We generate $n = 500$ independent samples from the multivariate Gaussian distribution on a regular grid on $U$ of length $R = 600$ and retain only $r=5$ observations for every sample after adding independent Gaussian noise of mean zero and variance $\eta \tr \bK$ for the noise levels $\eta = 0$, $0.01$ and $0.1$. We estimate $\bK$ using $\hbK_{\mathrm{sparse}}$ for $M = 20$ as proposed in Section \ref{sec:sparse_obs}, and also using the standard PACE covariance estimator $\hbK_{\mathrm{sparse}}$ (see \cite{fdapace}) for comparison. We then proceed as before for partitions corresponding to the resolutions $p = 10, 20, 30, 40$. The grid $\Lambda$ for implementing (\ref{eqn:ridge_choice}) was chosen to be same as the one for regular observations. The results are displayed in Table \ref{tab:simulations-sparse}. 

Due to the averaging effect of both $\hbK_{\mathrm{sparse}}$ and $\hbK_{\mathrm{pace}}$, the noise level does not make an appreciable difference in the AUC. The PACE covariance estimator tends to perform better for smoother covariances, as expected, and, perhaps counter-intuitively, also for larger values of $p$.


\begin{landscape}
\begin{table}[h]
\centering
\caption{\label{tab:simulations-reg} Medians $\pm$ mean absolute deviations (MAD) of area under the curve (AUC) for regular observations.}
\begin{tabular}{|| c | c || c | c | c || c | c | c || c | c | c ||} 
    \hline
    \multicolumn{2}{||c||}{Regime} 
    &\multicolumn{9}{|c||}{Regular Observations} \Tstrut \\[2px]
    \hline    
    \multicolumn{2}{||c||}{\multirow{2}{*}{Parameters}} 
    &\multicolumn{3}{|c||}{$\eta = 0\%$} &\multicolumn{3}{|c||}{$\eta = 1\%$} &\multicolumn{3}{|c||}{$\eta = 10\%$} \Tstrut \\[2px]
    \cline{3-11}
    \multicolumn{2}{||c||}{ } &\multicolumn{3}{|c||}{$n$} &\multicolumn{3}{|c||}{$n$} &\multicolumn{3}{|c||}{$n$} \Tstrut \\[2px]
    \hline
    $K$ & $p$ &50 &100 &200 &50 &100 &200 &50 &100 &200 \Tstrut\\ [1ex] 
    \hline\hline
\multirow{3}{*}{$K_{1}$} 
&20 &0.94$\pm$0.01 &0.95$\pm$0.00 &0.96$\pm$0.00 &0.94$\pm$0.01 &0.95$\pm$0.00 &0.96$\pm$0.00 &0.93$\pm$0.01 &0.94$\pm$0.00 &0.95$\pm$0.00\Tstrut \\
&30 &0.93$\pm$0.01 &0.94$\pm$0.00 &0.95$\pm$0.00 &0.93$\pm$0.01 &0.94$\pm$0.00 &0.95$\pm$0.00 &0.93$\pm$0.01 &0.94$\pm$0.00 &0.94$\pm$0.00\Tstrut \\
&40 &0.93$\pm$0.00 &0.94$\pm$0.00 &0.95$\pm$0.00 &0.93$\pm$0.00 &0.94$\pm$0.00 &0.95$\pm$0.00 &0.92$\pm$0.01 &0.93$\pm$0.01 &0.94$\pm$0.01\Tstrut \\
\hline
\multirow{3}{*}{$K_{2}$}
&20 &0.96$\pm$0.01 &0.97$\pm$0.00 &0.98$\pm$0.00 &0.96$\pm$0.01 &0.97$\pm$0.00 &0.98$\pm$0.00 &0.96$\pm$0.01 &0.96$\pm$0.00 &0.97$\pm$0.00\Tstrut \\
&30 &0.97$\pm$0.00 &0.97$\pm$0.00 &0.98$\pm$0.00 &0.97$\pm$0.00 &0.97$\pm$0.00 &0.98$\pm$0.00 &0.96$\pm$0.01 &0.97$\pm$0.00 &0.98$\pm$0.00\Tstrut \\
&40 &0.97$\pm$0.00 &0.97$\pm$0.00 &0.98$\pm$0.00 &0.97$\pm$0.00 &0.97$\pm$0.00 &0.98$\pm$0.00 &0.96$\pm$0.01 &0.97$\pm$0.00 &0.98$\pm$0.00\Tstrut \\
\hline
\multirow{3}{*}{$K_{3}$}
&20 &0.59$\pm$0.01 &0.59$\pm$0.01 &0.59$\pm$0.01 &0.59$\pm$0.01 &0.59$\pm$0.01 &0.59$\pm$0.01 &0.58$\pm$0.02 &0.58$\pm$0.01 &0.59$\pm$0.01\Tstrut \\
&30 &0.62$\pm$0.01 &0.62$\pm$0.01 &0.62$\pm$0.00 &0.62$\pm$0.01 &0.62$\pm$0.01 &0.62$\pm$0.00 &0.60$\pm$0.01 &0.61$\pm$0.01 &0.62$\pm$0.01\Tstrut \\
&40 &0.64$\pm$0.01 &0.64$\pm$0.01 &0.64$\pm$0.00 &0.63$\pm$0.01 &0.64$\pm$0.01 &0.64$\pm$0.00 &0.61$\pm$0.01 &0.63$\pm$0.01 &0.63$\pm$0.01\Tstrut \\
\hline
\multirow{3}{*}{$K_{4}$}
&20 &0.79$\pm$0.03 &0.83$\pm$0.02 &0.85$\pm$0.02 &0.79$\pm$0.03 &0.83$\pm$0.02 &0.85$\pm$0.01 &0.78$\pm$0.03 &0.83$\pm$0.02 &0.84$\pm$0.02\Tstrut \\
&30 &0.82$\pm$0.02 &0.85$\pm$0.02 &0.86$\pm$0.01 &0.82$\pm$0.02 &0.85$\pm$0.02 &0.86$\pm$0.01 &0.81$\pm$0.03 &0.84$\pm$0.02 &0.85$\pm$0.02\Tstrut \\
&40 &0.84$\pm$0.03 &0.87$\pm$0.02 &0.88$\pm$0.01 &0.84$\pm$0.02 &0.86$\pm$0.02 &0.88$\pm$0.02 &0.83$\pm$0.03 &0.86$\pm$0.02 &0.86$\pm$0.01\Tstrut \\
\hline
\multirow{3}{*}{$K_{5}$}
&20 &1.00$\pm$0.00 &1.00$\pm$0.00 &1.00$\pm$0.00 &1.00$\pm$0.00 &1.00$\pm$0.00 &1.00$\pm$0.00 &1.00$\pm$0.00 &1.00$\pm$0.00 &1.00$\pm$0.00\Tstrut \\
&30 &1.00$\pm$0.00 &1.00$\pm$0.00 &1.00$\pm$0.00 &1.00$\pm$0.00 &1.00$\pm$0.00 &1.00$\pm$0.00 &1.00$\pm$0.00 &1.00$\pm$0.00 &1.00$\pm$0.00\Tstrut \\
&40 &1.00$\pm$0.00 &1.00$\pm$0.00 &1.00$\pm$0.00 &1.00$\pm$0.00 &1.00$\pm$0.00 &1.00$\pm$0.00 &1.00$\pm$0.00 &1.00$\pm$0.00 &1.00$\pm$0.00\Tstrut \\
\hline
\end{tabular}
\end{table}
\end{landscape}

\begin{landscape}
\begin{table}[h]
\centering
\caption{\label{tab:simulations-sparse} Medians $\pm$ mean absolute deviations (MAD) of area under the curve (AUC) for sparse observations.}
\begin{tabular}{|| c | c || c | c || c | c || c | c ||} 
    \hline
    \multicolumn{2}{||c||}{Regime} 
    &\multicolumn{6}{|c||}{Sparse Observations} \Tstrut \\[2px]
    \hline    
    \multicolumn{2}{||c||}{\multirow{2}{*}{Parameters}} 
    &\multicolumn{6}{|c||}{$n = 500$} \Tstrut \\[2px]
    \cline{3-8}
    \multicolumn{2}{||c||}{ }&\multicolumn{2}{|c||}{$\eta = 0\%$} &\multicolumn{2}{|c||}{$\eta = 1\%$} &\multicolumn{2}{|c||}{$\eta = 10\%$} \Tstrut \\[2px]     
    \hline
    $K$ & $p$ &$\hbK_{\mathrm{sparse}}$ &$\hbK_{\mathrm{pace}}$ &$\hbK_{\mathrm{sparse}}$ &$\hbK_{\mathrm{pace}}$ &$\hbK_{\mathrm{sparse}}$ &$\hbK_{\mathrm{pace}}$ \Tstrut\\ [1ex] 
    \hline\hline
\multirow{4}{*}{$K_{1}$} 
&10 &0.55$\pm$0.08 &0.85$\pm$0.14 &0.56$\pm$0.08 &0.81$\pm$0.16 &0.57$\pm$0.07 &0.84$\pm$0.15 \Tstrut \\
&20 &0.74$\pm$0.03 &0.97$\pm$0.04 &0.74$\pm$0.03 &0.97$\pm$0.04 &0.73$\pm$0.03 &0.97$\pm$0.05 \Tstrut \\
&30 &0.89$\pm$0.00 &0.99$\pm$0.01 &0.89$\pm$0.00 &0.99$\pm$0.01 &0.89$\pm$0.00 &0.99$\pm$0.01 \Tstrut \\
&40 &0.91$\pm$0.00 &1.00$\pm$0.00 &0.91$\pm$0.00 &1.00$\pm$0.00 &0.91$\pm$0.00 &1.00$\pm$0.01 \Tstrut \\
\hline
\multirow{4}{*}{$K_{2}$}
&10 &0.58$\pm$0.08 &0.81$\pm$0.08 &0.59$\pm$0.08 &0.82$\pm$0.07 &0.59$\pm$0.07 &0.82$\pm$0.07 \Tstrut \\
&20 &0.65$\pm$0.03 &0.87$\pm$0.04 &0.66$\pm$0.03 &0.87$\pm$0.05 &0.66$\pm$0.03 &0.87$\pm$0.05 \Tstrut \\
&30 &0.89$\pm$0.00 &0.91$\pm$0.02 &0.89$\pm$0.00 &0.91$\pm$0.02 &0.89$\pm$0.00 &0.91$\pm$0.03 \Tstrut \\
&40 &0.91$\pm$0.00 &0.92$\pm$0.02 &0.91$\pm$0.00 &0.92$\pm$0.03 &0.91$\pm$0.00 &0.92$\pm$0.03 \Tstrut \\
\hline
\multirow{4}{*}{$K_{3}$}
&10 &0.54$\pm$0.07 &0.56$\pm$0.04 &0.53$\pm$0.06 &0.57$\pm$0.05 &0.53$\pm$0.06 &0.57$\pm$0.04 \Tstrut \\
&20 &0.58$\pm$0.01 &0.60$\pm$0.03 &0.58$\pm$0.01 &0.60$\pm$0.04 &0.58$\pm$0.01 &0.60$\pm$0.03 \Tstrut \\
&30 &0.68$\pm$0.00 &0.62$\pm$0.04 &0.68$\pm$0.00 &0.63$\pm$0.04 &0.68$\pm$0.00 &0.63$\pm$0.04 \Tstrut \\
&40 &0.69$\pm$0.00 &0.65$\pm$0.02 &0.69$\pm$0.00 &0.65$\pm$0.03 &0.69$\pm$0.00 &0.65$\pm$0.02 \Tstrut \\
\hline
\multirow{4}{*}{$K_{4}$}
&10 &0.54$\pm$0.10 &0.67$\pm$0.11 &0.54$\pm$0.10 &0.68$\pm$0.08 &0.53$\pm$0.09 &0.66$\pm$0.09 \Tstrut \\
&20 &0.64$\pm$0.02 &0.69$\pm$0.08 &0.64$\pm$0.02 &0.68$\pm$0.08 &0.63$\pm$0.02 &0.69$\pm$0.08 \Tstrut \\
&30 &0.74$\pm$0.00 &0.74$\pm$0.05 &0.74$\pm$0.00 &0.74$\pm$0.04 &0.74$\pm$0.00 &0.74$\pm$0.04 \Tstrut \\
&40 &0.75$\pm$0.00 &0.76$\pm$0.03 &0.75$\pm$0.00 &0.75$\pm$0.03 &0.76$\pm$0.00 &0.76$\pm$0.04 \Tstrut \\
\hline
\multirow{4}{*}{$K_{5}$}
&10 &0.98$\pm$0.02 &0.74$\pm$0.13 &0.98$\pm$0.03 &0.77$\pm$0.10 &0.95$\pm$0.07 &0.76$\pm$0.12 \Tstrut \\
&20 &0.84$\pm$0.04 &0.94$\pm$0.05 &0.83$\pm$0.04 &0.93$\pm$0.05 &0.84$\pm$0.04 &0.93$\pm$0.06 \Tstrut \\
&30 &0.72$\pm$0.03 &0.89$\pm$0.05 &0.72$\pm$0.03 &0.88$\pm$0.07 &0.73$\pm$0.03 &0.88$\pm$0.07 \Tstrut \\
&40 &0.67$\pm$0.02 &0.90$\pm$0.09 &0.67$\pm$0.02 &0.93$\pm$0.06 &0.67$\pm$0.02 &0.91$\pm$0.07 \Tstrut \\
\hline
\end{tabular}
\end{table}
\end{landscape}

\section{Illustrative Data Analysis}\label{sec:data_analysis}

In this section, we illustrate our method by analyzing two data sets. The first concerns infrared absorption spectra obtained from fruit purees where we expect the graph to have significant associations between distant locations. The second involves the intraday price of a certain stock where we expect the graph to resemble that of a Markov process as in Figure \ref{fig:plots_of_covariances} (a) or (b).

\subsection{Infrared Absorption Spectroscopy} 

\begin{figure}[h]
    \centering
    \includegraphics[width=0.8\textwidth, trim={0 0 0 3cm}, clip]{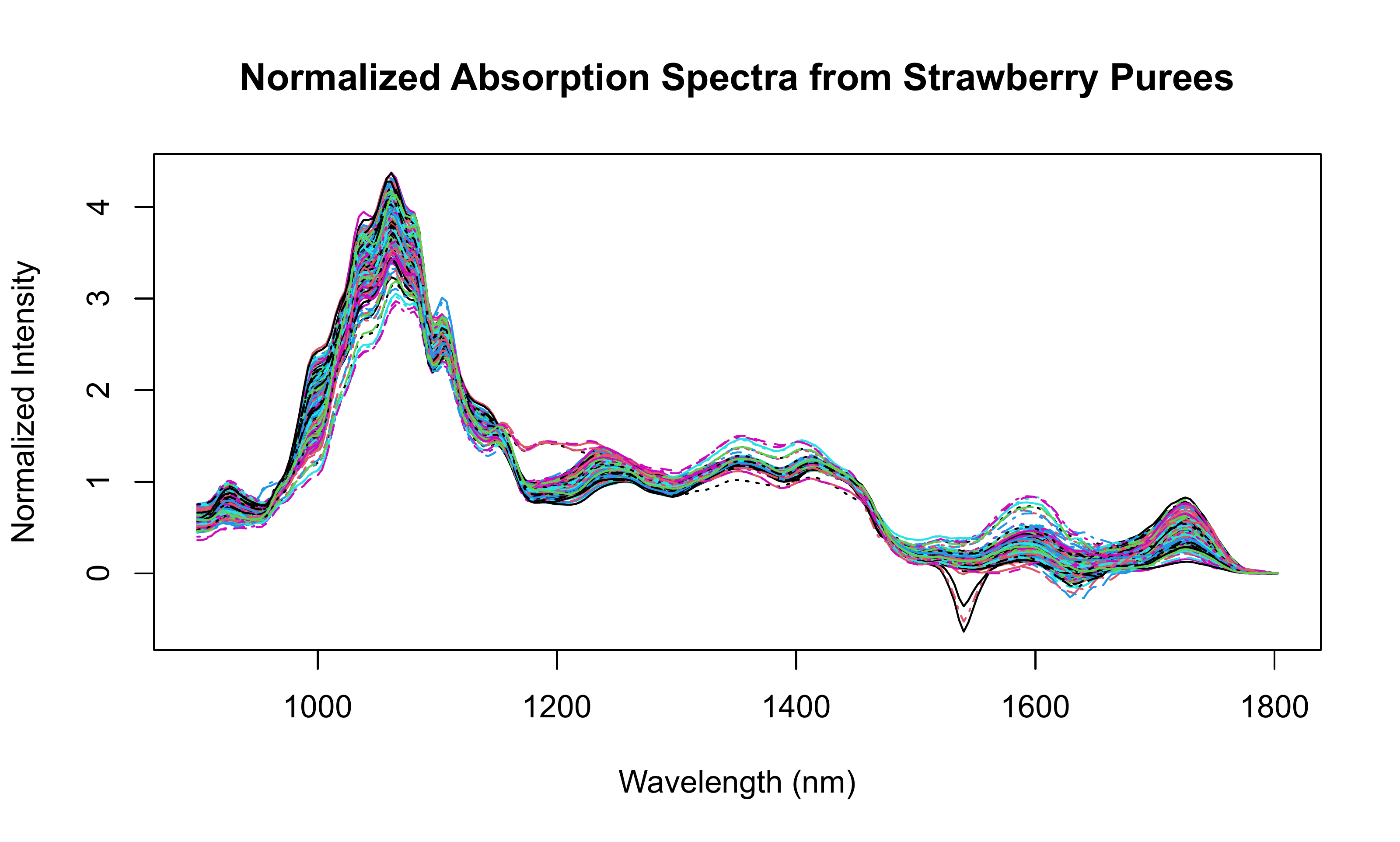}\\[0.5cm]
    \caption{Absorption spectra of strawberry purees.}
    \label{fig:plot-spectra-puree}
\end{figure}

\textcite{codazzi2022} model the absorption spectra obtained from a sample of strawberry purees as continuous functions, and produce a Bayesian inference procedure to infer the underlying dependence structure. This structure is of interest in determining the chemical composition of the puree samples. In particular, if different regions of the spectrum are related, then they might correspond to the same chemical component. The method of \textcite{codazzi2022} involves B-spline smoothing of the spectra, and uses the conditional dependence between the smoothing coefficients as a proxy for the conditional dependence structure of the spectra.   

\begin{figure}
    \centering
    \includegraphics[width=0.8\textwidth, trim={0 0 0 1.5cm}, clip]{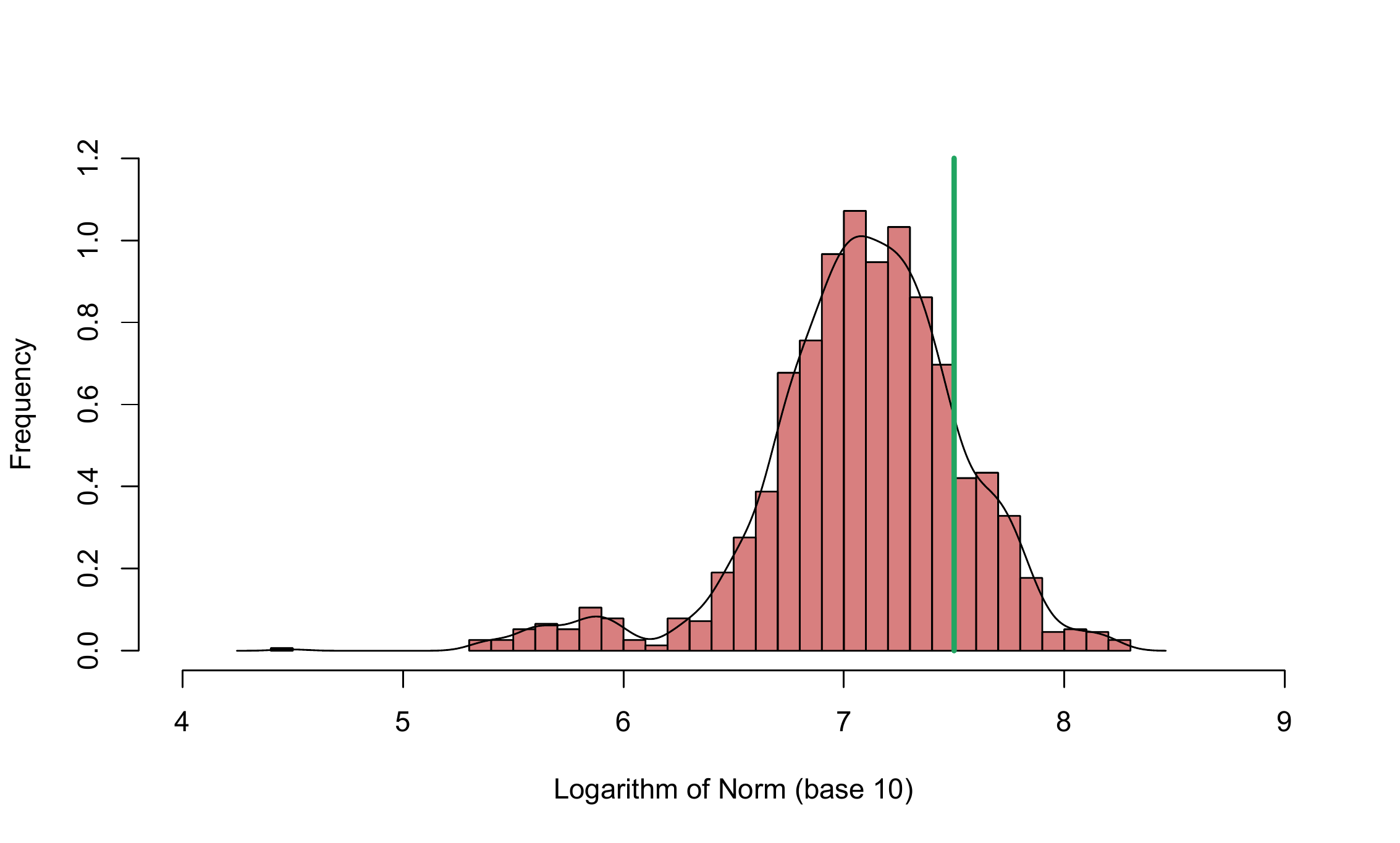}\\[0.5cm]
    \caption{Histogram and density of the log-norms $\{\log_{10} \|\hbP_{ij}\|: 1 \leq i,j \leq p\}$ for the strawberry puree data. The green line indicates the threshold $\rho$ chosen for the graph in Figure \ref{fig:plots-puree} (b). It has been manually chosen to be slightly less than the value corresponding to the elbow of the density curve which corresponds to $\rho = 10^{7.6}$.}
    \label{fig:density_puree}
\end{figure}

\begin{figure}
    \centering
    \begin{subfigure}{.48\textwidth}
        \centering
        \includegraphics[width=1.1\linewidth, trim={2.5cm 0 2.5cm 0}, clip, scale = 1]{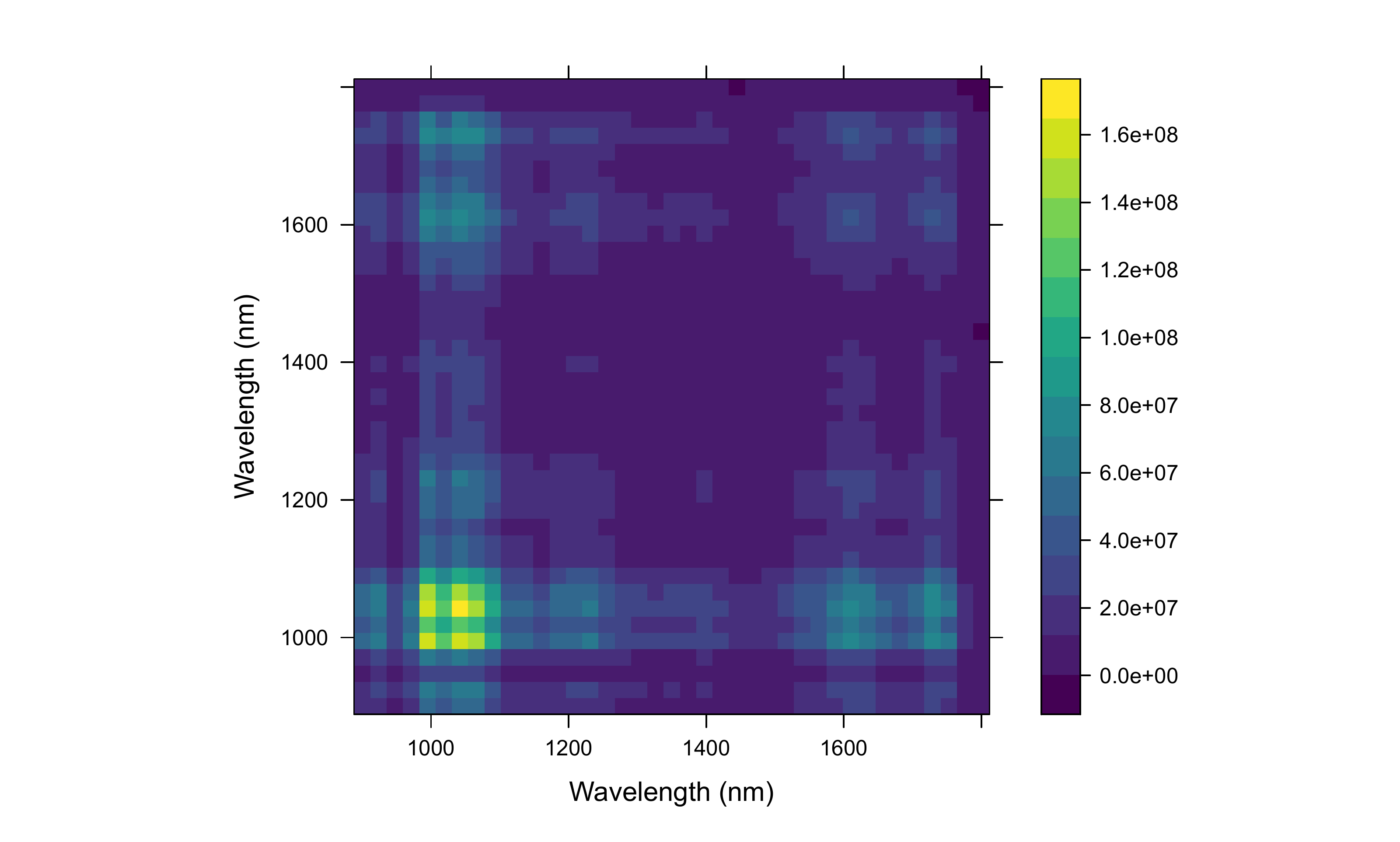}
        \caption{}
        \label{fig:mat-of-opnorms-puree}
    \end{subfigure}
    \begin{subfigure}{.48\textwidth}
        \centering
        \includegraphics[width=\linewidth, trim={2.5cm 0 2.5cm 0}, clip, scale = 1]{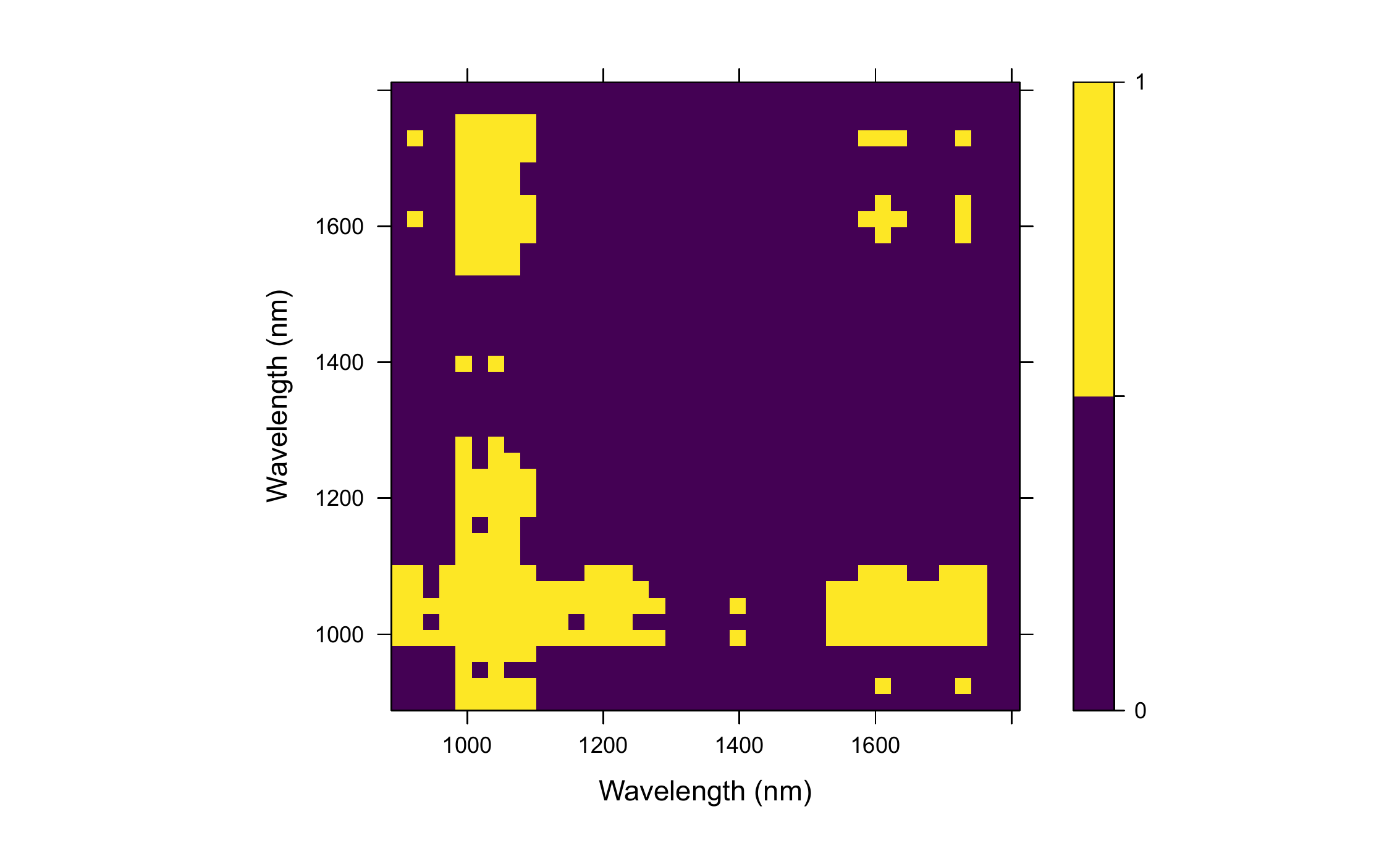}
        \caption{}
        \label{fig:graph-puree}
    \end{subfigure}\\[0.5cm]
    \caption{(a) The matrix of operator norms $[\|\bP_{ij} \|]_{i,j=1}^{p}$ and (b) the graph 
    $\tilde{\Omega}^{\pi}_{X}$ obtained for the threshold $\rho = 10^{7.5}$ for the absorption spectra of strawberry purees.}
    \label{fig:plots-puree}
\end{figure}

Using our method, we approach the problem directly. We calculate the covariance of $L^{1}$-normalized absorption spectra readings from the dataset \textcite{fruitdata2019}, obtained from $n = 351$ puree samples on a uniform grid of $235$ wavelengths on the interval $I = [899.327 \mathrm{~nm}, 1802.564 \mathrm{~nm}]$ (see Figure \ref{fig:plot-spectra-puree}). We discard the last wavelength so as to make it easier to divide the domain into $p = 39$ partitions and calculate the corresponding precision matrix, which is thresholded at a manually chosen level of $\rho = 10^{7.5}$ using the method described in Section \ref{sec:implementation} (see Figure \ref{fig:density_puree}). 
The results are summarized in Figure \ref{fig:plots-puree} and they strongly suggest the existence of chemical components corresponding simultaneously to low and high wavelengths. This corroborates the findings of \textcite{codazzi2022}, who arrive at a comparable graph structure, but using very different and more elaborate methods.

\subsection{Cumulative Log-returns for Pfizer Limited}
\begin{figure}
    \centering
    \includegraphics[width=0.8\textwidth, trim={0 0 0 1.5cm}, clip]{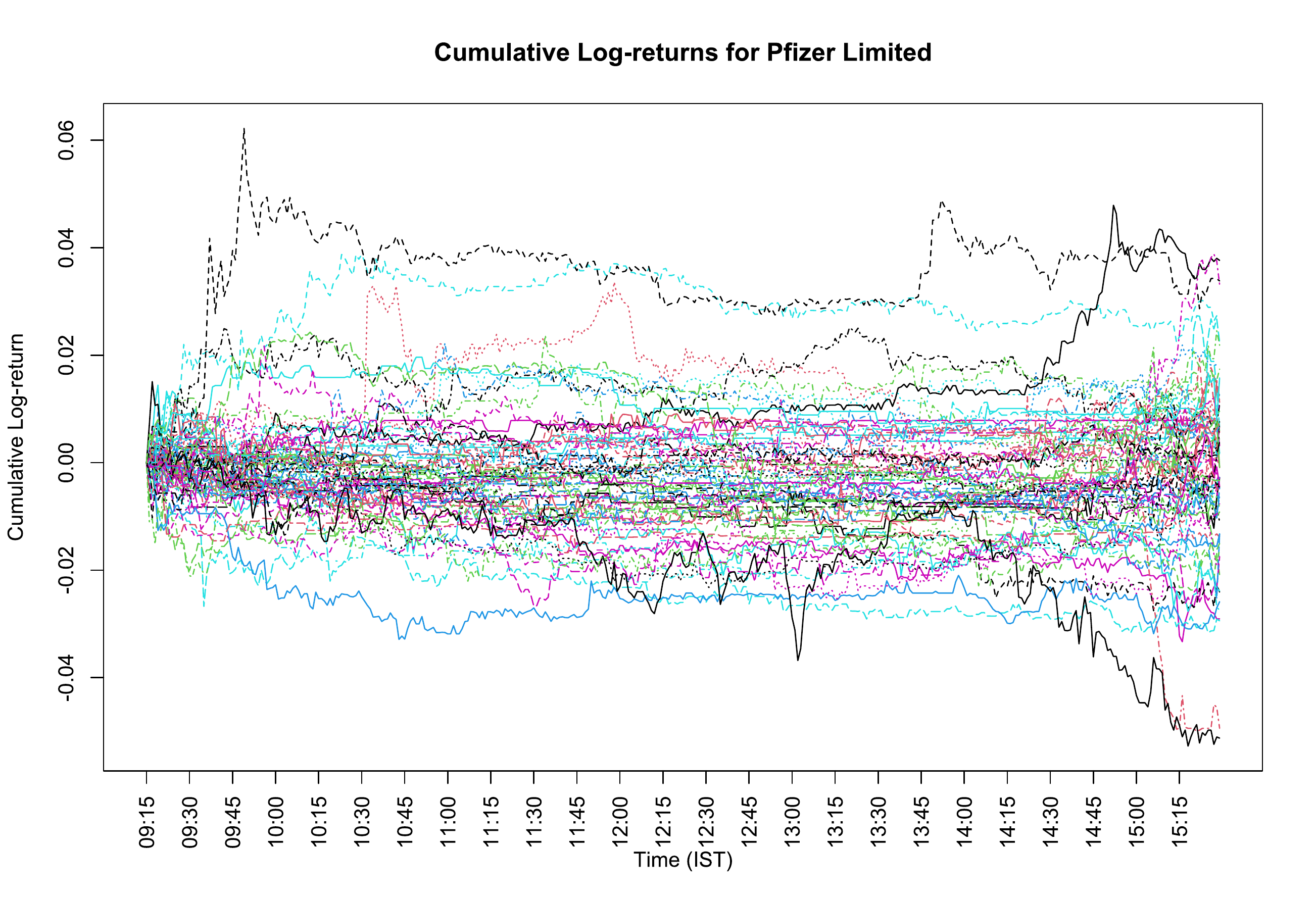}\\[0.5cm]
    \caption{Cumulative log-returns of Pfizer Limited during regular trading sessions from $2$nd January $2017$ to $1$st January $2021$.}
    \label{fig:plot-prices-pfizer}
\end{figure}

We consider the intraday price of Pfizer Limited (NSE: PFIZER) listed on India's National Stock Exchange (NSE) during $988$ regular trading sessions ($09$:$15$ AM - $15$:$30$ PM IST) from $2$nd January $2017$ to $1$st January $2021$ (see Figure \ref{fig:plot-prices-pfizer}). The data has been made freely available on Kaggle by \textcite{harshkumar}. For every day $j$, we calculate the cumulative log-returns $X_{j}(t) = \log (P_{jt}/P_{j0})$ from the closing price $P_{jt}$ of the stock for the $t$th minute on the $j$th day and the opening price $P_{j0}$ of the stock on that day (see \textcite{kokoszka2012}). Since log-returns are known to be reasonably close to being Gaussian in practice (see \textcite{tsay2005}), we can think of $\{X_{j}\}_{j=1}^{n}$ as a weakly dependent Gaussian functional time series (see \textcite{gabrys2013}). For such weakly serially dependent functional data, one can estimate the true covariance $\bK$  by way of the empirical covariance $\hbK$  at the same rate as under serial independence (see Theorem 3.1 of \textcite{horman2010}). 

On many days, the trading was halted during the session, which led to missing data. To circumvent this problem, we estimate the covariance of $X_{t}$ in a pairwise manner. The resulting estimate is almost but not exactly positive semidefinite, so we project it to the cone of positive semidefinite matrices by retaining only the positive part of its eigendecomposition. The resolution of the grid is $375$ and we choose $p = 25$. The choice of threshold using the method described in Section \ref{sec:implementation} is summarized in Figure \ref{fig:density_stock} and the kernel density estimate was automatically calculated using the density function in R (\textcite{rbase}) with default parameters as before. The results are summarized in Figure \ref{fig:plots-pfizer}. 

The graph almost exactly resembles what one would expect for a Markov process, except for a noticeable clique for times between $12$:$15$ and $13$:$45$. The almost Markov nature of the graph is to be expected since it is widely believed in the academic literature in finance that stocks are mostly efficiently priced, meaning that prices reflect all the available information concerning the stock. The existence of an edge between widely spaced times would contradict this since it would imply that earlier prices during the day have information about later prices which the prices at intermediate times do not possess. The apparent existence of a clique may or may not be an interesting feature open to financial interpretation.

\begin{figure}
    \centering
    \includegraphics[width=0.6\textwidth, trim={0 0 0 1.5cm}, clip]{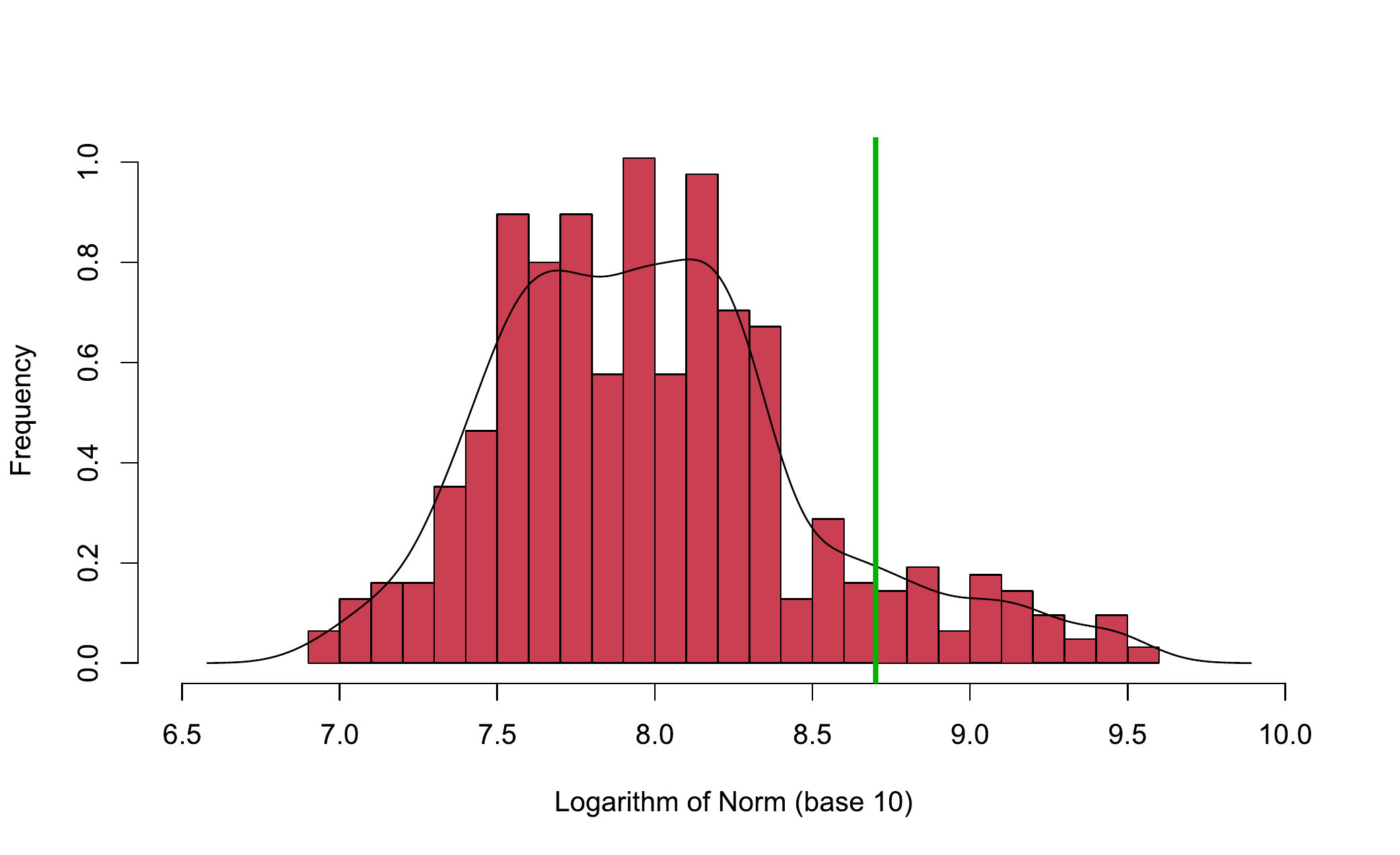}\\[0.5cm]
    \caption{Histogram and density of the log-norms $\{\log_{10}\|\hbP_{ij}\|: 1 \leq i,j \leq p\}$ for stock price data. The green line indicates the threshold $\rho$ chosen for the graph in Figure \ref{fig:plots-pfizer} (b). It has been chosen to be an elbow of the density curve which corresponds to $\rho = 10^{8.7}$}
    \label{fig:density_stock}
\end{figure}

\begin{figure}
    \centering
    \begin{subfigure}{.48\textwidth}
        \centering
        \includegraphics[width=\linewidth, trim={2.5cm 0 2.5cm 0}, clip]{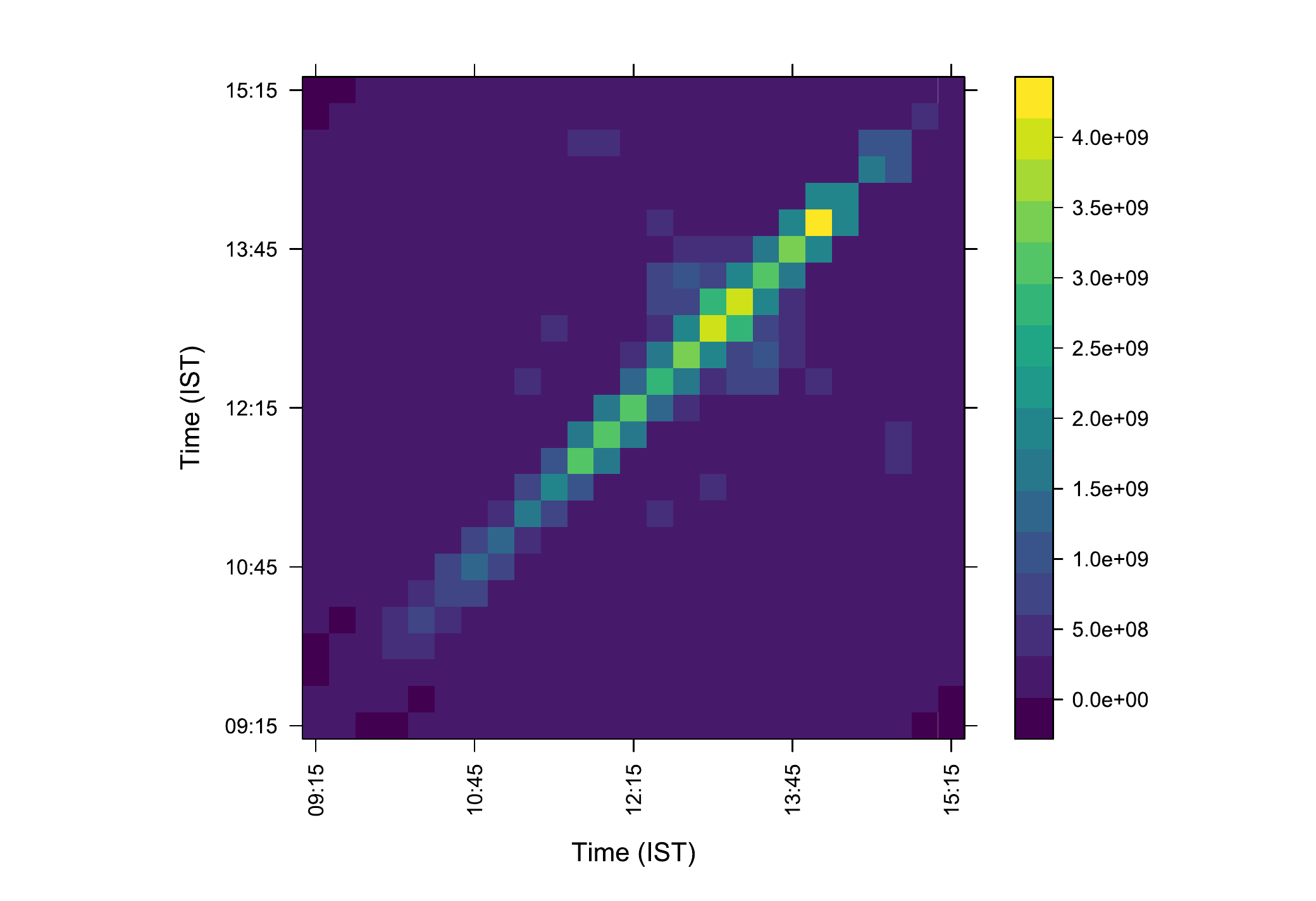}
        \caption{}
        \label{fig:mat-of-opnorms-prices}
    \end{subfigure}
    \begin{subfigure}{.48\textwidth}
        \centering
        \includegraphics[width=\linewidth, trim={2.5cm 0 2.5cm 0}, clip]{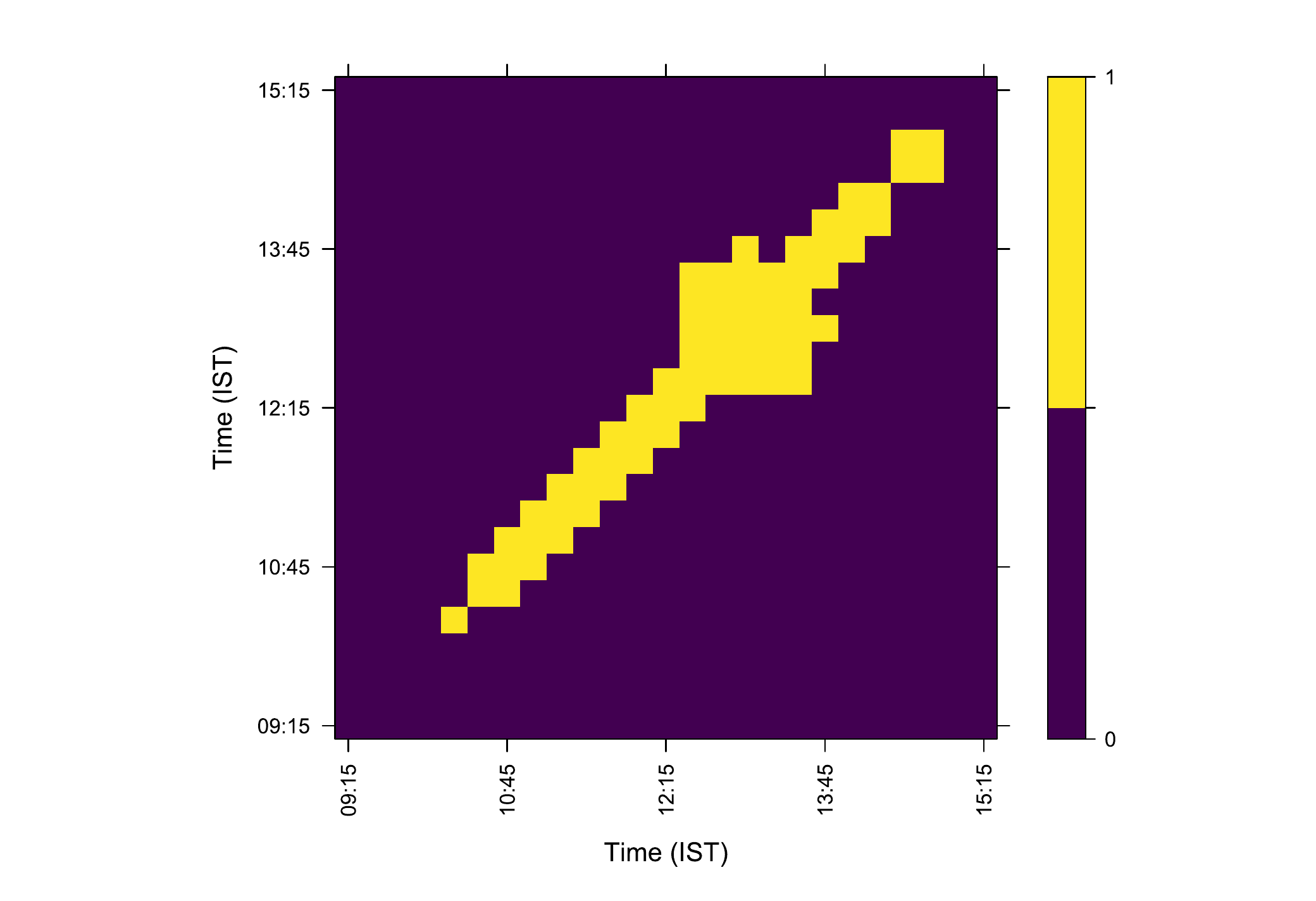}
        \caption{}
        \label{fig:graph-prices}
    \end{subfigure}\\[0.5cm]
    \caption{(a) The matrix of operator norms $[\|\bP_{ij}\|]_{i,j=1}^{p}$ and (b) the graph $\tilde{\Omega}^{\pi}_{X}$ obtained for the threshold $\rho = 10^{8.7}$ for the stock price of Pfizer Limited.}
    \label{fig:plots-pfizer}
\end{figure}

\section{Appendix}

This section collects the the proofs of the statements in the paper.

\subsection{Graphical Regularization}

\subsubsection{Approximate Inverse Zero Characterization}
\begin{proof}[Proof of Theorem \ref{thm:best_piapprx}]
    By Theorem 2.2.3 of \textcite{bakonyibook}, $\bP_{ij} = \bzero$ is equivalent to saying that
    \begin{equation}\label{eqn:sep_corrop}
        \bR_{ij} = [\bR_{ik}]_{k \in S}^{\top}[\bR_{kl}]_{k,l \in S}^{-1}[\bR_{lj}]_{l \in S}^{}
    \end{equation} 
    for $S = \{ m : m \neq i,j \}$. Through appropriate manipulations, this can be used to show that
    \begin{equation}\label{eqn:sep_matop}
        \bK_{ij} = \left(
            [\bK_{kl}]_{k,l \in S}^{-1/2}[\bK_{ki}]_{k\in S}^{\phantom{-1/2}}
            \right)^{\top}
            \left(
            [\bK_{kl}]_{k,l \in S}^{-1/2}[\bK_{lj}]_{l\in S}^{\phantom{-1/2}}
            \right).
    \end{equation}
    By Theorem 11.18. of \textcite{paulsen2016}, the above equality can be rewritten as
    \begin{equation}\label{eqn:sep_kernl}
        K(s,t) = \langle K(s, \cdot), K(\cdot, t) \rangle_{\cH(V)}
    \end{equation}
    for $s \in U_{i}$, $t \in U_{j}$ and $V = \cup_{k \in S}~ U_{k}$. It follows that $\Omega_{X} \subset (U_{i} \cup V)^{2} \cup (V \cup U_{j})^{2}$ or more simply, that $\Omega_{X}$ and $U_{i} \times U_{j}$ are disjoint. Thus implying that $U_{i} \times U_{j}$ and $\tilde{\Omega}_{X}^{\pi}$ are disjoint.

    The converse requires more work. Assume that $U_{i} \times U_{j}$ and $\tilde{\Omega}_{X}^{\pi}$ are disjoint. Now, if $x = (s, t)$ is in the closure of $U_{i} \times U_{j}$, there exists some closed $\Omega \supset \Omega_{X}$ for which (\ref{eqn:sep_kernl}) holds and $x \in \Omega^{c}$. It follows that there is an open ball $B_{x}$ centered at $x$ such that $B_{x} \subset \Omega_{c}$. The closure of $U_{i} \times U_{j}$ is contained in $\cup_{x} B_{x}$, and by compactness there exists a finite subcover $\cup_{i=1}^{q} B_{x_{i}}$. We now show that there exists a partition $\pi'$ of $U$ such that every pixel associated with $\pi'$ lies in one of the balls $B_{x_{i}}$.    
    
    Define the function $d: U_{i} \times U_{j} \to \bbR_{+}$ as
    \begin{equation*}
        d(x) = \max \{ d(x, B_{x_{i}}^{c}): x \in B_{x_{i}} \}.
    \end{equation*}
    Alternatively, $d$ maps every $x$ to the maximum of its distance from the set $B_{x_{i}}^{c}$ for every $i$ such that $x \in B_{x_{i}}$. Observe that $R = \inf_{x} d(x) > 0$. So long as we partition $U$ such that every pixel $U_{k}' \times U'_{l}$ satisfies that the maximum distance between two points in it is less than $R/2$, every pixel will be contained entirely in one of the balls $B_{x_{i}}$. 

    The precision operator $\bP' = \bP_{\pi'}$ corresponding to this new partition $\pi'$ satisfies $\bP'_{i'j'} = \bzero$ for every $i', j'$ corresponding to a pixel contained in in the closure of $U_{i} \times U_{j}$. Since such operators $\bP'_{i'j'}$ can be considered together as an operator, we can write the $\pi'$-analogue of (\ref{eqn:sep_corrop}) and work our way to (\ref{eqn:sep_kernl}) using appropriate manipulations. But (\ref{eqn:sep_kernl}) is partition independent, we can work our way backwards, this time for $\pi$ instead of $\pi'$ and derive that $\bP_{ij} = \bzero$. This completes the proof.
\end{proof}

\subsubsection{Identifiability}

\begin{proof}[Proof of Corollary \ref{cor:identifiability}]
    The first part is a tautology. For the second part, notice that for some $\epsilon_{\pi} > 0$, we can write with a slight abuse of notation that the set $\cap_{\epsilon > 0} (\Omega_{X} + \bbB_{\epsilon})^{\pi}$ is equal to $(\Omega_{X} + \bbB_{\epsilon})^{\pi}$ if $\epsilon < \epsilon_{\pi}$. Thus for $\epsilon < \epsilon_{\pi_{1}} \wedge \epsilon_{\pi_{2}}$ we have 
    \begin{eqnarray*}
        &\Bigl[\cap_{\epsilon > 0} (\Omega_{X} + \bbB_{\epsilon})^{\pi_{1}}\Bigr] \cap \Bigl[\cap_{\epsilon > 0} (\Omega_{X} + \bbB_{\epsilon})^{\pi_{2}}\Bigr] 
        &= (\Omega_{X} + \bbB_{\epsilon})^{\pi_{1}} \cap (\Omega_{X} + \bbB_{\epsilon})^{\pi_{2}} \\
        &&= (\Omega_{X} + \bbB_{\epsilon})^{\pi_{1} \wedge \pi_{2}}\\
        &&= \cap_{\epsilon > 0}(\Omega_{X} + \bbB_{\epsilon})^{\pi_{1} \wedge \pi_{2}}.
    \end{eqnarray*}
    It follows that $\cap_{j = 1}^{\infty} \tilde{\Omega}_{X}^{\pi_{j}} = \lim_{k \to \infty} \tilde{\Omega}_{X}^{ \wedge_{j = 1}^{k}\pi_{j}}$. If $(u, v) \in U \times U$ is not contained in the closure of $\Omega_{X}$, then for a small enough $\delta > 0$ the $\delta$-ball $(u,v) + \bbB_{\delta}$ does not intersect with the closure of $\Omega$. For a sufficiently large $k$, there will be a pixel induced by $\wedge_{j = 1}^{k}\pi_{j}$ containing $(u,v)$ and which is itself contained in the $\delta$-ball, for otherwise this would imply that the partitions do not separate points. For a small enough $\epsilon > 0$, this pixel will not be included in $(\Omega_{X} + \bbB_{\epsilon})^{\wedge_{j = 1}^{k}\pi_{j}}$. It can be worked out from the zero entries of the operator matrices $\bP_{\pi_{j}}$ for $1 \leq j \leq k$ that this pixel and hence the point is indeed not contained in the closure of ${\Omega}_{X}$. Similarly, if $(u, v)$ is in the closure of $\Omega_{X}$ we can show that no pixel containing it will ever be rejected by a finite number of precision operator matrices $\bP_{j}$. This establishes the claim.
\end{proof}

\subsection{Estimation of the Precision Operator Matrix}

\subsubsection{Correlation Operator Matrix}
\begin{proof}[Proof of Theorem \ref{thm:error_corr}]    
    We decompose the difference $\hbR - \bR$ into approximation and estimation terms as follows
    \begin{equation*}
        \hbR - \bR = \hbR - \bR_{e} + \bR_{e} - \bR
    \end{equation*}
    where $\bR_{e} = \bI + [\epsilon\bI + \dg \bK]^{-1/2}\bK_{0}[\epsilon\bI + \dg \bK]^{-1/2}$. By Lemma \ref{lem:estmn_error} and \ref{lem:approx_error} it follows that
    \begin{equation*}
        \|\hbR-\bR\| \leq 5\| \bR \| \left[\frac{\|\hbK - \bK\|^{2}}{\epsilon^{2}} + \frac{\|\hbK - \bK\|}{\epsilon} \right] + 2\epsilon^{\beta} \cdot \|\Phi_{0}\| \cdot \|\bK\|^{\beta}
    \end{equation*}
    Choosing $\epsilon = \|\hbK-\bK\|^{\frac{1}{\beta+1}}$ gives 
    \begin{equation*}
        \|\hbR-\bR\| \leq 10(\|\bR\|\vee\|\Phi_{0}\|\|\bK\|^{\beta}) \cdot \|\hbK - \bK\|^{\frac{\beta}{\beta+1}}
    \end{equation*}
    Similarly, for the case $\beta > 1$, we can choose $\epsilon = \|\hbK-\bK\|^{\frac{1}{2}}$ and argue likewise to conclude that
    \begin{equation*}
        \|\hbR-\bR\| \leq 10(\|\bR\|\vee\|\Phi_{0}\|\|\bK\|^{2\beta-1}) \cdot \|\hbK - \bK\|^{\frac{1}{2}}.
    \end{equation*}
\end{proof}

\begin{lemma}\label{lem:estmn_error}
    We have
	\begin{equation*}
		\| \hbR - \bR_{e} \| \leq 5\| \bR \| \left[\frac{\|\hbK - \bK\|^{2}}{\epsilon^{2}} + \frac{\|\hbK - \bK\|}{\epsilon} \right]
	\end{equation*}
\end{lemma}
\begin{proof}
    The following equation can be verified with some calculation.
    \begin{eqnarray*}
		&\hbR - \bR_{e}
		&= \left[ [\epsilon\bI + \dg\hbK]^{-1/2} - [\epsilon\bI + \dg\bK]^{-1/2} \right][\hbK_{0} - \bK_{0}][\epsilon\bI + \dg\hbK]^{-1/2} \\
		& &+\quad \left[ [\epsilon\bI + \dg\hbK]^{-1/2} - [\epsilon\bI + \dg\bK]^{-1/2} \right]\bK_{0}\left[[\epsilon\bI + \dg\hbK]^{-1/2} - [\epsilon\bI + \dg\bK]^{-1/2}\right] \\
		&&+\quad \left[ [\epsilon\bI + \dg\hbK]^{-1/2} - [\epsilon\bI + \dg\bK]^{-1/2} \right]\bK_{0}[\epsilon\bI + \dg\bK]^{-1/2} \\
		&&+\quad [\epsilon\bI + \dg\bK]^{-1/2}[\hbK_{0} - \bK_{0}][\epsilon\bI + \dg\hbK]^{-1/2} \\
		&&+\quad [\epsilon\bI + \dg\bK]^{-1/2}\bK_{0}\left[[\epsilon\bI + \dg\hbK]^{-1/2} - [\epsilon\bI + \dg\bK]^{-1/2}\right]
	\end{eqnarray*}
    Using $\bK = [\dg\bK]^{1/2}\bR[\dg\bK]^{1/2}$ we can write this expansion as
	\begin{eqnarray*}
		&&= \bD[\hbK_{0} - \bK_{0}][\epsilon\bI + \dg\hbK]^{-1/2} + \bA\bR_{0}\bA^{\ast} + \bA\bR_{0}[\dg\bK]^{1/2}[\epsilon\bI + \dg\bK]^{-1/2} \\
		&&+\quad [\epsilon\bI + \dg\bK]^{-1/2}[\hbK_{0} - \bK_{0}][\epsilon\bI + \dg\hbK]^{-1/2} + [\epsilon\bI + \dg\bK]^{-1/2}[\dg\bK]^{1/2}\bR_{0}\bA^{\ast}
	\end{eqnarray*}
    where
	\begin{eqnarray*}
        &\bR_{0} &= \bR - \bI\\
		&\bD &= [\epsilon\bI + \dg\hbK]^{-1/2} - [\epsilon\bI + \dg\bK]^{-1/2} \\
		&\bA &= \left[ [\epsilon\bI + \dg\hbK]^{-1/2} - [\epsilon\bI + \dg\bK]^{-1/2} \right][\dg\bK]^{1/2}.
	\end{eqnarray*}
    So,
	\begin{eqnarray*}
		&\| \hbR - \bR_{e} \| 
		&\leq \| \bD \|\cdot \|\hbK_{0} - \bK_{0}\| \cdot \frac{1}{\sqrt{\epsilon}} + \| \bA \| \cdot \|\bR_{0}\| \cdot \| \bA \| + \| \bA\| \cdot \|\bR_{0}\| \cdot 1 \\
		&&\qquad+\quad \frac{1}{\sqrt{\epsilon}} \cdot \| \hbK_{0} - \bK_{0} \| \cdot \frac{1}{\sqrt{\epsilon}} + 1 \cdot \|\bR_{0}\| \cdot \| \bA \|.
	\end{eqnarray*}
    Applying Lemma \ref{lem:approx_terms} (also see Remark \ref{rem:nonpos_estmtr}) to $\hbA = \dg \hbK$ and $\bA = \dg \bK$, we derive 
    \begin{equation*}
        \|\bD\| \leq \|\dg \hbK - \dg \bK\|/\epsilon^{3/2} \qquad \mbox{and} \qquad \|\bA\| \leq \|\dg \hbK - \dg \bK\|/\epsilon.
    \end{equation*}    
    Using the simple observation that
    \begin{eqnarray*}
        &\|\dg\bA\|  &= \max_{i} \|\bA_{ii}\| \leq \|\bA\| \\
        &\|\bA_{0}\| &= \|\bA - \dg\bA\| \leq \|\bA\| + \|\dg\bA\| \leq 2 \|\bA\|
    \end{eqnarray*}
    we can write
	\begin{eqnarray*}
		&\|\hbR - \bR_{e}\| 
		&\leq  \frac{\|\dg\hbK - \dg\bK\|\|\hbK_{0} - \bK_{0}\|}{\epsilon^{2}} 
        +~ \|\bR_{0}\| \frac{\|\dg\hbK - \dg\bK\|^{2}}{\epsilon^{2}} \\
        &&+~ \|\bR_{0}\|\frac{\|\dg\hbK - \dg\bK\|}{\epsilon} + \frac{\|\hbK_{0} - \bK_{0}\|}{\epsilon} + \|\bR_{0}\|\frac{\|\dg\hbK - \dg\bK\|}{\epsilon} \\
        &&\leq \frac{\|\hbK - \bK\|^{2}}{\epsilon^{2}} 
        + \|\bR_{0}\| \frac{\|\hbK - \bK\|^{2}}{\epsilon^{2}} \\
        &&+ \|\bR_{0}\|\frac{\|\hbK - \bK\|}{\epsilon} + \frac{\|\hbK - \bK\|}{\epsilon} + \|\bR_{0}\|\frac{\|\hbK - \bK\|}{\epsilon} \\
        &&\leq \left(2\|\bR_{0}\| + 1\right)\left[\frac{\|\hbK - \bK\|}{\epsilon} + \frac{\|\hbK - \bK\|^{2}}{\epsilon^{2}}\right] \\
        &&\leq 5\|\bR\| \left[\frac{\|\hbK - \bK\|}{\epsilon} + \frac{\|\hbK - \bK\|^{2}}{\epsilon^{2}}\right]
	\end{eqnarray*}
    since $\|\bR_{0}\| = \|\bR-\bI\| \leq \|\bR\|+1$ and $\|\bR\| \geq 1$.
	This completes the proof.
\end{proof}

\begin{lemma}\label{lem:approx_terms}  
	If $\hbA$ and $\bA$ are positive, then
	\begin{eqnarray*}
		&\|[\epsilon\bI + \hbA]^{-1/2} - [\epsilon\bI + \bA]^{-1/2}\| 
		&\leq \| \hbA - \bA \|/\epsilon^{3/2} \\
		&\Big\|\left[[\epsilon\bI + \hbA]^{-1/2} - [\epsilon\bI + \bA]^{-1/2}\right]\bA^{1/2}\Big\| 
		&\leq \| \hbA - \bA \|/\epsilon.
	\end{eqnarray*}
\end{lemma}
\begin{proof}
	Notice that
	\begin{eqnarray*}
		&&[\epsilon\bI + \hbA]^{-1/2} - [\epsilon\bI + \bA]^{-1/2} \\
		&&= [\epsilon\bI + \hbA]^{-1/2}\left[ [\epsilon\bI + \hbA]^{1/2} - [\epsilon\bI + \bA]^{-1/2} \right][\epsilon\bI + \bA]^{1/2} \\
		&&= [\epsilon\bI + \hbA]^{-1/2}\left[ [\epsilon\bI + \hbA]^{1/2} + [\epsilon\bI + \bA]^{1/2} \right]^{-1}\left[ [\epsilon\bI + \hbA] - [\epsilon\bI + \bA] \right][\epsilon\bI + \bA]^{-1/2} \\
		&&= \left[ \epsilon\bI + \hbA + [\epsilon\bI + \bA]^{1/2}[\epsilon\bI + \hbA]^{1/2} \right]^{-1} [\hbA - \bA] [\epsilon\bI + \bA]^{-1/2}
	\end{eqnarray*}
	Since $\hbA + [\epsilon\bI + \bA]^{1/2}[\epsilon\bI + \hbA]^{1/2}$ is positive, we can write
	\begin{eqnarray*}
		&& \| [\epsilon\bI + \hbA]^{-1/2} - [\epsilon\bI + \bA]^{-1/2} \| \\
		&&\leq 
            \Big\|\left[
            \epsilon\bI + \hbA + [\epsilon\bI + \bA]^{1/2}[\epsilon\bI + \hbA]^{1/2}
            \right]^{-1}\Big\| 
        \cdot \|\hbA - \bA\| 
        \cdot \|[\epsilon\bI + \bA]^{-1/2}\|\\
		&&\leq \frac{1}{\epsilon} \cdot \|\hbA - \bA\| \cdot \frac{1}{\epsilon^{1/2}}
	\end{eqnarray*}
	and similarly,
	\begin{eqnarray*}
		&&\Big\| \left[[\epsilon\bI + \hbA]^{-1/2} - [\epsilon\bI + \bA]^{-1/2}\right]\bA^{1/2} \Big\| \\
		&&\leq \Big\|\left[ \epsilon\bI + \hbA + [\epsilon\bI + \bA]^{1/2}[\epsilon\bI + \hbA]^{1/2} \right]^{-1}\Big\| 
        \cdot \|\hbA - \bA\| 
        \cdot \|[\epsilon\bI + \bA]^{-1/2}\bA^{1/2}\| \\
		&&\leq \frac{1}{\epsilon} \cdot \|\hbA - \bA\| \cdot 1.
	\end{eqnarray*}
	This completes the proof.
\end{proof}
\begin{remark}\label{rem:nonpos_estmtr}
    The above result holds even if $\hbA$ is not positive so long as $\|\hbA - \bA\|$ is small enough. In this case, we use the following more complicated estimate 
    \begin{eqnarray*}
    &&\Big\|\left[\epsilon\bI + \hbA + [\epsilon\bI + \bA]^{1/2}[\epsilon\bI + \hbA]^{1/2}\right]^{-1}\Big\| \\
    &&= \Big\|\left[2(\epsilon\bI + \bA) + (\hbA - \bA) + [\epsilon\bI + \bA]^{1/2}\Big[[\epsilon\bI + \hbA]^{1/2} - [\epsilon\bI + \bA]^{1/2}\Big]\right]^{-1}\Big\| \\
    &&\leq \Big[ 2\epsilon - \|\hbA - \bA\| - \Big\|(\epsilon\bI + \bA)^{1/2} \Big[(\epsilon\bI + \hbA)^{1/2} - (\epsilon\bI + \bA)^{1/2}\Big]\Big\| \Big]^{-1} \\
    &&\leq \Big[ 2\epsilon - \|\hbA - \bA\| - 2\|\bA\|^{1/2}\|\hbA - \bA\|^{1/2} \Big]^{-1}, 
    \end{eqnarray*}
    which follows from the observation 
    \begin{eqnarray*}
    &&\Big\|(\epsilon\bI + \bA)^{1/2} \Big[(\epsilon\bI + \hbA)^{1/2} - (\epsilon\bI + \bA)^{1/2}\Big]\Big\|  \\
    &&= \Big\|(\epsilon\bI + \bA)^{1/2} \Big[(\epsilon\bI + \hbA)^{1/2} + (\epsilon\bI + \bA)^{1/2}\Big]^{-1} \Big[(\epsilon\bI + \hbA) - (\epsilon\bI + \bA)\Big]\Big\| \\
    &&\leq (\epsilon + \|\bA\|)^{1/2}\|\hbA - \bA\|  \Big\|\Big[(\epsilon\bI + \hbA)^{1/2} + (\epsilon\bI + \bA)^{1/2}\Big]^{-1}\Big\| \\
    &&\leq 2\|\bA\|^{1/2}\|\hbA - \bA\|  \Big\|\Big[(\epsilon\bI + \hbA)^{1/2} + (\epsilon\bI + \bA)^{1/2}\Big]^{-1}\Big\| \\
    &&\leq 2\|\bA\|^{1/2}\|\hbA - \bA\|  \Big[ \sqrt{\epsilon - \|\hbA - \bA\|} + \sqrt{\epsilon} \Big]^{-1} \\
    &&= 2\|\bA\|^{1/2}\Big[  \sqrt{\epsilon} - \sqrt{\epsilon - \|\hbA - \bA\|}  \Big] \\
    &&\leq 2\|\bA\|^{1/2}\|\hbA - \bA\|^{1/2}.
\end{eqnarray*}
It can be shown that for the choice of $\epsilon$ implied by Theorem \ref{thm:error_corr}, the inequalities in Lemma \ref{lem:approx_terms} continue to hold up to multiplicative constants and consequently, we can proceed in the same way as we would assuming $\hbA$ is positive.
\end{remark}

Now, we shall find an upper bound for the approximation error under a regularity condition.

\begin{lemma}\label{lem:approx_error}
	If $\bR_{0} = [\dg \bK]^{\beta}\Phi_{0}[\dg \bK]^{\beta}$ for some bounded operator matrix $\Phi_{0}$ with the diagonal entries all zero and $\beta > 0$, then
	\begin{equation*}
		\| \bR_{e} - \bR \| \leq 
		\begin{cases}
			2\epsilon^{\beta} \cdot \|\Phi_{0}\| \cdot \|\bK\|^{\beta} & 0 < \beta \leq 1 \\
			2\epsilon^{\phantom{\beta}} \cdot \|\Phi_{0}\| \cdot \|\bK\|^{2\beta-1} & 1 < \beta < \infty
		\end{cases}
	\end{equation*}
\end{lemma}
\begin{proof}
	We decompose the difference as follows:
	\begin{eqnarray*}
		& \bR_{e} - \bR
		& = [\epsilon \bI + \dg \bK]^{-1/2}[\dg \bK]^{1/2}\bR_{0}[\dg \bK]^{1/2}[\epsilon \bI + \dg \bK]^{-1/2} - \bR_{0} \\
		&&= \Big[ [\epsilon \bI + \dg \bK]^{-1/2} - [\dg \bK]^{-1/2} \Big][\dg \bK]^{1/2}\bR_{0}[\dg \bK]^{1/2}[\epsilon \bI + \dg \bK]^{-1/2} \\
		&&\qquad+\quad \bR_{0}[\dg \bK]^{1/2}\Big[ [\epsilon \bI + \dg \bK]^{-1/2} - [\dg \bK]^{-1/2} \Big] \\
		&&= \Big[ [\epsilon \bI + \dg \bK]^{-1/2} - [\dg \bK]^{-1/2} \Big][\dg \bK]^{1/2+\beta}\Phi_{0}[\dg \bK]^{1/2+\beta}[\epsilon \bI + \dg \bK]^{-1/2} \\
		&&\qquad+\quad [\dg \bK]^{\beta}\Phi_{0}[\dg \bK]^{1/2+\beta}\Big[ [\epsilon \bI + \dg \bK]^{-1/2} - [\dg \bK]^{-1/2} \Big] \\
	\end{eqnarray*}
	Using $\| [\dg\bK]^{1/2+\beta}[\epsilon\bI+\dg\bK]^{-1/2}\| \leq \|\dg\bK\|^{\beta} \leq \|\bK\|^{\beta} $, it follows that
	\begin{eqnarray*}
		&\|\bR - \bR_{e}\| &\leq \Big\|\Big[[\epsilon\bI+\dg\bK]^{-1/2} - [\dg \bK]^{-1/2} \Big][\dg \bK]^{1/2+\beta} \Big\| \|\Phi_{0}\| \|\dg\bK\|^{\beta} \\
		& &\qquad+\quad \|\dg\bK \|^{\beta}\|\Phi_{0}\|\Big\|[\dg \bK]^{1/2+\beta} \Big[[\epsilon\bI+\dg\bK]^{-1/2} - [\dg\bK]^{-1/2}\Big]\Big\|
	\end{eqnarray*}
	The conclusion is now an obvious consequence of Lemma \ref{lem:diff_sqrts_beta_ineq}.
\end{proof}

\begin{lemma} \label{lem:diff_sqrts_beta_ineq}
	We have 
	\begin{equation*}
		\Big\| \Big[ [\epsilon \bI + \dg \bK]^{-1/2} - [\dg \bK]^{-1/2} \Big][\dg \bK]^{1/2+\beta} \Big\| \leq \begin{cases}
			\epsilon^{\beta} & 0 < \beta \leq 1 \\
			\epsilon \cdot \| \dg \bK \|^{\beta-1} & 1 < \beta < \infty
		\end{cases}
	\end{equation*}
\end{lemma}
\begin{proof}
	By the spectral mapping theorem,
	\begin{eqnarray*}
		\Big\| \Big[ [\epsilon \bI + \dg \bK]^{-1/2} - [\dg \bK]^{-1/2} \Big][\dg \bK]^{1/2+\beta} \Big\| \leq \sup_{0 \leq \lambda \leq  \| \dg \bK \|} \left\lbrace \left| \frac{1}{\sqrt{\epsilon + \lambda}} - \frac{1}{\sqrt{\lambda}} \right| \cdot \lambda^{1/2+\beta} \right\rbrace
	\end{eqnarray*}
	It can be shown using some elementary calculations that
	\begin{eqnarray*}
		\left| \frac{1}{\sqrt{\epsilon + \lambda}} - \frac{1}{\sqrt{\lambda}} \right| \cdot \lambda^{1/2+\beta}
		= \frac{\epsilon\lambda^{\beta}}{\sqrt{\epsilon + \lambda}(\sqrt{\lambda} + \sqrt{\epsilon+\lambda})} 
		\leq \begin{cases}
			\epsilon\left[\frac{\lambda^{\beta}}{\epsilon + \lambda}\right] & 0 < \beta < 1/2 \\
			\epsilon\left[\frac{\lambda^{2\beta-1}}{\epsilon + \lambda}\right]^{1/2} & 1/2 \leq \beta < 1 \\
			\epsilon \lambda^{\beta-1} & 1 \leq \beta < \infty
		\end{cases}
	\end{eqnarray*}
	The conclusion follows from Lemma \ref{lem:lambda_epsilon_ineq}.
\end{proof}
\begin{lemma} \label{lem:lambda_epsilon_ineq}
	For $0 < x < 1$ and $\lambda \geq 0$, we have
	\begin{equation*}
		\frac{\lambda^{x}}{\epsilon + \lambda} \leq \frac{\epsilon^{x-1}}{2}
	\end{equation*}
\end{lemma}
\begin{proof}
	Consider the reciprocal expression. It follows from elementary differential calculus that the minimum of the reciprocal occurs at $\lambda_{\ast} = x\epsilon/(1-x)$. Therefore,
	\begin{equation*}
		\frac{\epsilon}{\lambda^{x}} + \lambda^{1-x} \geq \frac{\epsilon}{\lambda_{\ast}^{x}} + \lambda_{\ast}^{1-x} = \frac{\epsilon^{1-x}}{x^{x}(1-x)^{1-x}} \geq \frac{\epsilon^{1-x}}{\max_{0 < x < 1} [x^{x}(1-x)^{1-x}]} = 2\epsilon^{1-x}
	\end{equation*}
\end{proof}

\subsubsection{Concentration Inequalities}

\begin{proof}[Proof of Theorem \ref{thm:large_dev}]
    Apply Theorem 9 from \textcite{koltchinskii2017} and replace $t$ with $nt^{2}/\|\bK\|^{2}$, simplify and restate the conditions accordingly.
\end{proof}

We now prove a concentration inequality for the correlation operator.
\begin{proof}[Proof of Theorem \ref{thm:conc_ineq}]
\begin{enumerate}
    \item This is a straightforward consequence of Theorem \ref{thm:large_dev} and \ref{thm:error_corr}.
    \begin{equation*}
        \bbP[\|\hbR-\bR\| > \rho] \leq \bbP[\|\hbK-\bK\| > (\rho/M_{R})^{1+1/\beta \wedge 1}] \leq \exp \left[-c_{R}n\rho^{2+2/\beta \wedge 1}\right].
    \end{equation*}
    \item Under Assumption \ref{asm:eigenvalues}, $r = 1 + \inf_{k} \lambda_{k}(\bR_{0}) > 0$. Thus, $\bR \geq r\bI$. By the spectral mapping theorem, $\|\bP\| \leq 1/r$. For $\|f\| = 1$, we have
    \begin{equation*}
        \langle f, [\hbR - (r - \rho)\bI]f\rangle = \rho + \langle f,[\hbR - \bR]f \rangle + \langle f, [\bR - r\bI]f \rangle 
    \end{equation*}
    and so,
    \begin{eqnarray*}
        & \inf_{f} \langle f, [\hbR - (r - \rho)\bI]f\rangle 
        & \geq \rho + \inf_{f} \langle f,[\hbR - \bR]f \rangle + \inf_{f} \langle f, [\bR - r\bI]f \rangle \\
        && \geq \rho - \|\hbR - \bR\|.
    \end{eqnarray*}
    The result follows by the spectral mapping theorem from the following observation
    \begin{equation*}
        \bbP[\|\hbP\| > (r - \rho)^{-1}] \leq \bbP[\|\hbR-\bR\| > \rho].
    \end{equation*}
    \item Using a union bound, we have 
    \begin{eqnarray*}
        &\bbP[\|\hbP - \bP\| > \rho] 
        &\leq \bbP[ \|\hbP\| > (r - \rho)^{-1}] + \bbP[\|\hbR - \bR\| > \rho(r - \rho)/\|\bP\|] \\
        &&\leq \exp \left\lbrace-c_{R}n\rho^{2+2/(\beta \wedge 1)}\right\rbrace + \exp \left\lbrace-c_{R}n\left[\rho(r - \rho)/\|\bP\|\right]^{2+2/(\beta \wedge 1)}\right\rbrace
    \end{eqnarray*}
    Now we need only notice that since $0 < r \leq 1$ and $\|\bP\| = 1/r$, we must have $\rho > \rho(r - \rho)/\|\bP\|$. If we require that $\rho \leq r/2$, then $\rho(r - \rho)/\|\bP\| \geq \rho r^{2}/2$ and the conclusion follows.
\end{enumerate}    
\end{proof}

\subsection{Model Selection Consistency}

\begin{proof}[Proof of Theorem \ref{thm:consistency}]
    Notice that $\hat{\Omega}^{\pi} \neq \Omega^{\pi}$ if and only if for some $1 \leq i,j \leq p$ we have
    \begin{enumerate}
        \item $\|\bP_{ij}\|\neq 0$ and $\|\hbP_{ij}\| < \rho$, or
        \item $\|\bP_{ij}\|  =  0$ and $\|\hbP_{ij}\| \geq \rho$.
    \end{enumerate}
    If we require that $\rho < \frac{1}{2}\min_{i,j} \|\bP_{ij}\|$, then this implies that for some $(i,j)$ we must have
    \begin{equation*}
        \|\hbP_{ij} - \bP_{ij}\| > \rho.
    \end{equation*}
    Therefore,
    \begin{eqnarray*}
        & \bbP[\hat{\Omega}^{\pi} \neq \Omega^{\pi}] 
        & = \bbP \cup_{i,j} [\|\hbP_{ij} - \bP_{ij}\| > \rho] \\
        &&\leq \sum_{i,j=1}^{p} \bbP [\|\hbP_{ij} - \bP_{ij}\| > \rho] \\
        &&\leq ~p^{2}\cdot \bbP [\|\hbP - \bP\| > \rho]. 
    \end{eqnarray*}
    Now we apply Theorem \ref{thm:conc_ineq} (3).
\end{proof}

\begin{proof}[Proof of Theorem \ref{thm:consistency2}]
    The proof is a straightforward application of the Borel-Cantelli lemma. Since,
    \begin{equation*}
        \sum_{j = 1}^{\infty}\bbP[\hat{\Omega}_{j} \neq \tilde{\Omega}_{X}^{\pi_{j}}] \leq \sum_{j = 1}^{\infty} \alpha_{j} < \infty
    \end{equation*}
    it follows that $\bbP[\hat{\Omega}_{j} \neq \tilde{\Omega}_{X}^{\pi_{j}} ~i.o.] = 0$. With probability 1, there exists some $j_{0} \geq 1$ such that for all $j \geq j_{0}$ we have $\hat{\Omega}_{j} = \tilde{\Omega}_{X}^{\pi_{j}}$. The conclusion follows from observing that $\cap_{j \geq j_{0}} \tilde{\Omega}_{X}^{\pi_{j}} = \Omega_{X}$.
\end{proof}

\subsection{Discrete Observations with Noise}
\begin{proof}[Proof of Theorem \ref{thm:reg-obs-estmtr}]
Notice that 
\begin{equation*}
    \bbE [ \hat{K}_{\mathrm{regular}}(s, t) ] = \frac{\sum_{i,j=0,1} (1-\delta_{p+i,q+j})K(T_{p+i}, T_{q+j})}{\sum_{i,j=0,1} (1 - \delta_{p+i,q+j})} \qquad \mbox{ for } (s, t) \in I_{p} \times I_{q}.
\end{equation*}
Furthermore, 
    \begin{align*}
        \|\hat{\bK}_{\mathrm{regular}} - \bK\|
        &\leq |\hat{K}_{\mathrm{regular}}(s, t) - K(s, t)| \\
        &\leq |\hat{K}_{\mathrm{regular}}(s, t) - \bbE[\hat{K}_{\mathrm{regular}}(s, t)]| + |\bbE[\hat{K}_{\mathrm{regular}}(s, t)] - K(s, t)|\\
        &\leq \frac{1}{\sum_{i,j=0,1} (1 - \delta_{p+i,q+j})} \sum_{i,j=0,1} (1-\delta_{p+i,q+j})|\hat{F}_{p+i,q+j} - K(T_{p+i}, T_{q+j})| \\
        &+ \frac{1}{\sum_{i,j=0,1} (1 - \delta_{p+i,q+j})}\sum_{i,j=0,1} (1-\delta_{p+i,q+j})|K(T_{p+i}, T_{q+j}) - K(s, t)| \\
        &\leq \max_{\substack{i,j = 0, 1 \\
        p+i \neq q+j}} |\hat{F}_{p+i,q+j} - K(T_{p+i}, T_{q+j})| + \max_{i,j = 0, 1} |K(T_{p+i}, T_{q+j}) - K(s, t)|
    \end{align*}
Therefore, $\|\hat{\bK}_{\mathrm{regular}} - \bK\| \leq \max_{i \neq j} |\hat{F}_{i,j} - K(T_{i}, T_{j})| + \frac{2\sD}{M}$ and using Bernstein's inequality, we can write for $0 \leq t \leq \kappa^{2}$,
\begin{align*}
    \bbP\left[ \|\hat{\bK}_{\mathrm{regular}} - \bK\| \geq t + \frac{2\sD}{M}\right] 
    &\leq \bbP\left[ \max_{i \neq j} |\hat{F}_{i,j} - K(T_{i}, T_{j})| \geq t \right] \\
    &\leq \sum_{i \neq j} \bbP\left[ |\hat{F}_{i,j} - K(T_{i}, T_{j})| \geq t \right] \\
    &\leq [(M+1)^{2}-(M+1)] \cdot 2\exp\left[-\frac{cNt^{2}}{\kappa^{4}}\right] \\
    &\leq 4M^{2} \exp\left[-\frac{cNt^{2}}{\kappa^{4}}\right]
\end{align*}
This completes the proof.
\end{proof}

\begin{proof}[Proof of Theorem \ref{thm:sparse-obs-estmtr}]
Define $K_{\mathrm{sparse}}: I \times I \to \bbR$ as $K_{\mathrm{sparse}}(s, t) = M^{2}K_{pq}$ for $(s, t) \in I_{p} \times I_{q}$. Notice that 
\begin{align*}
    \|\hat{\bK}_{\mathrm{sparse}} - \bK\| &\leq \sup_{u,v \in I} |\hat{K}(u, v) - K(u, v)| \\
    &\leq \sup_{u,v \in I} |\hat{K}(u, v) - K_{\mathrm{sparse}}(u, v)| + \sup_{u,v \in I} |K_{\mathrm{sparse}}(u, v) - K(u, v)| \\
    &\leq M^{2} \cdot \max_{p,q} |\hat{K}_{pq} - K_{pq}| + \max_{p,q} \Big[ \sup \{|M^{2}K_{pq} - K(u, v)|: u \in I_{p}, v \in I_{q}\} \Big] \\
    &\leq M^{2} \cdot \max_{p,q} |\hat{K}_{pq} - K_{pq}| + \frac{1}{M} \sD
\end{align*}
since for $(u, v) \in I_{p} \times I_{q}$,
\begin{align*}
    \sup_{u, v} |K_{pq} - K(u, v)| 
    &= M^{2}\cdot\sup_{u, v} \iint_{I_{p} \times I_{q}} |K(s, t)- K(u, v)| ~ds~dt\\
    &\leq M^{2}\cdot\sup_{u \neq v} \left|\frac{\partial}{\partial u}K(u, v)\right| \cdot \sup_{u, v} \int_{0}^{1/M}\int_{0}^{1/M} [s + t] ~ds~dt\\
    &= \frac{1}{M}\sD.
\end{align*}
Using union bound and Bernstein's inequality (Corollary 2.8.3 of \textcite{vershynin2018}) together with the observation that $\bbE[\hat{K}_{pq}] = K_{pq}$, we have
\begin{align*}
    \bbP\left[\|\hat{\bK}_{\mathrm{sparse}} - \bK\| \geq t + \tfrac{1}{M}\sD\right] 
    &\leq \bbP\left[ \max_{p,q} |\hat{K}_{pq} - K_{pq}| \geq  t/M^{2} \right] \\
    &\leq \sum_{p,q = 1}^{M} \bbP\left[ |\hat{K}_{pq} - K_{pq}| \geq  t/M^{2} \right]\\
    &\leq M^{2} \cdot 2  \exp \left[- \frac{cNt^{2}}{M^{4}\kappa^{4}}\right]
\end{align*}
where the last statement follows from Bernstein's inequality using the estimates:
\begin{align*}
    \left\| \frac{1}{n_{k}(n_{k}-1)}\sum_{i, j=1}^{n_{k}} (1-\delta_{ij}) Y_{ki}Y_{kj} \mathbf{1}_{\{T_{ki} \in I_{p}\}}\mathbf{1}_{\{T_{kj} \in I_{q}\}}  \right\|_{\psi_{1}} 
    &\leq \left\| Y_{ki}Y_{kj} \mathbf{1}_{\{T_{ki} \in I_{p}\}}\mathbf{1}_{\{T_{kj} \in I_{q}\}}  \right\|_{\psi_{1}} \leq \left\| Y_{ki} \right\|_{\psi_{2}}^{2}   
\end{align*}
and $\left\| Y_{ki} \right\|_{\psi_{2}} \leq  \left\| X(T) \right\|_{\psi_{2}} + \left\| \xi_{ki} \right\|_{\psi_{2}} \leq  \kappa$. This completes the proof.
\end{proof}

The proofs of Corollary \ref{corollary-reg-1}, Corollary \ref{corollary-reg-2} and Corollary \ref{corollary-sparse-1} follow from Theorem \ref{thm:reg-obs-estmtr} and Theorem \ref{thm:sparse-obs-estmtr} in a straightforward way as in the complete observations case and are therefore omitted.

\printbibliography
\end{document}